\documentclass[a4paper,11pt]{amsart}

\usepackage[usenames]{xcolor}

\usepackage[foot]{amsaddr}

\usepackage{amsmath}

\usepackage{mathrsfs}

\usepackage{hyperref}

\usepackage{tikz}
\usepackage{pgfplots}

\newcommand{\SlopeTriangle}[6]
{
    \pgfplotsextra
    {
        \pgfkeysgetvalue{/pgfplots/xmin}{\xmin}
        \pgfkeysgetvalue{/pgfplots/xmax}{\xmax}
        \pgfkeysgetvalue{/pgfplots/ymin}{\ymin}
        \pgfkeysgetvalue{/pgfplots/ymax}{\ymax}

        \pgfmathsetmacro{\xArel}{#1}
        \pgfmathsetmacro{\yArel}{#3}
        \pgfmathsetmacro{\xBrel}{#1-#2}
        \pgfmathsetmacro{\yBrel}{\yArel}
        \pgfmathsetmacro{\xCrel}{\xArel}
        \pgfmathsetmacro{\lnxB}{\xmin*(1-(#1-#2))+\xmax*(#1-#2)} % in [xmin,xmax].
        \pgfmathsetmacro{\lnxA}{\xmin*(1-#1)+\xmax*#1} % in [xmin,xmax].
        \pgfmathsetmacro{\lnyA}{\ymin*(1-#3)+\ymax*#3} % in [ymin,ymax].
        \pgfmathsetmacro{\lnyC}{\lnyA+#4*(\lnxA-\lnxB)}
        \pgfmathsetmacro{\yCrel}{\lnyC-\ymin)/(\ymax-\ymin)} % THE IMPROVED EXPRESSION WITHOUT 'DIMENSION TOO LARGE' ERROR.

        \coordinate (A) at (rel axis cs:\xArel,\yArel);
        \coordinate (B) at (rel axis cs:\xBrel,\yBrel);
        \coordinate (C) at (rel axis cs:\xCrel,\yCrel);

        \draw[#6]   (A)-- node[anchor=north] {#5}
                    (B)--
                    (C)--
                    cycle;
    }
}

\usepackage{todonotes}

\newcommand{\GD}{{\Gamma_{\rm D}}}
\newcommand{\GA}{{\Gamma_{\rm A}}}

\newcommand{\grad}{\boldsymbol \nabla}
\renewcommand{\div}{\grad {\cdot}}

\newcommand{\ddiv}{\operatorname{div}}

\newcommand{\xx}{\boldsymbol x}
\newcommand{\yy}{\boldsymbol y}
\newcommand{\nn}{\boldsymbol n}
\newcommand{\dd}{\boldsymbol d}
\newcommand{\HH}{\boldsymbol H}
\newcommand{\LL}{\boldsymbol L}

\newcommand{\vv}{\boldsymbol v}

\newcommand{\sig}{\boldsymbol \sigma}
\newcommand{\ttau}{\boldsymbol \tau}

\newcommand{\tpi}{\widetilde \pi}

\newcommand{\Cba}{{\sigma_{\rm ba}}}
\newcommand{\tCba}{{\widetilde \sigma}_{\rm ba}}

\newcommand{\Ci}{C_{\rm i}(\kappa)}

\newcommand{\Cstab}{C_{\rm stab}(\widehat \Omega,\widehat \Gamma_{\rm D})}
\newcommand{\Cstabx}{C_{\rm stab}(\widehat \Omega,\widehat \Gamma_{\rm D},\xx_0)}

\newcommand{\Ctra}{C_{{\rm tr},\aaa}}

\newcommand{\Cqi}{C_{\rm qi}}
\newcommand{\Cqia}{C_{{\rm qi},\aaa}}

\newcommand{\RR}{\mathcal R}

\newcommand{\TT}{\mathcal T}
\newcommand{\VV}{\mathcal V}
\newcommand{\FF}{\mathcal F}

\newcommand{\PP}{\mathcal P}
\newcommand{\PPP}{\boldsymbol \PP}

\newcommand{\RT}{\boldsymbol{RT}}

\renewcommand{\Re}{\operatorname{Re}}
\renewcommand{\Im}{\operatorname{Im}}

\newcommand{\aaa}{{\boldsymbol a}}

\newcommand{\oa}{{\omega_\aaa}}

\newcommand{\oK}{{\omega_K}}

\newcommand{\supp}{\operatorname{supp}}

\newcommand{\ssig}{\boldsymbol \sig}

\newcommand{\Cupper}{{C_{\rm up}}}

\newcommand{\Clow}{C_{\rm low}}
\newcommand{\ClowK}{C_{{\rm low},\oK}}
\newcommand{\Clowa}{C_{{\rm lb},\aaa}}

\newcommand{\Clowerhat}{C_{{\rm lb},{\rm hat},\aaa}}
\newcommand{\Clowerqi }{C_{{\rm lb},{\rm qi },\aaa}}

\newcommand{\II}{\mathcal I}

\newcommand{\lift}{\mathscr L}

\newcommand{\anglepw}{\nu}

\newcommand\CcPFa{C_{\mathrm{cont}, \mathrm{PF},\aaa}}
\newcommand\Csta{C_{\mathrm{st}, \aaa}}

\newcommand\eq{:=}

\newcommand\ie{i.e.}
\newcommand\cf{cf.}
\newcommand\eg{e.g.}

\usepackage{a4wide}
\usepackage{amssymb}

\newtheorem{theo}{Theorem}
\newtheorem{coro}[theo]{Corollary}
\newtheorem{lemm}[theo]{Lemma}
\newtheorem{prop}[theo]{Proposition}
\newtheorem{rema}[theo]{Remark}
\newtheorem{defi}[theo]{Definition}
\newtheorem{assu}[theo]{Assumption}

\numberwithin{equation}{section}
\numberwithin{theo}{section}
\newcommand{\osc}{\operatorname{osc}}

\newcommand{\CPa}{C_{\mathrm{PF},\aaa}}

\newcommand{\norm}[1]{{\left\vert\kern-0.25ex\left\vert\kern-0.25ex\left\vert #1 \right\vert\kern-0.25ex\right\vert\kern-0.25ex\right\vert}}

\title[A posteriori error estimates for the Helmholtz equation]%
{On the derivation of guaranteed and $\MakeLowercase{p}$-robust a posteriori error
estimates for the Helmholtz equation$^\star$}

\author{T. Chaumont-Frelet$^{1,2}$}
\author{A. Ern$^{3,4}$}
\author{M. Vohral\'ik$^{4,3}$}
\address{\vspace{-.5cm}}
\address{\noindent \tiny \textup{$^\star$This project has received funding from the European Research Council (ERC) under the European Union’s Horizon 2020}}
\address{\noindent \tiny \textup{\hspace{.2cm}research and innovation program (grant agreement No 647134).}}
\address{\noindent \tiny \textup{$^1$Inria, 2004 Route des Lucioles, 06902 Valbonne, France}}
\address{\noindent \tiny \textup{$^2$Laboratoire J.A. Dieudonn\'e, Parc Valrose, 28 Avenue Valrose, 06108 Nice Cedex 02, 06000 Nice, France}}
\address{\noindent \tiny \textup{$^3$CERMICS, Ecole des Ponts, 77455 Marne-la-Vall\'ee, France}}
\address{\noindent \tiny \textup{$^4$Inria, 2 rue Simone Iff, 75589 Paris, France}}

\begin{document}

\maketitle

\begin{abstract}
We propose a novel {\em a posteriori} error estimator for conforming finite element
discretizations of two- and three-dimensional Helmholtz problems.
The estimator is based on an equilibrated flux that is computed by solving patchwise
mixed finite element problems. We show that the estimator is reliable up to a prefactor that
tends to one with mesh refinement or with polynomial degree increase. We also derive a fully
computable upper bound on the prefactor for several common settings of domains and boundary
conditions. This leads to a guaranteed estimate without any assumption on the mesh size or
the polynomial degree, though the obtained guaranteed bound may lead to large error
overestimation.  We next demonstrate that the estimator is locally efficient, robust in all
regimes with respect to the polynomial degree, and asymptotically robust with respect to the
wavenumber. Finally we present numerical experiments that illustrate our analysis and indicate
that our theoretical results are sharp.

\vspace{.5cm}
\end{abstract}

\section{Introduction}

In this work, we are concerned with the following Helmholtz equation: find
$u: \Omega \to \mathbb{C}$ such that
\begin{equation}
\label{eq_helmholtz_strong}
\left \{
\begin{array}{rcll}
-k^2 u - \Delta u &=& f & \text{ in } \Omega,
\\
u &=& 0 & \text{ on } \GD,
\\
\grad u {\cdot} \nn - iku &=& g & \text{ on } \GA.
\end{array}
\right.
\end{equation}
Here $\Omega \subset \mathbb R^d$, $d=2$ or $3$, is a Lipschitz domain with polygonal or polyhedral
boundary $\partial \Omega$ that is partitioned into two disjoint relatively open sets $\GD$ and $\GA$,
$\nn$ denotes the unit vector normal to $\partial \Omega$ pointing outward $\Omega$,
$f: \Omega \to \mathbb C$ and $g: \GA \to \mathbb C$ are prescribed data, and
the real number $k \in (0,+\infty)$ is the wavenumber.

Problem~\eqref{eq_helmholtz_strong} can accurately model the propagation of time-harmonic
acoustic or polarized electromagnetic waves. It is also insightful
for more elaborate wave propagation models employed, for instance, in elasticity
and electromagnetics, with industrial applications in acoustics, radar, and
medical or subsurface imaging, see, \eg~\cite{%
chavent_papanicolaou_sacks_symes_2012a,%
colton_kress_2012a,%
davidson_2010a,%
marburg_2002a,%
taus_zepedanunez_hewett_demanet_2017a}
and the references therein.

The above setting can represent the {\em scattering problem} by a sound-soft obstacle $D$,
see Figure~\ref{figure_domains}, (a) and (b). Then $\GD = \partial D$ represents the
boundary of the obstacle, $\GA = \partial \Omega_0$ where $\Omega_0$ is some ``computational box''
that includes and surrounds $D$, and $\Omega = \Omega_0 \setminus D$. An important scenario is
the case of a {\em non-trapping obstacle} $D$ in which both $D$ and $\Omega_0$ are star-shaped
with respect to a common center, see Figure~\ref{figure_domains},
(a); for a {\em trapping obstacle}, see Figure~\ref{figure_domains}, (b),
such a condition is not satisfied. In {\em cavity} problems,
see Figure~\ref{figure_domains}, (c), $\GD$ represents
the basis of a cavity and $\GA$ is typically planar.
The obstacle-free case with $\GD = \emptyset$ represents the propagation of a wave
in {\em free space}. Finally, the so-called \emph{interior problem} where $\GA = \emptyset$
is well-posed provided that $k^2$ is not an eigenvalue of the Laplace operator in $\Omega$.

\begin{figure}
\centering
\begin{minipage}{.30\linewidth}
\centering
\begin{tikzpicture}[scale=2]
\draw[dashed] (-1,-1) rectangle (1,1);
\draw (0,1) node[anchor=south] {$\GA$};

\draw (0,0) -- (.5,.5) -- (0,-.5) -- (-.5,.5) -- cycle;
\draw (-.01,-.03) node[anchor=north] {$D$};
\draw (.30,-.20) node[anchor=north] {$\GD$};
\draw (-.85,-.7) node[anchor=north] {$\Omega$};
\end{tikzpicture}

(a) non-trapping obstacle
\end{minipage}
\begin{minipage}{.30\linewidth}
\centering
\begin{tikzpicture}[scale=2]
\draw[dashed] (-1,-1) rectangle (1,1);
\draw (0,1) node[anchor=south] {$\GA$};

\draw (-.5,.5) -- (-.5,-.5) -- (0,-.5) -- (-.25,0) -- (0,.5) -- cycle;
\draw (.3,0) -- (.05,.25) -- (.5,0) -- (.05,-.25) -- cycle;
\draw (.35,.12) node[anchor=north] {\scriptsize{$D$}};
\draw (-.30,-.15) node[anchor=north] {$D$};
\draw (.2,-.22) node[anchor=north] {$\GD$};
\draw (-.85,-.7) node[anchor=north] {$\Omega$};
\end{tikzpicture}

(b) trapping obstacle
\end{minipage}
\begin{minipage}{.30\linewidth}
\centering
\begin{tikzpicture}[scale=2]
\draw[dashed] (-1,1) -- (1,1);
\draw (0,1) node[anchor=south] {$\GA$};

\draw (-1,1) -- (-0.8,-1) -- (-.7,-.5) -- (.6,-.7) -- (.9,-.7) -- (1,1);
\draw (.9,-.7) node[anchor=south east] {$\GD$};
\draw (-.7,-.20) node[anchor=north] {$\Omega$};
\end{tikzpicture}

(c) cavity
\end{minipage}
\caption{Examples of domains $\Omega$ and boundaries $\GD$ and $\GA$}
\label{figure_domains}
\end{figure}
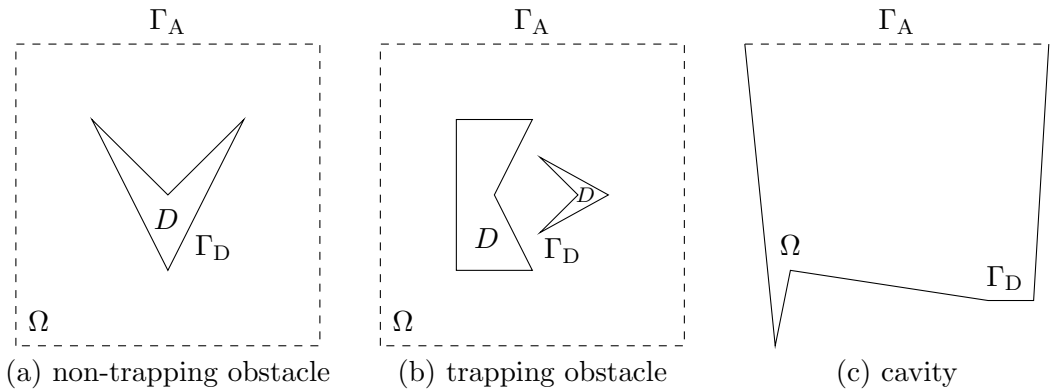

Numerical discretizations, often based on the finite element method, are nowadays vastly used
to find approximate solutions to~\eqref{eq_helmholtz_strong}. Here, a finite-dimensional space
$V_h$ is constructed which typically consists of piecewise polynomials of degree $p \geq 1$ defined
over a computational mesh $\TT_h$ of $\Omega$ with a maximal cell size $h$.
In this case, $(\frac{kh}{2\pi p})^{-1}$ is a measure of the number of degrees of freedom per wavelength.
Another important quantity, following~\cite{dorfler_sauter_2013a,melenk_sauter_2011a},
is the approximation factor\footnote{Our definition is slightly different
from~\cite{melenk_sauter_2011a} as we include the wavenumber $k$ in $\Cba$. This
way, the quantity $\Cba$ is adimensional, and invariant under rescaling.}
\[
\Cba \eq k
\sup_{\phi \in L^2(\Omega) \setminus \{0\}}
\inf_{v_h \in V_h} \frac{|u_\phi^\star - v_h|_{1,\Omega}}{\|\phi\|_{0,\Omega}},
\]
where $u_\phi^\star$ solves an adjoint problem to~\eqref{eq_helmholtz_strong}
with data $f = \phi$ and $g=0$ (see Section~\ref{sec:adjoint_solutions} for more details).
The real number $\Cba$ describes the ability of the discrete space $V_h$ to approximate (adjoint)
solutions to~\eqref{eq_helmholtz_strong}. This quantity combines a measure
of the approximation capacity of $V_h$ in the $H^1$-seminorm with the stability of the
associated adjoint problem. Incidentally we notice that $\Cba$ is bounded uniformly
with respect to $h$ and $p$, as can be seen by taking $v_h = 0$ in the above definition.
For the case of scattering by a smooth non-trapping obstacle, the following upper bound
is available:
\begin{equation}
\label{eq_theta_nontrapping}
\Cba \leq C(\widehat \Omega,\widehat \Gamma_{\rm D}) C(\kappa)
\left (
\frac{kh}{p} + k h_\Omega \left (\frac{kh}{p} \right )^p
\right ),
\end{equation}
where the constants $C(\widehat \Omega,\widehat \Gamma_{\rm D})$ and
$C(\kappa)$, respectively, only depend on the shape of the computational domain $\Omega$
(but not on its diameter $h_\Omega$) and on the mesh shape-regularity parameter $\kappa$,
see \cite[Proposition 5.3]{melenk_sauter_2011a}. We speak of {\em unresolved regime} when
$\frac{kh}{2 \pi p} > 1$, and of {\em resolved regime} when $\frac{kh}{2 \pi p} \leq 1$. We also speak of {\em asymptotic regime} when $\Cba \leq 1$ and {\em preasymptotic regime} otherwise.
We point out that the condition $\Cba \leq 1$ typically translates into stronger
requirements on $h$ and $p$ than the resolved regime.

The sesquilinear form $b({\cdot},{\cdot})$ associated with the boundary value
problem~\eqref{eq_helmholtz_strong} is typically not coercive, as it contains a
negative $L^2(\Omega)$ contribution, and thus does not lead to an energy norm.
The finite element solution (when it exists), is almost meaningless in the unresolved
regime, where even the best approximation of the solution $u$ of~\eqref{eq_helmholtz_strong}
in the discrete space $V_h$ is inaccurate. In the resolved regime, this best approximation
already starts to accurately represent $u$, but the finite element solution $u_h \in V_h$
is not necessarily quasi-optimal and is typically still inaccurate. This phenomena is known in
the literature as the ``pollution effect''. It is only in the asymptotic regime that the
quasi-optimality of the finite element solution is ensured, \cf~\cite{melenk_sauter_2011a}
and the references therein. The typical dependence~\eqref{eq_theta_nontrapping} encourages
the use of finite elements with high polynomial degree $p$ to solve problems with high wavenumber
$k$, and this is a usual practice, \cf~\cite{%
ainsworth_2004a,%
beriot_prinn_gabard_2016a,%
chaumontfrelet_nicaise_2019a,%
congreve_gedicke_perugia_2019a,%
melenk_sauter_2011a,%
taus_zepedanunez_hewett_demanet_2017a}
and the references therein.

The design of suitable {\em a posteriori} estimators is of paramount importance
both for control of the error between $u$ and $u_h$ and for the efficiency of
algorithms adaptively refining $h$ and/or $p$
\cite{demkowicz_2006a,Verf_13}.
Pioneering works on a posteriori error estimation for the Helmholtz
equation are reported in~\cite{babuska_ihlenburg_strouboulis_gangaraj_1997a,%
babuska_ihlenburg_strouboulis_gangaraj_1997b}.
The authors focus on first-order discretizations ($p=1$) of one-dimensional
problems and prove that in the asymptotic regime, the residual~\cite{Verf_13} and the
Zienkiewicz--Zhu~\cite{Zien_Zhu_simple_a_post_87}
estimators are reliable (yielding an error upper bound) and efficient (yielding an error lower bound).
Numerous refinements of these results, including goal-oriented estimation and $hp$ adaptivity,
can be found in~\cite{%
darrigrand_pardo_muga_2015a,%
irimie_bouillard_2001a,%
peraire_patera_1999a,%
sarrate_peraire_patera_1999a,%
stewart_hughes_1996a,%
stewart_hughes_1997a},
see also the references therein.

More recently, {\em residual estimators} for high-order finite
element as well as discontinuous Galerkin discretizations of two- and
three-dimensional problems have
been studied in~\cite{dorfler_sauter_2013a,sauter_zech_2015a},
where an upper bound is obtained even in the preasymptotic regime,
taking the form (when the data $f$ and $g$ are piecewise polynomial)
\begin{subequations}
\begin{equation}
\label{eq_upper_bound_sauter_dorfler}
(k^2 \|u-u_h\|_{0,\Omega}^2 + |u-u_h|_{1,\Omega}^2)^{\frac12}
\leq
\Cupper \eta,
\quad
\Cupper = C(\kappa) + \theta_{0}(\Cba),
\end{equation}
where $\eta$ is the a posteriori error estimator which is fully computable from $u_h$ and the
data. The constant $C(\kappa)$ is not computable, but only depends on the shape-regularity
parameter $\kappa$
of the mesh $\TT_h$, and $\lim_{t \to 0} \theta_{0}(t) = 0$. One sees that in the asymptotic regime
where $\Cba \leq 1$, the prefactor $\Cupper$ in~\eqref{eq_upper_bound_sauter_dorfler} simplifies
into an (unknown) constant that only depends on the mesh shape-regularity parameter. Moreover,
the authors in~\cite{dorfler_sauter_2013a,sauter_zech_2015a} prove that as long as there are
sufficiently many degrees of freedom per wavelength (essentially in the resolved regime where
$\frac{kh}{2\pi p} \leq 1$),
\begin{equation}
\label{eq_lower_bound_sauter_dorfler}
\eta_K \leq \ClowK (k^2 \|u-u_h\|_{0,\oK}^2 + |u-u_h|_{1,\oK}^2)^{\frac12}, \quad
\ClowK = C(\kappa_{\TT_K},p) \left ( 1 + \frac{kh_{\oK}}{p} \right ),
\end{equation}
\end{subequations}
where the factor $(1 + \frac{kh_{\omega_K}}{p})$ naturally appears for $\GA = \emptyset$,
$\eta = (\sum_{K \in \TT_h} \eta_K^2)^{\frac12}$, $\TT_K$ is the patch of elements
around the mesh element $K$ with shape-regularity $\kappa_{\TT_K}$ and the corresponding
subdomain $\oK$, and the constant $C(\kappa_{\TT_K},p)$ deteriorates as the polynomial
degree $p$ increases. We also refer the reader to \cite{hoppe_sharma_2013a}, where the
convergence of an adaptive discontinuous Galerkin method based on a residual estimator
and D\"orfler marking is studied.

Another approach to a posteriori error estimation is based on a construction of an
{\em equilibrated flux}, \cf~\cite{Dest_Met_expl_err_CFE_99,Lad_these_75,prager_synge_1947a}.
In this case, an unknown-constant-free upper bound is obtained for elliptic problems.
Moreover, it has been recently
shown~\cite{braess_pillwein_schoberl_2009a,ern_vohralik_2015a,ErnVo:20}
that such a strategy provides a $p$-robust estimator which means that the lower bound
does not depend on the polynomial degree $p$. The recent work~\cite{congreve_gedicke_perugia_2019a}
proposes to use this type of estimators for discontinuous Galerkin discretizations of the Helmhotz
problem. The authors consider a simpler configuration of~\eqref{eq_helmholtz_strong} with
$\GD = \emptyset$ and $d=2$ and evaluate the error in the $H^1(\Omega)$-seminorm $|u-u_h|_{1,\Omega}$
only. While the proposed estimator is $p$-robust, both reliability and efficiency only hold up to the uncomputable scaled error terms $k \|u-u_h\|_{0,\Omega}$ and
$k^{\frac12}\|u-u_h\|_{0,\GA}$. These terms asymptotically
vanish, but their actual size is unknown.

In the present work, we follow some of the main arguments developed in~\cite{dorfler_sauter_2013a}
but we employ equilibrated flux estimators instead of residual estimators.
We establish our results in the {\em energy norm}
\begin{equation}
\label{eq_en}
\norm{u-u_h}_{1,k, \Omega}^2
\eq
k^2 \|u-u_h\|_{0,\Omega}^2 + k\|u-u_h\|_{0,\GA}^2 + |u-u_h|_{1,\Omega}^2,
\end{equation}
and prove the global upper bound
\begin{subequations}
\begin{equation} \label{eq_rel}
\norm{u-u_h}_{1,k,\Omega} \leq \Cupper \eta, \qquad
\Cupper=\sqrt{2} + \theta_1(\Cba),
\end{equation}
where $\eta$ is a fully computable equilibrated flux estimator
and $\theta_1(t):=\sqrt{2t + 2t^2}$. The factor $\sqrt{2}$
can be further decreased to the optimal value of $1$ modulo additional technical developments
that we detail in Theorem~\ref{theorem_upper_bounds} below. Thus, asymptotically, \eqref{eq_rel}
is {\em unknown-constant-free}, in improvement of~\eqref{eq_upper_bound_sauter_dorfler}, and
additionally the entire energy norm (including the scaled $L^2(\Omega)$ and $L^2(\GA)$ terms)
is controlled in contrast to~\cite{congreve_gedicke_perugia_2019a}. Though the asymptotic nature
of the approximability factor $\Cba$ is known, \cf~\eqref{eq_theta_nontrapping}, $\Cba$ cannot
be computed or estimated from~\eqref{eq_theta_nontrapping}, so that the bound~\eqref{eq_rel} is not guaranteed
in general. We succeed in removing this remaining deficiency and find a fully computable
upper bound on $\Cba$ in several configurations of interest,
see Theorem~\ref{theorem_cba} below. Remarkably, the bound~\eqref{eq_rel} then
becomes {\em guaranteed}, and this in {\em any regime} (unresolved, resolved, asymptotic),
though the upper bound may largely overestimate the error.

The second main property of our estimators is the {\em local efficiency}
\begin{equation}
\label{eq_lower_bound_flux}
\eta_K \leq \ClowK \norm{u-u_h}_{1,k,\oK}, \quad \ClowK = C(\kappa_{\TT_K})\left (
1 + \left (\frac{kh_{\oK}}{p}\right )^{\frac12} + \frac{kh_{\oK}}{p}
\right ),
\end{equation}
\end{subequations}
for every mesh element $K$, where the term $(\frac{kh_{\omega_K}}{p})^{\frac12}$
appears for $\GA \neq \emptyset$.
Importantly, in contrast to~\eqref{eq_lower_bound_sauter_dorfler},
the constant $C(\kappa_{\TT_K})$ in~\eqref{eq_lower_bound_flux} is independent of the
polynomial degree $p$ (only depends on the local shape-regularity parameter $\kappa_{\TT_K}$),
and, in contrast to~\cite{congreve_gedicke_perugia_2019a}, the scaled $L^2(\Omega)$ and $L^2(\GA)$
terms are included. Since $\Cupper$ is also independent of $p$, we conclude that our a posteriori
error estimator is $p$-robust in all regimes. Moreover the estimator is robust with respect to
the wavenumber $k$ in the asymptotic regime where $\Cba\le 1$ (and $\frac{kh}{2 \pi p} \leq 1$).

Our manuscript is organized as follows.
In Section~\ref{sec:settings}, we make precise the functional and discrete settings and
state our main results. We perform a general analysis of the relationships between the
energy error~\eqref{eq_en}
and dual norms of the residual in Section~\ref{sec:abstract_error_estimates}, and we employ
these results in the context of flux equilibration in Section~\ref{sec:equilibrated_estimator}.
In Section~\ref{sec_pre_asymptotic}, we prove our fully computable upper bounds on the factor $\Cba$.
We present numerical experiments that illustrate our findings in
Section~\ref{sec:numerical_experiments} and draw our conclusions in Section~\ref{sec:conclusion}.
Finally, in Appendix~\ref{sec_tcba}, we establish an intermediate result related to boundary
and volume data.

\section{Setting and main results}
\label{sec:settings}

This section details the setting and presents our main results.

\subsection{Variational formulation}

We recast problem
\eqref{eq_helmholtz_strong} into a weak form that consists in
finding $u \in H^1_\GD(\Omega)$ such that
\begin{equation}
\label{eq_helmholtz_weak}
b(u,v) = (f,v) + (g,v)_\GA \qquad \forall v \in H^1_\GD(\Omega),
\end{equation}
where
\begin{equation}
\label{eq_b}
b(u,v) \eq -k^2 (u,v) -ik(u,v)_\GA + (\grad u,\grad v)
\end{equation}
and
\begin{equation*}
H^1_\GD(\Omega) \eq \left \{ v \in H^1(\Omega)| \; v = 0 \text{ on } \GD \right \}.
\end{equation*}

Above and hereafter, for $m \in \mathbb N$ and an open set $\mathcal O$,
$H^m(\mathcal O)$ denotes the Sobolev space of order $m$.
We also employ the notations
$\|{\cdot}\|_{m,\mathcal O}$ and $|{\cdot}|_{m,\mathcal O}$ for the
norm and semi-norm on $H^m(\mathcal O)$. Also, $({\cdot},{\cdot})_{\mathcal O}$
denotes the $L^2(\mathcal O)$ inner-product, and we drop the subscript
when $\mathcal O = \Omega$. We refer to~\cite{adams_fournier_2003a,ciarlet_2002a,girault_raviart_1986a}
for the definition and essential properties of the Sobolev spaces.
In the following, we assume that the operator $B:
H^1_\GD(\Omega)\to (H^1_\GD(\Omega))'$ associated with the sesquilinear form $b$
is a bounded isomorphism. This assumption holds for all $k>0$ if $|\GA|>0$, and it
allows for $|\GA| = 0$ provided that $k^2$ is not an eigenvalue of the
Laplace operator in $\Omega$. We equip $H^1_\GD(\Omega)$ with the {\em energy norm}
\begin{equation} \label{eq_en_norm}
\norm{v}_{1,k,\Omega}^2 \eq k^2 \|v\|_{0,\Omega}^2 + k\|v\|_{0,\GA}^2 + |v|_{1,\Omega}^2
\qquad \forall v \in H^1_\GD(\Omega).
\end{equation}
Remark that this choice leads to the sharp continuity estimate
\begin{equation}
\label{eq_cont_en}
|b(\phi,v)| \leq \norm{\phi}_{1,k,\Omega}\norm{v}_{1,k,\Omega} \qquad \forall \phi, v \in H^1(\Omega),
\end{equation}
which is of particular interest in the context of goal-oriented error estimation
\cite{darrigrand_pardo_muga_2015a,peraire_patera_1999a,sarrate_peraire_patera_1999a}.

\subsection{Discrete solution}

Consider a mesh $\TT_h$ of $\Omega$ that consists of triangular
or tetrahedral elements $K$ such that $\bigcup_{K \in \TT_h} \overline{K} = \overline{\Omega}$ and such that for two distinct elements $K_\pm \in \TT_h$, the intersection
$\overline{K_+} \cap \overline{K_-}$ is either empty, a single vertex,
a full edge, or a full face of $K_+$ and $K_-$. We further require
$\TT_h$ to be compatible with the partition $\overline{\GD} \cup \overline{\GA}$
of $\partial \Omega$, \ie, that each boundary mesh face $\overline{F} \subset \partial \Omega$
is either included in $\overline{\GD}$ or in $\overline{\GA}$.
For every element $K \in \TT_h$, we also define its diameter $h_K$ and its inscribed ball radius $\rho_K$ by
\begin{equation*}
h_K \eq \sup_{\xx,\yy \in K} |\xx - \yy|, \quad
\rho_K \eq \sup \left \{ r > 0| \; \exists \xx \in K; \; B(\xx,r) \subset K
\right \}.
\end{equation*}
We assume that the mesh is shape-regular in the sense that
every $K \in \TT_h$ satisfies
\begin{equation*}
\kappa_K \eq \frac{\rho_K}{h_K} \geq \kappa > 0
\end{equation*}
for a fixed constant $\kappa$. These hypotheses are standard~\cite{ciarlet_2002a} and
not restrictive in particular in the sense that they do not preclude strong grading
of the mesh. We will also sometimes use the notation $h \eq \max_{K \in \TT_h} h_K$.
If $\TT \subset \TT_h$ is a subset of cells $K$, we write
$\kappa_{\TT} \eq \min_{K \in \TT} \kappa_K$.

For an integer $p\geq1$, consider the space $\PP_p(K)$ of polynomials
on $K$ of degree at most $p$, the broken space
\begin{equation*}
\PP_p(\TT_h) \eq \left \{
v \in L^2(\Omega)| \; v|_K \in \PP_p(K); \; \forall K \in \TT_h\right \},
\end{equation*}
and define the approximation space
\begin{equation} \label{eq_V_h}
V_h \eq \PP_p(\TT_h) \cap H^1_\GD(\Omega).
\end{equation}
The discrete version of~\eqref{eq_helmholtz_weak} seeks for $u_h \in V_h$
such that
\begin{equation}
\label{eq_helmholtz_discrete}
b(u_h,v_h) = (f,v_h) + (g,v_h)_\GA \qquad \forall v_h \in V_h.
\end{equation}
In what follows we assume that there is at least
one solution to~\eqref{eq_helmholtz_discrete}. Uniqueness is not required for the
present analysis to hold, \ie, $u_h\in V_h$ can be any solution to~\eqref{eq_helmholtz_discrete}.

\subsection{Approximation properties of adjoint solutions}
\label{sec:adjoint_solutions}

The approximation properties of the space $V_h$ from~\eqref{eq_V_h} play a
central role in the forthcoming analysis.
Specifically, following~\cite{dorfler_sauter_2013a,melenk_sauter_2011a},
we consider the real number $\Cba$ defined in the introduction,
\ie
\begin{equation}
\label{eq_Cba_bis}
\Cba \eq k
\sup_{\phi \in L^2(\Omega) \setminus \{0\}}
\inf_{v_h \in V_h} \frac{|u_\phi^\star - v_h|_{1,\Omega}}{\|\phi\|_{0,\Omega}},
\end{equation}
where $u_\phi^\star$ denotes the unique (adjoint) solution in $H^1_\GD(\Omega)$
such that
\begin{equation} \label{eq_helmholtz_weak_adjoint}
b(w,u_\phi^\star) = (w,\phi) \quad \forall w \in H^1_\GD(\Omega).
\end{equation}
In strong form, we have $-k^2 u_\phi^\star - \Delta u_\phi^\star = \phi$
in $\Omega$, $u_\phi^\star= 0$ on $\GD$, and
$\grad u_\phi^\star {\cdot} \nn + iku_\phi^\star = 0$ on $\GA$.
The real number $\Cba$ is independent of the data $f$ and $g$ but implicitly
depends on $\Omega$, $\GD$, $k$, and $V_h$. It combines the ability of the
discrete space $V_h$ to approximate the solution $u_\phi^\star$
with the stability of the adjoint problem~\eqref{eq_helmholtz_weak_adjoint}.
Indeed, we have
\begin{equation}
\label{eq_best_approximation}
k \inf_{v_h \in V_h} |u_\phi^\star - v_h|_{1,\Omega} \leq \Cba \|\phi\|_{0,\Omega} \qquad \forall \phi \in L^2(\Omega).
\end{equation}

The dependence of $\Cba$ on $\Omega$, $\GD$, $k$, and $V_h$
(or less precisely but more concretely on the discretization parameters $h$ and $p$)
is not known in general. As previously stated, $\Cba$ is bounded uniformly with respect
to $h$ and $p$ (take $v_h = 0$ in~\eqref{eq_Cba_bis}). However, sharper bounds on $\Cba$
are expected to hold by taking other approximating functions $v_h\in V_h$, leading to
upper bounds on $\Cba$ that depend on $h$ and $p$.  For instance, for the case of scattering
by a smooth non-trapping obstacle~\cite{melenk_sauter_2011a}, see
also~\cite{chaumontfrelet_nicaise_2018a,chaumontfrelet_nicaise_2019a},
the upper bound~\eqref{eq_theta_nontrapping} is available. More generally,
as long as the domain $\Omega$ features an elliptic regularity shift for the Laplace operator
(see \eg~\cite{grisvard_1992a}), one can state, \cf\  Section~\ref{sec_pre_asymptotic} below, that
\begin{equation} \label{eq:bnd_Cba_shift}
\Cba \leq C(\widehat \Omega,\widehat \Gamma_{\rm D},k,\kappa) \left (\frac{kh}{p}\right )^\varepsilon
\end{equation}
for some $\varepsilon = \varepsilon(\widehat \Omega,\widehat \Gamma_{\rm D}) > 0$.
Here and above, $\widehat \Omega \eq (1/h_\Omega)\Omega$ and
$\widehat \Gamma_{\rm D} \eq (1/h_\Omega) \GD$ denote scaled versions of $\Omega$ and $\GD$
which are independent of the size of the computational domain $\Omega$.
We provide a computable upper bound on $\Cba$ in
some settings of interest in Theorem~\ref{theorem_cba} below.

% We further point out that for any
% $f \in L^2(\Omega)$, there exists a unique $u^\star_f \in H^1_\GD(\Omega)$
% such that $b(w,u^\star_f) = (w,f)$ for all $w \in H^1_\GD(\Omega)$, and
% \eqref{eq_best_approximation} holds for $u_f^\star$ as well.
% We will also need the constant $\tCba$
% \begin{equation*}
% \tCba =
% \sup_{g \in L^2(\GA) \setminus \{0\}}
% \inf_{v_h \in V_h} \frac{|u_g - v_h|_{1,\Omega}}{\|g\|_{0,\GA}},
% \end{equation*}
% where $u_g$ is the unique element of $H^1_\GD(\Omega)$ such that
% $b(u_g,v) = (g,v)_\GA$. An estimate analog to \eqref{eq_best_approximation},
% namely
% \begin{equation}
% \label{eq_best_approximation_boundary}
% \inf_{v_h \in V_h} |u_g - v_h|_{1,\Omega} \leq \tCba \|g\|_{0,\GA},
% \end{equation}
% holds for $\tCba$. We remark that $u_g$ is a weak solution to
% \eqref{eq_helmholtz_strong} with $f = 0$ and the right-hand side
% $g$ in the boundary condition on $\GA$. We again note that
% there also exists a unique $u^\star_g \in H^1_\GD(\Omega)$
% such that $b(w,u^\star_g) = (w,g)_\GA$ for all $w \in H^1_\GD(\Omega)$,
% and that \eqref{eq_best_approximation_boundary} holds for $u_g^\star$
% as well.

In addition to~\eqref{eq_Cba_bis}, we also introduce
\begin{equation}
\label{eq_tCba}
\tCba \eq k^{\frac12} \sup_{\psi \in L^2(\GA) \setminus \{0\}} \inf_{v_h \in V_h}
\frac{|\widetilde u_\psi^\star - v_h|_{1,\Omega}}{\|\psi\|_{0,\GA}},
\end{equation}
where $\widetilde u_\psi^\star$ is the (adjoint) solution in $H^1_\GD(\Omega)$
such that
\begin{equation} \label{eq_helmholtz_adjoint_bc}
b(w,\widetilde u_\psi^\star) = (w,\psi)_\GA \quad \forall w \in H^1_\GD(\Omega).
\end{equation}
In strong form, we have $-k^2 \widetilde u_\psi^\star - \Delta \widetilde u_\psi^\star = 0$
in $\Omega$, $\widetilde u_\psi^\star= 0$ on $\GD$, and
$\grad \widetilde u_\psi^\star {\cdot} \nn + ik \widetilde u_\psi^\star = \psi$ on $\GA$.
The quantity $\tCba$ is similar to $\Cba$ since it measures the best-approximation
error and the stability of the adjoint problem~\eqref{eq_helmholtz_adjoint_bc}
which features a boundary right-hand side instead of a volume right-hand side as
in~\eqref{eq_helmholtz_weak_adjoint}. As in the case of volume data, $\tCba$ is bounded uniformly with respect to $h$ and $p$ (take $v_h = 0$ in~\eqref{eq_tCba}) and if the domain $\Omega$
features an elliptic regularity shift, we have
\begin{equation}\label{eq:bnd_tCba_shift}
\tCba \leq C(\widehat \Omega,\widehat \Gamma_{\rm D},k,\kappa)
\left (\frac{kh}{p} \right )^{\widetilde \varepsilon},
\end{equation}
where $\widetilde \varepsilon = \widetilde \varepsilon(\widehat \Omega,\widehat \Gamma_{\rm D}) > 0$.
In addition, under a reasonable assumption that is
satisfied in all the configurations depicted in Figure~\ref{figure_domains},
we prove in Appendix~\ref{sec_tcba} the estimate
\begin{equation}
\label{eq_upper_bound_tcba}
\tCba \leq
C(\widehat \Omega,\widehat\Gamma_{\rm D})
\Cqi(\kappa)
\left (\left (\frac{kh}{p}\right )^{\frac12} + \frac{kh}{p} \right ) + \left (
C(\widehat \Omega,\widehat\Gamma_{\rm D}) \frac{1}{k h_\Omega}
\left ( 1 + \frac{1}{k h_\Omega}\right ) + 2 \right ) \Cba,
\end{equation}
where $\Cqi(\kappa)$ is a quasi-interpolation constant that only
depends on the shape-regularity parameter $\kappa$~\cite{hiptmair_pechstein_2017a,melenk_2005a}.
In particular, in the high-wavenumber regime when $k h_\Omega \to \infty$, we have
\begin{equation*}
\tCba \leq 2 \Cba + \mathcal O\left (\left (\frac{kh}{p}\right)^{\frac12}\right ).
\end{equation*}

\subsection{Equilibrated flux reconstruction}

For each vertex $\aaa$ in the set of vertices $\VV_h$ of the mesh $\TT_h$, set
\begin{subequations}
\label{eq_db}
\begin{equation}
d_\aaa \eq \psi_\aaa \pi_h^p (f) + \psi_\aaa k^2 u_h - \grad \psi_\aaa {\cdot} \grad u_h,
\end{equation}
and
\begin{equation}
b_\aaa \eq
\left \{
\begin{array}{ll}
- \psi_\aaa \tpi_h^p(g) - \psi_\aaa i k u_h & \text{ on } \partial \oa \cap \GA,
\\
0 & \text{ on } \partial \oa \setminus \GA.
\end{array}
\right .
\end{equation}\end{subequations}
Here $u_h$ is a finite element solution to~\eqref{eq_helmholtz_discrete}
and $\psi_\aaa$ is the ``hat function'' associated with the vertex $\aaa$: the unique function in
$\PP_1(\TT_h) \cap H^1(\Omega)$ such that $\psi_\aaa(\aaa) = 1$ and $\psi_\aaa(\aaa') = 0$
for all $\aaa' \in \VV_h \setminus \{\aaa\}$. Moreover, the elementwise/facewise $L^2$ projectors
$\pi_h^q$ and $\tpi_h^q$ are respectively defined for any integer $q \geq 0$, any $v \in L^2(\Omega)$,
and any $w \in L^2(\GA)$ by
\begin{subequations}
\label{eq_proj}
\begin{alignat}{4}
& \pi_h^q (v)|_K \in \PP_q(K), \quad & &
(\pi_h^q (v)|_K,\phi)_K = (v|_K,\phi)_K \quad & & \forall \phi \in \PP_q(K), \, \forall K \in \TT_h,\\
& \tpi_h^q (w)|_F \in \PP_q(F), & &
(\tpi_h^q (w)|_F,\phi)_F = (v|_F,\phi)_F & & \forall \phi \in \PP_q(F), \, \forall F \in \FF_h,
\, F \subset \overline{\GA},
\end{alignat}
\end{subequations}
where $\FF_h$ is the set of mesh faces.

For an integer $q \geq0$, let $\PPP_q(K)$ be the set of vector-valued
functions that have all components in $\PP_q(K)$.
We introduce the space $\RT_q(K) \eq \xx \PP_q(K) + \PPP_q(K)$
from~\cite{nedelec_1980a,raviart_thomas_1977a} and define
\begin{eqnarray*}
\PP_q(\TT)
&\eq&
\left \{
v \in L^2(\omega)| \; v|_K \in \PP_q(K); \; \forall K \in \TT
\right \},
\\
\RT_q(\TT)
&\eq&
\left \{
\vv \in \LL^2(\omega)| \; \vv|_K \in \RT_q(K); \; \forall K \in \TT
\right \}
\end{eqnarray*}
for any subset of mesh elements $\TT \subset \TT_h$ with the corresponding open subdomain $\omega$.
In particular, for each vertex $\aaa \in \VV_h$, we denote by $\TT_\aaa$ the patch of all elements
$K \in \TT_h$ having $\aaa$ as a vertex and by $\oa$ the corresponding open subdomain.

Following~\cite{braess_pillwein_schoberl_2009a,Dest_Met_expl_err_CFE_99,ern_vohralik_2015a},
we consider the following inexpensive and fully parallel post-processing procedure:

\begin{defi}[Equilibrated flux reconstruction]
\label{def_sigma_a}
For every vertex $\aaa \in \VV_h$, let
\begin{equation}
\label{eq_definition_flux}
\ssig^\aaa_h \eq
\underset{\substack{
\ttau^\aaa_h \in \RT_{p+1}(\TT_\aaa) \cap \HH(\ddiv,\oa)
\\
\div \ttau^\aaa_h = d_\aaa \; \text{\em in } \oa
\\
\ttau^\aaa_h {\cdot} \nn = b_\aaa \; \text{\em on } \Gamma_\aaa
}}{\operatorname{argmin}}
\|\ttau_h^\aaa + \psi_\aaa \grad u_h\|_{0,\oa},
\end{equation}
where $\Gamma_{\aaa}$ is the boundary $\partial \omega_{\aaa}$ without those faces of
$\Gamma_{\mathrm{D}}$ sharing the vertex $\aaa$. The equilibrated flux reconstruction
is computed by summing up the patchwise contributions (after zero extension) as
\begin{equation} \label{eq_definition_flux_glob}
\ssig_h \eq \sum_{\aaa \in \VV_h} \ssig_h^\aaa \in \RT_{p+1}(\TT_h) \cap \HH(\ddiv,\Omega).
\end{equation}
\end{defi}

Since $d_\aaa \in \mathcal P_{p+1}(\TT_\aaa)$ and $b_\aaa \in \mathcal P_{p+1}(F)$
for each face $F \subset \partial \oa$, we see that the minimization set in~\eqref{eq_definition_flux}
is non-empty when
\begin{equation}
\label{tmp:compatibility_condition}
(d_\aaa,1)_\oa = (b_\aaa,1)_{\partial \oa} \qquad \forall \aaa \notin \overline{\GD}.
\end{equation}
Since $\psi_\aaa \in V_h$ when $\aaa \notin \overline{\GD}$, we have
from~\eqref{eq_helmholtz_discrete} in combination with~\eqref{eq_proj}
\begin{equation*}
(\psi_\aaa \pi_h^p (f),1)_\oa
=
(f,\psi_\aaa)
=
b(u_h,\psi_\aaa) - (g,\psi_\aaa)_\GA
= b(u_h,\psi_\aaa) - (\psi_\aaa \tpi_h^p(g),1)_\GA.
\end{equation*}
As a result, the compatibility condition~\eqref{tmp:compatibility_condition} follows from
the definition~\eqref{eq_b} of $b(\cdot,\cdot)$ and
\begin{align*}
(\psi_\aaa \pi_h^p (f),1)_\oa
={}&
-k^2(u_h,\psi_\aaa)_\oa -ik (u_h,\psi_\aaa)_{\partial \oa \cap \GA} +
(\grad u_h,\grad \psi_\aaa)_\oa - (\psi_\aaa \tpi_h^p(g),1)_\GA
\\
={}&
(\grad \psi_\aaa {\cdot} \grad u_h - k^2 \psi_\aaa u_h,1)_\oa + (b_\aaa,1)_{\partial \oa}.
\end{align*}
Then the uniqueness of the minimizer in~\eqref{eq_definition_flux}
follows from standard convexity arguments.

\begin{rema}[Equilibrated flux]
In practice, the local fluxes $\ssig_h^\aaa$ are obtained by
invoking the Euler--Lagrange conditions of the constrained
minimization problem~\eqref{eq_definition_flux}. This leads to
the problem of finding the unique pair
$(\ssig_h^\aaa,r^\aaa_h) \in \RT_{p+1}(\TT_\aaa) \cap \HH(\ddiv,\oa)
\times \PP^0_{p+1}(\TT_\aaa)$ such that
$\ssig_h^\aaa {\cdot} \nn = b_\aaa$ on $\Gamma_\aaa$ and
\begin{equation*}
\left \{
\begin{array}{rcll}
(\ssig_h^\aaa,\vv_h)_{\oa} - (r^\aaa_h,\div \vv_h)_{\oa} &=& - (\psi_\aaa \grad u_h,\vv_h)_{\oa}
&
\forall \vv_h \in \RT_{p+1}(\TT_\aaa) \cap \HH_\GA(\ddiv,\oa),
\\
(\div \ssig_h^\aaa,v_h)_{\oa} &=& (d_\aaa,v_h)_{\oa}
&
\forall v_h \in \PP_{p+1}^0(\TT_\aaa),
\end{array}
\right .
\end{equation*}
where $\HH_\GA(\ddiv,\oa)$ imposes zero normal flux through $\Gamma_\aaa$
and where $\PP_{p+1}^0(\TT_\aaa)$ is composed of the functions
$v_h \in \PP_{p+1}(\TT_\aaa)$ such that $(v_h,1)_\oa = 0$ if $\aaa \not \in \overline{\GD}$
whereas $\PP_{p+1}^0(\TT_\aaa) \eq \PP_{p+1}(\TT_\aaa)$ otherwise.
\end{rema}

\subsection{Data oscillation}

For each element $K \in \TT_h$, we define
\begin{equation}\label{eq_osc}
\osc_K(f,g) \eq \frac{h_K}{\pi}\|f - \pi_h^p (f)\|_{0,K} +
C_{{\rm tr},K}
\left (
\frac{h_K}{\pi} \right )^{\frac12} \|g - \tpi_h^p(g)\|_{0,\partial K \cap \GA}
\end{equation}
with
\begin{equation*}
C_{{\rm tr},K}^2 \eq N_{K,\GA} \frac{3}{4\pi}
\left ( 1 + \frac{1}{\pi} \right )
\left (\frac{h_K}{\rho_K}\right )^d,
\end{equation*}
where $N_{K,\GA}$ is the number of faces of $K$ that belong to $\GA$, and the remaining
part of $C_{{\rm tr},K}$ comes from the standard trace inequality for one face of $K$,
\cf~\cite[Section 4.2]{cheddadi_fucik_prieto_vohralik_2009a}.
Owing to this definition of $\osc_K(f,g)$ and to the Poincar\'e inequality
(recall that $K$ is convex)
\begin{equation*}
% \|v - \pi_h^p v\|_{0,K} \leq
\|v\|_{0,K} \leq \frac{h_K}{\pi} |v|_{1,K}
\end{equation*}
as well as to the trace inequality
\begin{equation*}
\|v\|_{0,\partial K \cap \partial \GA}
\leq
C_{{\rm tr},K} \bigg( \frac{h_K}{\pi}\bigg)^{\frac12}|v|_{1,K},
\end{equation*}
we have
\begin{equation}
\label{eq_upper_bound_oscK}
\|f-\pi_h^p (f)\|_{0,K}\|v\|_{0,K}
+
\|g - \tpi_h^p(g)\|_{0,\partial K \cap \GA}\|v\|_{0,\partial K \cap \GA}
\leq
\osc_K(f,g) |v|_{1,K},
\end{equation}
for any function $v \in H^1(K)$ with zero mean value on $K$.
If $\TT = \{K\}$ is a collection of several elements $K$,
we set
\begin{equation*}
\osc_\TT(f,g) \eq \left (
\sum_{K \in \TT} \osc_K(f,g)^2
\right )^{\frac12},
\end{equation*}
and we omit the subscript when $\TT = \TT_h$.

\subsection{Main results}

Let $u$ be the weak solution to~\eqref{eq_helmholtz_weak} and let $u_h$ be any
conforming finite element approximation solving~\eqref{eq_helmholtz_discrete}.
Let the equilibrated flux reconstruction $\ssig_h$ be prescribed by Definition~\ref{def_sigma_a}
and define the local and global error estimators as
\begin{equation}
\label{eq_definition_eta}
\eta_K \eq \|\ssig_h + \grad u_h\|_{0,K}, \quad
\eta \eq \left (
\sum_{K \in \mathcal T_h} \eta_K^2
\right )^{\frac12}.
\end{equation}
Recall finally the approximability factors $\Cba$ and $\tCba$ defined respectively
by~\eqref{eq_Cba_bis} and~\eqref{eq_tCba}. Our main result on the error upper bound is the following.

\begin{theo}[Upper bounds]
\label{theorem_upper_bounds}
The following global upper bound holds true:
\begin{equation}
\label{eq_est_up}
\norm{u-u_h}_{1,k,\Omega} \leq \Cupper
\left (
\eta + \osc(f,g)
\right ),
\end{equation}
where
\begin{equation} \label{eq:def_Cupper}
\Cupper \eq \min(\sqrt{2} + \tilde \theta_1(\Cba),1+\tilde\theta_2(\Cba,\tCba))
\end{equation}
and
\begin{align}\label{eq_Cupperc}
0\le \tilde\theta_1(t)\eq {}&\sqrt{\left (\frac{1}{2} + \sqrt{\frac{1}{4} + t^2}\right ) +
 \left (\frac{1}{2} + \sqrt{\frac{1}{4} +
t^{2}}\right )^2 + t^2} - \sqrt{2} \\
\leq {}& \theta_1(t) \eq \sqrt{2t + 2t^2}, \nonumber
\end{align}
together with
\begin{align} \label{eq_Cupperf}
0 \le \tilde\theta_2(t,\tilde t) \eq {}& \sqrt{\left (
\frac{1}{2} + \sqrt{\frac{1}{4} + t^2}
\right )^2 + t^2 + \tilde t^2} -1 \\
\nonumber
\leq {}& \theta_2(t,\tilde t) \eq \sqrt{t + 2t^2 + \tilde t^2}.
\end{align}%
\end{theo}

Note that the bounds $0\le \tilde\theta_1(t)\le \theta_1(t)$ and
$0\le \tilde\theta_2(t,\tilde t) \le \theta_2(t,\tilde t)$ in \eqref{eq_Cupperc}
and \eqref{eq_Cupperf}, respectively, are straightforward, so that the only estimate
to prove is~\eqref{eq_est_up}. This result is established in
Propositions~\ref{proposition_upper_bound_full_norm_coarse},
\ref{proposition_global_upper_bound_fine}, and~\ref{proposition_global_upper_bound}.

\begin{rema}[Upper bounds in the asymptotic regime] \label{rem_ass_reg}
When $\Cba \to 0$ and $\tCba \to 0$, the presence of $\tilde\theta_2$
in~\eqref{eq:def_Cupper}
implies that $\Cupper \to 1$. This indeed happens in most configurations of
practical interest when $\frac{h}{p}\to0$ at fixed $k$, see~\eqref{eq:bnd_Cba_shift} and~\eqref{eq:bnd_tCba_shift} or~\eqref{eq_upper_bound_tcba}. Thus, the a
posteriori error estimate~\eqref{eq_est_up} is asymptotically
constant-free. Unfortunately, the approximation factors $\Cba$ and $\tCba$ cannot be computed explicitly. To circumvent this issue, we establish a computable estimate
on $\Cba$ in Theorem~\ref{theorem_cba} below, to be used in conjunction with $\tilde\theta_1$
in~\eqref{eq:def_Cupper}. This will make the bound~\eqref{eq_est_up}
guaranteed but, unfortunately, the computable bound on $\Cba$ is in general too rough and leads to a large error overestimation. This is in particular reflected by the fact that this computable bound
does not tend to zero as $\frac{h}{p}\to0$ at fixed $k$ in some cases.
\end{rema}

Our main result on the error lower bound is as follows.

\begin{theo}[Lower bounds]
\label{theorem_lower_bounds}
The following local lower bounds also hold true:
\begin{equation} \label{eq_est_low_loc}
\eta_K \leq \ClowK
\norm{u-u_h}_{1,k,\oK} +
C(\kappa_{\TT_K}) \osc_{\TT_K}(f,g) \quad \forall K \in \TT_h
\end{equation}
with
\[
\ClowK \eq C(\kappa_{\TT_K}) \left (
1 + \left (\frac{kh_{\oK}}{p}\right )^{\frac12} + \frac{kh_{\oK}}{p}
\right ),
\]
where $C(\kappa_{\TT_K})$ only depends on the shape-regularity parameter $\kappa_{\TT_K}$
of the mesh in the patch $\TT_K$ of elements sharing a vertex with $K$ and where $h_\oK$ is the
diameter of the corresponding subdomain $\oK$.

In addition, the following global lower bound holds true:%
\begin{equation} \label{eq_est_low_glob}
\eta \leq \Clow
\norm{u-u_h}_{1,k,\Omega} + C(\kappa) \osc(f,g),
\end{equation}
where
\[
\Clow \eq C(\kappa) \left (
1 + \left (\frac{kh}{p}\right )^{\frac12} + \frac{kh}{p}
\right ),
\]
and $C(\kappa)$ only depends on the shape regularity-parameter $\kappa$ of the mesh.
\end{theo}

\begin{rema}[$p$-robustness, $k$-robustness]
Since the approximation factor $\Cba$ defined in~\eqref{eq_Cba_bis}
is bounded independently of the mesh size $h$ and polynomial degree
$p$ (taking $v_h = 0$ in~\eqref{eq_Cba_bis}), the same holds for $\Cupper$. 
Moreover, trivially, $\frac{kh}{p} \leq k h_\Omega$.
Then, \eqref{eq_est_up} and either~\eqref{eq_est_low_loc} or~\eqref{eq_est_low_glob}
establish the equivalence of the error and the estimator independently of the mesh size $h$ and
polynomial degree $p$, leading in particular to 
polynomial-degree robustness in all regimes.
More precisely, for any wavenumber $k$, there exists a constant $C$, independent of $h$ and $p$,
such that the effectivity index $\Cupper \eta / \norm{u-u_h}_{1,k,\Omega}$ is bounded by $C$ for any $h$ and $p$; here $C$ can depend on $k$ and $h_\Omega$.
Moreover, robustness with respect to the wavenumber $k$ is achieved in the asymptotic regime where $\Cba\le 1$ (and $\frac{kh}{p}\le1$), so that $\Cupper$
and $\ClowK$ or $\Clow$ are bounded independently of $k$, $h$, and $p$. Indeed, $\Cba\le 1$ (and $\frac{kh}{p}\le1$), there exists a constant $C$, independent of $k$, $h$, and $p$, such that the effectivity index $\Cupper \eta / \norm{u-u_h}_{1,k,\Omega}$ is bounded by $C$ for any $k$, $h$, and $p$; here C is of order $1$ for shape-regular meshes.
\end{rema}

The local lower bound~\eqref{eq_est_low_loc} is established in
Proposition~\ref{prop_local_lower_bounds}. Though~\eqref{eq_est_low_glob} is easily obtained
from~\eqref{eq_est_low_loc} by summation over the mesh cells, we provide an estimate with a
sharper constant in Proposition~\ref{prop_global_lower_bound}.

Crucially, for several configurations of interest, the upper bound of Theorem~\ref{theorem_upper_bounds}
can be turned into a {\em guaranteed estimate} in {\em any regime} starting from the unresolved one
(on any mesh $\TT_h$ and for any polynomial degree $p$) through the following computable
upper bounds on the approximability factor $\Cba$.

\begin{theo}[Computable bounds on $\Cba$]
\label{theorem_cba} $ $

Case 1a) {\em (Scattering by a non-trapping obstacle)}. Assume that
$\Omega = \Omega_0 \setminus \overline{D}$, $\GA = \partial \Omega_0$,
and $\GD = \partial D$, where $\Omega_0,D \subset \mathbb R^d$ are two
open, bounded, connected sets such that $\overline D$ is a proper subset of $\Omega_0$ and
assume that the subset
\begin{equation} \label{eq_def_Omega_x0}
\mathcal{O}_{\GD,\GA} :=\{ \xx_0 \in \mathbb R^d \; | \;
(\xx - \xx_0) \cdot \nn \leq 0 \; \forall \xx \in \GD, \
(\xx - \xx_0) \cdot \nn >    0 \; \forall \xx \in \GA \}
\end{equation}
is nonempty (recall that $\nn$ points outward $\Omega$). Let
\begin{equation} \label{eq_Cstab}
\Cstab \eq \inf_{\xx_0 \in \mathcal{O}_{\GD,\GA}} \Cstabx,
\end{equation}
with
\begin{equation*}
\Cstabx \eq \frac{1}{h_\Omega} \left \{
\sup_{\xx \in \Omega} |\xx-\xx_0| +
\sup_{\xx \in \GA}
\left (
2 (\xx - \xx_0) {\cdot} \nn + \frac{|(\xx-\xx_0) \times \nn|^2}{(\xx-\xx_0) {\cdot} \nn}
\right )
\right \}.
\end{equation*}
Then, we have
\begin{equation}
\label{eq_cba_coarse}
\Cba \leq \left (
\left (
(d-1) + \Cstab k h_\Omega
\right )
+
\left (
(d-1) + \Cstab k h_\Omega
\right )^2
\right )^{\frac12}.
\end{equation}

Case 1b) {\em (Wave propagation in free space)}. Assume that $\Omega$ is convex and
$\GD = \emptyset$, so that $\mathcal{O}_{\GD,\GA}=\Omega$. Let $\Cstab$ be defined as above.
Then, we have
\begin{equation}
\label{eq_cba_fine}
\Cba \leq \Ci \left (d + \Cstab k h_\Omega\right ) \frac{kh}{p^\beta},
\end{equation}
where $\Ci$ is any approximation constant satisfying
\begin{equation} \label{eq_approx}
\inf_{v_h \in V_h} |v - v_h|_{1,\Omega}
\leq
\Ci \frac{h}{p^\beta} |v|_{2,\Omega}
\qquad \forall v \in H^1_\GD(\Omega) \cap H^2(\Omega),
\end{equation}
and such that it only depends on the mesh regularity parameter $\kappa$; typically $\beta=0$ or $\beta=1$
in~\eqref{eq_cba_fine} and~\eqref{eq_approx}.

Case 2a) {\em (Interior problem)}. Assume that $\GA = \emptyset$. Then
\begin{equation} \label{eq_cba_coarse_int}
\Cba \leq k \max_{j \in \mathbb N} \frac{\sqrt{\lambda_j}}{|\lambda_j - k^2|},
\end{equation}
where $\lambda_j$ is the $j$-th eigenvalue of the Laplace operator in $\Omega$ with Dirichlet
boundary conditions.

Case 2b) {\em (Convex interior problem)}. Assume that $\GA = \emptyset$ and that $\Omega$ is convex. Then
\begin{equation} \label{eq_cba_fine_int}
\Cba \leq \Ci \left (1 + \frac{k^2}{\min_{j \in \mathbb N}|\lambda_j - k^2|} \right )
\frac{kh}{p^\beta},
\end{equation}
where $\Ci$ is any approximation constant satisfying~\eqref{eq_approx}.
\end{theo}

\begin{rema}[Bounds of Theorem~\ref{theorem_cba}] \label{rem_Thm_sba}
In contrast to the a priori bound~\eqref{eq:bnd_Cba_shift}
which gives $\Cba\to0$ when $\frac{h}{p}\to0$ with fixed $k$,
our computable upper bounds
on $\Cba$ of Theorem~\ref{theorem_cba} do not share this property in cases
1a) and 2a). Further work is needed to close this gap, and we leave it here as an open question.
\end{rema}

The proof of Theorem~\ref{theorem_cba} is carried out in Section~\ref{sec_pre_asymptotic}.
The estimates~\eqref{eq_cba_coarse_int} and~\eqref{eq_cba_fine_int} follow from
standard properties of spectral decomposition and we only sketch the proof
at the beginning of Section~\ref{sec_pre_asymptotic}. The more involved
upper bounds~\eqref{eq_cba_coarse} and~\eqref{eq_cba_fine} are respectively
established in Sections~\ref{sec_cba_coarse} and~\ref{sec_cba_fine}.

Finally, the constant $\Cstab$ can be easily bounded
when $\GA$ has a simple geometry (which is usually the case in scattering
applications):

\begin{rema}[Constant $\Cstab$]
\label{remark_stability}
In the Case 1a) (resp. 1b) above, if $\Omega_0$ (resp. $\Omega$) is a circle or a
ball centered at $\xx_0$, then  $\Cstab \leq \frac32$.
If $\Omega_0$ (resp. $\Omega$) is a square centered at
$\xx_0$, then $\Cstab\leq \frac{3+\sqrt{2}}{2\sqrt{2}}$.
Finally, if $\Omega_0$ (resp. $\Omega$) is a cube centered at
$\xx_0$, then $\Cstab \leq \frac{3+\sqrt{3}}{2\sqrt{3}}$.
\end{rema}

Similarly, several expressions for the constant $\Ci$ can be
found in the literature~\cite{arcangeli_gout_1976a,carstensen_gedicke_rim_2012a,kobayashi_tsuchiya_2014a,liu_kikuchi_2010a},
leading to~\eqref{eq_approx} with $\beta=0$:

\begin{rema}[Constant $\Ci$] \label{rem_Ci}
A more precise definition than~\eqref{eq_approx} is
\[
\Ci \eq \frac{p^\beta}{h}\sup_{v \in H^1_\GD(\Omega) \cap H^2(\Omega)\setminus\{0\}}
\inf_{v_h \in V_h} \frac{|v - v_h|_{1,\Omega}}{|v|_{2,\Omega}}.
\]
The value of $\Ci$ can then be estimated by considering any interpolation operator.
For any $v \in H^1_\GD(\Omega) \cap H^2(\Omega)$, we can define its $\mathcal P_1$ nodal
interpolant $\II_h (v) \in V_h$ by
\begin{equation*}
(\II_h (v))(\aaa) = v(\aaa) \qquad \forall \aaa \in \VV_h.
\end{equation*}
Then we have (with $\beta=0$)
\begin{equation*}
%\label{eq_p1_interpolation}
\Ci \le \sup_{v \in H^1_\GD(\Omega) \cap H^2(\Omega)\setminus\{0\}}
\frac{|v - \II_h(v)|_{1,\Omega}}{h|v|_{2,\Omega}},
\end{equation*}
where  $\Ci$ only depends on $\kappa$. In particular, following Theorem 1.1 of
\cite{arcangeli_gout_1976a}, we have $\Ci \leq 3/\kappa$ if $d=2$ and $\Ci \leq 8/\kappa$
if $d=3$. Furthermore, sharper estimates are available for the specific case of triangles when
$d=2$, and we refer the reader
to~\cite{carstensen_gedicke_rim_2012a,kobayashi_tsuchiya_2014a,liu_kikuchi_2010a}.
In particular, for a mesh formed by isosceles right-angled triangles,
we will employ the estimate $\Ci \leq 0.493/\sqrt 2$ from~\cite{liu_kikuchi_2010a}.
\end{rema}

\begin{rema}[Eigenvalues]
Apart from very simple domain geometries, analytic expressions of
the eigenvalues $\lambda_j$ appearing
in~\eqref{eq_cba_coarse_int} and~\eqref{eq_cba_fine_int} are not
available. However, guaranteed a posteriori error estimators can be
employed to reliably estimate their value, and thus $\Cba$,
see~\cite{cances_dusson_maday_stamm_vohralik_2017a,Cars_Ged_LB_eigs_14,Liu_fram_eigs_15}
and the references therein. We also refer the reader to \cite{carstensen_storn_2018a},
where related arguments are employed in the context of least-squares discretizations.
\end{rema}

\begin{rema}[Data oscillation]
Actually, a slightly sharper upper bound than~\eqref{eq_est_up}, with data oscillation
integrated to the local error estimators, follows from~\eqref{eq_global_upper_bound} below. Also,
upon modifying the treatment of the data $f$ and $g$ as in~\cite{Dol_Ern_Voh_hp_16,ern_vohralik_2015a},
one could replace in $\osc_K(f,g)$ in~\eqref{eq_osc} the projections $\pi_h^{p}$ and $\tpi_h^{p}$ by
$\pi_h^{p+1}$ and $\tpi_h^{p+1}$ and gain an additional order for data oscillation in the upper
bound~\eqref{eq_est_up}. This is, however,  not possible in the lower bounds~\eqref{eq_est_low_loc}
and~\eqref{eq_est_low_glob}. We do not present these extensions here for the clarity of exposition.
\end{rema}

\section{Relation between the error and the dual norm of the residual}
\label{sec:abstract_error_estimates}

The aim of this section is to obtain lower and upper bounds on the
error between $u$ and $u_h$ by dual residual norms. We define the {\em residual}
$\RR(u_h) \in \left (H^1_\GD(\Omega)\right )'$ stemming from the weak
formulation~\eqref{eq_helmholtz_weak} by
\begin{equation} \label{eq_res_def}
\langle \RR(u_h),v \rangle \eq b(u-u_h,v) = (f,v) + (g,v)_\GA - b(u_h,v)
\qquad \forall v \in H^1_{\GD}(\Omega).
\end{equation}
We further introduce the residual norm
\begin{equation}\label{eq_res_norm_def}
\|\RR(u_h)\|_{-1,\Omega} \eq \sup_{v \in H^1_\GD(\Omega) \setminus \{0\}}
\frac{|\langle \RR(u_h),v\rangle|}{|v|_{1,\Omega}}.
\end{equation}
Note that in virtue of~\eqref{eq_helmholtz_discrete}, we have
\begin{equation}
\label{eq_res_orth}
\langle \RR(u_h),v_h \rangle = 0 \qquad \forall v_h \in V_h.
\end{equation}

\subsection{Global upper bounds}

We start by an upper bound on the $L^2(\Omega)$-norm
using the approximation factor $\Cba$ defined in~\eqref{eq_Cba_bis}.
The proof follows the lines of Lemma 4.7 of~\cite{dorfler_sauter_2013a}
and employs the usual Aubin--Nitsche duality argument.

\begin{lemm}[$L^2(\Omega)$-norm upper bound by the dual norm of the residual]
\label{lemma_upper_bound_l2}
We have
\begin{equation}
\label{eq_global_upper_bound_l2}
k\|u-u_h\|_{0,\Omega} \leq \Cba \|\RR(u_h)\|_{-1,\Omega}.
\end{equation}
\end{lemm}

\begin{proof}
We introduce $\xi$ as the unique element of $H^1_\GD(\Omega)$ such that
$b(w,\xi) = (w,u-u_h)$ for all $w \in H^1_\GD(\Omega)$. Selecting the
test function $w = u-u_h$ and using~\eqref{eq_res_def} and~\eqref{eq_res_orth},
we see that
\begin{eqnarray*}
k\|u-u_h\|_{0,\Omega}^2
=
k\langle \RR(u_h),\xi \rangle
=
k\langle \RR(u_h),\xi - \xi_h \rangle
\leq
k\|\RR(u_h)\|_{-1,\Omega}|\xi - \xi_h|_{1,\Omega}
\end{eqnarray*}
for all $\xi_h \in V_h$. Hence, using~\eqref{eq_Cba_bis} and its consequence~\eqref{eq_best_approximation}
and recalling the definition of $\xi$, we have
\begin{equation*}
k\|u-u_h\|_{0,\Omega}^2 \leq
\|\RR(u_h)\|_{-1,\Omega} \left (
k \inf_{\xi_h \in V_h} |\xi - \xi_h|_{1,\Omega}
\right )
\leq
\Cba \|\RR(u_h)\|_{-1,\Omega} \|u-u_h\|_{0,\Omega},
\end{equation*}
and~\eqref{eq_global_upper_bound_l2} follows. \qed
\end{proof}

We now record a useful quadratic inequality.
\begin{lemm}[Quadratic inequality]
\label{lemma_polynomial}
Assume that $x \geq 0$ satisfies
\begin{equation*}
ax^2 \leq c + bx,
\end{equation*}
for some constants $a,b,c > 0$. Then we have
\begin{equation*}
ax \leq \frac{b}{2} +  \sqrt{\frac{b^2}{4} + ac}.
\end{equation*}
\end{lemm}

With the help of Lemmas~\ref{lemma_upper_bound_l2} and~\ref{lemma_polynomial}, we obtain an upper bound on the $H^1(\Omega)$ semi-norm.

\begin{lemm}[$H^1(\Omega)$ semi-norm upper bound by the dual norm of the residual]
\label{lemma_upper_bound_h1}
We have
\begin{equation}
\label{eq_global_upper_bound_h1}
|u-u_h|_{1,\Omega} \leq \left (
\frac{1}{2} + \sqrt{\frac{1}{4} + \big(\Cba\big)^2}
\right ) \|\RR(u_h)\|_{-1,\Omega}.
\end{equation}
\end{lemm}

\begin{proof}
Using the definition~\eqref{eq_b} of $b$, \eqref{eq_res_def}, and Lemma~\ref{lemma_upper_bound_l2},
we have
\begin{align*}
|u-u_h|_{1,\Omega}^2
&=
\Re b(u-u_h,u-u_h) + k^2\|u-u_h\|_{0,\Omega}^2
\\
&\leq
\|\RR(u_h)\|_{-1,\Omega}|u-u_h|_{1,\Omega} +
\big(\Cba\big)^2\|\RR(u_h)\|_{-1,\Omega}^2,
\end{align*}
and~\eqref{eq_global_upper_bound_h1} follows from Lemma~\ref{lemma_polynomial}
with $a=1$, $b=\|\RR(u_h)\|_{-1,\Omega}$, and $c=\big(\Cba\big)^2\|\RR(u_h)\|_{-1,\Omega}^2$. \qed
\end{proof}

% \begin{proof}
% We see that we actually have
% \begin{equation*}
% P(x) = ax^2 - bx - c \leq 0.
% \end{equation*}
% Since the discriminant $\Delta = b^2 + 4ac$ is positive,
% $P$ has two distinct roots $x_\star < x^\star$. In addition,
% because the leading coefficient of $P$ is positive, $P(x) \leq 0$
% if and only if $x \in [x_\star,x^\star]$. But we have
% \begin{equation*}
% x_\star = \frac{b -  \sqrt{b^2 + 4ac}}{2a} \leq 0,
% \qquad
% x^\star = \frac{b +  \sqrt{b^2 + 4ac}}{2a} \geq 0.
% \end{equation*}
% Since $x \geq 0$, we have $x \geq x_\star$, and therefore,
% $P(x) \leq 0$ implies that $0 \leq ax \leq ax^\star$. The
% conclusion follows by observing that
% \begin{equation*}
% ax^\star = \frac{b}{2} +  \sqrt{\frac{b^2}{4} + ac}.
% \end{equation*}
% \end{proof}

We are now ready to establish global upper bounds in the energy norm.
Our first estimate is explicit in terms of $\Cba$.

\begin{prop}[First upper bound by the dual norm of the residual]
\label{proposition_upper_bound_full_norm_coarse}
We have
\begin{equation*}
\norm{u-u_h}_{1,k,\Omega} \leq (\sqrt{2}+\tilde\theta_1(\Cba)) \|\RR(u_h)\|_{-1,\Omega},
\end{equation*}
with $\tilde\theta_1$ defined in~\eqref{eq_Cupperc}.
\end{prop}

\begin{proof}
Employing the definitions~\eqref{eq_b} and~\eqref{eq_res_def}, we have
\begin{equation*}
k\|u-u_h\|_{0,\GA}^2 = -\Im b(u-u_h,u-u_h) \leq \|\RR(u_h)\|_{-1,\Omega}|u-u_h|_{1,\Omega},
\end{equation*}
and using Lemma~\ref{lemma_upper_bound_h1}, it follows that
\begin{equation*}
k\|u-u_h\|_{0,\GA}^2 \leq \left (\frac{1}{2} + \sqrt{\frac{1}{4} + (\Cba)^2}\right )
\|\RR(u_h)\|_{-1,\Omega}^2.
\end{equation*}
Hence, from~\eqref{eq_en_norm} and~\eqref{eq_global_upper_bound_l2}--\eqref{eq_global_upper_bound_h1},
we infer that
\begin{align*}
\norm{u-u_h}_{1,k,\Omega}^2
\leq
\left ((\Cba)^2 + \left (\frac{1}{2} + \sqrt{\frac{1}{4} + (\Cba)^2}\right ) +
 \left (\frac{1}{2} + \sqrt{\frac{1}{4} +
(\Cba)^{2}}\right )^2 \right )\|\RR(u_h)\|_{-1,\Omega}^2,
\end{align*}
and the claim follows from the definition of $\tilde\theta_1$. \qed
\end{proof}

Proposition~\ref{proposition_upper_bound_full_norm_coarse} is not completely
satisfactory, since the asymptotic value for vanishing $\Cba$ of the prefactor
is $\sqrt{2}$.
We now provide a sharper analysis that shows that the asymptotic constant
can be brought to the optimal value of $1$. Recall the definition~\eqref{eq_tCba} of $\tCba$.

\begin{lemm}[$L^2(\GA)$-norm upper bound by the dual norm of the residual]
We have
\begin{equation}
\label{eq_global_upper_bound_boundary}
k^{\frac12} \|u-u_h\|_{0,\GA} \leq \tCba \|\RR(u_h)\|_{-1,\Omega}.
\end{equation}
\end{lemm}

\begin{proof}
We use again a duality argument.
We define $\chi$ as the unique element of $H^1_\GD(\Omega)$ such that
$b(w,\chi) = (w,u-u_h)_\GA$ for all $w \in H^1_\GD(\Omega)$.
Selecting the test function $w = u-u_h$, employing the definition of the residual~\eqref{eq_res_def},
and taking advantage of Galerkin's orthogonality~\eqref{eq_res_orth}, we have
\begin{equation*}
k \|u-u_h\|_{0,\GA}^2
=
k b(u-u_h,\chi)
=
k \langle \RR(u_h),\chi - \chi_h \rangle
\leq
k \|\RR(u_h)\|_{-1,\Omega} |\chi - \chi_h|_{1,\Omega}
\end{equation*}
for all $\chi_h \in V_h$. Then, recalling the above definition of $\chi$ and the
definition \eqref{eq_tCba} of $\tCba$, we obtain~\eqref{eq_global_upper_bound_boundary}
by taking the infinimum over all $\chi_h \in V_h$. \qed
\end{proof}

Our second estimate is explicit in terms of the two constants $\Cba$ and $\tCba$:

\begin{prop}[Second upper bound by the dual norm of the residual]
\label{proposition_global_upper_bound_fine}
We have
\[
\norm{u-u_h}_{1,k,\Omega} \leq (1+\tilde\theta_2(\Cba,\tCba)) \|\RR(u_h)\|_{-1,\Omega},
\]
where $\tilde\theta_2$ is defined in~\eqref{eq_Cupperf}.
\end{prop}

\begin{proof}
The proof combines~\eqref{eq_global_upper_bound_l2}, \eqref{eq_global_upper_bound_h1},
and~\eqref{eq_global_upper_bound_boundary}. \qed
\end{proof}

\subsection{Local lower bounds}

Define the local Sobolev space $H^1_\star(\oa)$ as
\begin{equation*}
%\label{eq_H1_loc}
H^1_\star(\oa) \eq
\left \{
\begin{array}{ll}
\{ v \in H^1(\oa) \ | \ \int_\oa v = 0 \}
&
\text{ when } \aaa \not \in \overline{\GD},
\\
\left \{
v \in H^1(\oa)\ |\  v = 0 \text{ on the part of } \GD \text{ where }
\psi_\aaa \neq 0
\right \}
&
\text{ when } \aaa \in \overline{\GD}.
\end{array}
\right .
\end{equation*}
We will now employ a localized dual norm of the residual from~\eqref{eq_res_def}
that is defined for each $\aaa \in \VV_h$ by
\begin{equation} \label{eq_res_loc}
\|\RR(u_h)\|_{-1,\aaa} \eq \sup_{v \in H^1_\star(\oa) \setminus \{0\}}
\frac{|\langle \RR(u_h),\psi_\aaa v\rangle|}{|v|_{1,\oa}},
\end{equation}
recalling that $\psi_\aaa$ is the hat function associated with the vertex $\aaa \in \VV_h$.

We record that there exists a (Poincar\'e or Poincar\'e--Friedrichs) constant $\CPa$
that depends only on the local shape-regularity parameter
$\kappa_{\TT_\aaa} \eq \min_{K \in \TT_\aaa} \kappa_K$ such that
\begin{equation}
\label{eq_poincare_oa}
\|v\|_{0,\oa} \leq \CPa h_\aaa |v|_{1,\oa} \qquad \forall v \in H^1_\star(\oa),
\end{equation}
where $h_\aaa \eq \sup_{\xx,\yy \in \oa} |\xx - \yy|$ is the diameter of the patch subdomain $\oa$.
When $\omega_\aaa$ is convex and $\aaa \notin \overline{\GD}$,
we have $\CPa = 1/\pi$, and we refer the reader to~\cite{Vees_Verf_Poin_stars_12} and the references therein
for a discussion on the value of the constant $\CPa$ in the case of
non-convex patches or patches around a Dirichlet boundary vertex.
As observed in~\cite{braess_pillwein_schoberl_2009a,ern_vohralik_2015a}, Leibniz's rule
in conjunction with~\eqref{eq_poincare_oa} shows that
\begin{equation}
\label{eq_stab_h1_hat}
|\psi_\aaa v|_{1,\oa} \leq \CcPFa |v|_{1,\oa} \qquad \forall v \in H^1_\star(\oa),
\end{equation}
where the constant
$\CcPFa \eq 1 + \CPa |\psi_\aaa|_{1,\infty,\oa} h_\aaa$
only depends on $\kappa_{\TT_\aaa}$. We will also employ the trace inequality
\begin{equation}
\label{eq_trace_oa}
\|v\|_{0,\partial \oa \cap \GA} \leq \Ctra
h_\aaa^{\frac12} |v|_{1,\oa} \qquad \forall v \in H^1_\star(\oa),
\end{equation}
where the constant $\Ctra$ again only depends on $\kappa_{\TT_\aaa}$.
%Trace inequality~\eqref{eq_trace_oa} is easily derived from~\eqref{eq_poincare_oa}
%and the trace inequality given in Section 4.2 of~\cite{cheddadi_fucik_prieto_vohralik_2009a}.
%
Finally, we introduce for each vertex $\aaa \in \VV_h$ the local norm
\begin{equation*}
\norm{v}_{1,k,\oa}^2
\eq
k^2 \|v\|_{0,\oa}^2 +
k \|v\|_{0,\partial \oa \cap \GA}^2 +
|v|_{1,\oa}^2 \quad \forall v \in H^1(\oa).
\end{equation*}

%We point out that for all $\aaa \in \VV_h$ such that $\aaa \notin \overline{\GD}$,
%we have $\psi_\aaa \in V_h$.

For all vertices $\aaa \in \VV_h$, if $v \in H^1_\GD(\Omega)$ with $\supp v \subset \oa$,
there exists a discrete function $Q_{hp}^\aaa (v) \in V_h$ with
$\supp (Q_{hp}^\aaa (v)) \subset \oa$ such that
\begin{equation}
\label{eq_quasi_interpolation}
\norm{v - Q_{hp}^\aaa (v)}_{1,k,\oa}
\leq
\Cqia \left (1 + \frac{kh_\aaa}{p} + \left (\frac{kh_\aaa}{p}\right )^2\right )^{\frac12}
|v|_{1,\oa},
\end{equation}
where the constant $\Cqia$ only depends on $\kappa_{\TT_\aaa}$.
We can use a quasi-interpolation operator $Q_{hp}^\aaa$ to achieve
\eqref{eq_quasi_interpolation}. The construction of such an operator
is presented in Theorem 3.3 of~\cite{melenk_2005a} in two space dimensions.
For three space dimensions, the corresponding operator is constructed
in~\cite{hiptmair_pechstein_2017a}.

\begin{lemm}[Local lower bounds of the error by the dual norm of the residual]
\label{lem_lower_bounds}
For all vertices $\aaa \in \VV_h$, we have
\[
\|\RR(u_h)\|_{-1,\aaa}
\leq
\Clowa \norm{u-u_h}_{1,k,\oa},
\]
where $\Clowa \eq \min(\Clowerhat,\Clowerqi)$ and
\begin{align*}
\Clowerhat
&\eq
\CcPFa \left (
1
+
\frac{\Ctra^2}{\CcPFa^2} kh_\aaa
+
\frac{\CPa^2}{\CcPFa^2} (kh_\aaa)^2
\right )^{\frac12},
\\
\Clowerqi
&\eq
\CcPFa \Cqia \left (
1
+
\frac{kh_\aaa}{p}
+
\left (
\frac{kh_\aaa}{p}
\right )^2
\right )^{\frac12}.
\end{align*}
\end{lemm}

\begin{proof}
Let $v \in H^1_\star(\oa)$. From~\eqref{eq_cont_en} and~\eqref{eq_res_def}, we observe that
\begin{equation} \label{eq_dev}
|\langle \RR(u_h), \psi_\aaa v\rangle|
=
|b(u-u_h,\psi_\aaa v)|
\leq
\norm{u-u_h}_{1,k,\oa}\norm{\psi_\aaa v}_{1,k,\oa}.
\end{equation}
Then, using~\eqref{eq_poincare_oa}, \eqref{eq_stab_h1_hat}, and~\eqref{eq_trace_oa}, we have
\begin{align*}
\norm{\psi_\aaa v}_{1,k,\oa}^2
&=
k^2 \|\psi_\aaa v\|_{0,\oa}^2
+
k\|\psi_\aaa v\|_{0,\partial \oa \cap \GA}^2
+
|\psi_\aaa v|_{1,\oa}^2
\\
&\leq
k^2 \|v\|_{0,\oa}^2
+
k\|v\|_{0,\partial \oa \cap \GA}^2
+
|\psi_\aaa v|_{1,\oa}^2
\\
&\leq
\left (
\CPa^2 k^2 h_\aaa ^2
+
\Ctra^2 k h_\aaa
+
\CcPFa^2
\right )
|v|_{1,\oa}^2
\\
&=
\Clowerhat^2 |v|_{1,\oa}^2.
\end{align*}

On the other hand, since $Q_{hp}^\aaa (\psi_\aaa v) \in V_h$ with $\supp (Q_{hp}^\aaa (\psi_\aaa v)) \subset \oa$,
using Galerkin's orthogonality~\eqref{eq_res_orth} in~\eqref{eq_dev}
and~\eqref{eq_quasi_interpolation} with~\eqref{eq_stab_h1_hat}, we have
\begin{align*}
|\langle \RR(u_h),\psi_\aaa v \rangle |
&\leq
\norm{u-u_h}_{1,k,\oa} \norm{\psi_\aaa v - Q^\aaa_{hp}(\psi_\aaa v)}_{1,k,\oa}
\\
&\leq
\Cqia \left (
1 + \frac{kh_\aaa}{p} + \left (\frac{kh_\aaa}{p} \right )^2
\right )^{\frac12} \norm{u-u_h}_{1,k,\oa} |\psi_\aaa v|_{1,\oa}
\\
&\leq
\CcPFa \Cqia \left (
1 + \frac{kh_\aaa}{p} + \left (\frac{kh_\aaa}{p} \right )^2
\right )^{\frac12} \norm{u-u_h}_{1,k,\oa} |v|_{1,\oa}
\\
&=
\Clowerqi \norm{u-u_h}_{1,k,\oa}|v|_{1,\oa}.
\end{align*}
The expected result follows by combining the two bounds together
with the definition~\eqref{eq_res_loc} of the localized dual norm. \qed
\end{proof}

\section{Bounds on the dual norm of the residual by equilibrated fluxes}
\label{sec:equilibrated_estimator}

In the previous section, we derived upper and lower bounds for the
finite element error based on dual norms of the residual. These dual
norms are not directly computable, as they are defined using a supremum
over infinite-dimensional spaces. In this section, we use the technique of equilibrated
flux construction to achieve guaranteed computable upper and lower bounds
of these dual norms.

\subsection{Global upper bound}

Recall the definitions~\eqref{eq_osc} of the data oscillation,
\eqref{eq_definition_eta} of the error estimator $\eta$,
\eqref{eq_res_def} and~\eqref{eq_res_norm_def} of
the residual and its dual norm, and finally Definition~\ref{def_sigma_a}
of the equilibrated flux $\sig_h$.

\begin{prop}[Upper bound on the dual norm of the residual]
\label{proposition_global_upper_bound}
The following holds true:
\begin{equation}
\label{eq_global_upper_bound}
\|\RR(u_h)\|_{-1,\Omega}
\leq
\left (
\sum_{K \in \mathcal T_h}
\left (
\eta_K + \osc_K(f,g)
\right )^2
\right )^{\frac12}.
\end{equation}
\end{prop}

\begin{proof}
We first observe that since the hat functions form a partition of unity, we have
\begin{equation*}
\sum_{\aaa \in \VV_h} \psi_\aaa(\xx) = 1 \qquad \forall \xx \in \overline \Omega.
\end{equation*}
The summation over all vertices $\aaa \in \VV_h$ in~\eqref{eq_definition_flux_glob}
together with the divergence and normal trace constraints in~\eqref{eq_definition_flux} lead to
\begin{equation*}
\div \ssig_h = \pi_h^p (f) + k^2 u_h \text{ in } \Omega, \quad
\ssig_h {\cdot} \nn = -\left (\tpi_h^p(g) + ik u_h\right ) \text{ on } \GA.
\end{equation*}
Then, if $v \in H^1_\GD(\Omega)$, we have
\begin{align*}
b(u_h,v)
&=
-k^2(u_h,v) -ik(u_h,v)_\GA + (\grad u_h,\grad v)
\\
&=
(\pi_h^p (f),v) + (\tpi_h^p(g),v)_\GA -
(\pi_h^p (f) + k^2 u_h,v) -(\tpi_h^p(g) + iku_h,v)_\GA + (\grad u_h,\grad v)
\\
&=
(\pi_h^p (f),v) + (\tpi_h^p(g),v)_\GA -
(\div \ssig_h,v) + (\sig_h {\cdot} \nn,v)_{\GA} + (\grad u_h,\grad v)
\\
&=
(\pi_h^p (f),v) + (\tpi_h^p(g),v)_\GA + (\ssig_h + \grad u_h,\grad v).
\end{align*}
It follows that
\begin{align*}
\langle \RR(u_h),v\rangle
&=
(f-\pi_h^p (f),v) + (g - \tpi_h^p(g),v)_\GA + (\pi_h^p (f),v) + (\tpi_h^p(g),v)_\GA - b(u_h,v)
\\
&=
(f-\pi_h^p (f),v) + (g-\tpi_h^p(g),v)_\GA - (\ssig_h + \grad u_h,\grad v).
\end{align*}
Since the restriction of $\pi_h^p(v)$ to each mesh face $F \subset \GA$
% of the mesh
belongs to $\PP_p(F)$, and $v - \pi_h^p(v)$ has zero mean-value on each
mesh cell $K$, we have
\begin{align*}
&
|(f-\pi_h^p (f),v)_K + (g-\tpi_h^p(g),v)_{\partial K \cap \partial \GA}| \\
& =
|(f-\pi_h^p (f),v-\pi_h^p(v))_K +
(g-\tpi_h^p(g),v-\pi_h^p(v))_{\partial K \cap \GA}|
\\
& \leq
\|f-\pi_h^p (f)\|_{0,K}\|v-\pi_h^p(v)\|_{0,K} + \|g -\tpi_h^p(g)\|_{0,\partial K \cap \GA}
\|v - \pi_h^p(v)\|_{0,\partial K \cap \GA}
\\
& \leq
\osc_K(f,g)|v|_{1,K},
\end{align*}
where we used~\eqref{eq_upper_bound_oscK}.
The Cauchy--Schwarz inequality now implies~\eqref{eq_global_upper_bound}. \qed
\end{proof}

At this point, the upper bound~\eqref{eq_est_up} of Theorem~\ref{theorem_upper_bounds} follows
from Propositions~\ref{proposition_upper_bound_full_norm_coarse},
\ref{proposition_global_upper_bound_fine}, and~\ref{proposition_global_upper_bound}
by setting $\Cupper \eq \min \left (\sqrt{2}+\tilde\theta_1(\Cba),
1+\tilde\theta_2(\Cba,\tCba)\right)$
and bounding~\eqref{eq_global_upper_bound} further by the triangle inequality.

\subsection{Local lower bound}

We first introduce a residual with projected source terms
\begin{equation} \label{eq_res_proj}
\langle \RR_h(u_h),v\rangle
\eq
(\pi_h^p (f),v) + (\tpi_h^p(g),v)_\GA - b(u_h,v)
\qquad \forall v \in H^1_\GD(\Omega),
\end{equation}
where the original right-hand sides $f$ and $g$ have been respectively replaced
by their elementwise and facewise $L^2$ projections. We employ the same notation
for the dual norms of $\RR_h(u_h)$ as for those of $\RR(u_h)$.

\begin{lemm}[Data oscillations] \label{lem_data_osc}
We have
\[
\|\RR_h(u_h)\|_{-1,\aaa}
\leq
\|\RR(u_h)\|_{-1,\aaa} + C(\kappa_{\TT_\aaa})\osc_{\TT_\aaa}(f,g) \qquad \forall \aaa \in \VV_h,
\]
where the constant $C(\kappa_{\TT_\aaa})$ only depends on the shape-regularity paremeter
of the elements in the patch $\TT_\aaa$.
\end{lemm}

\begin{proof}
Fix a vertex $\aaa \in \VV_h$. For all $v \in H^1_\star(\oa)$, we have
\begin{equation*}
\langle \RR_h(u_h), \psi_\aaa v \rangle
=
\langle \RR(u_h), \psi_\aaa v \rangle - (f - \pi_h^p (f),\psi_\aaa v) -
(g - \tpi_h^p(g),\psi_\aaa v)_{\GA}.
\end{equation*}
Consequently, we infer that
\begin{align*}
&|\langle \RR_h(u_h),\psi_\aaa v \rangle|
\\
&\leq
|\langle \RR(u_h),\psi_\aaa v \rangle| +
\|f - \pi_h^p (f)\|_{0,\oa}\|\psi_\aaa v\|_{0,\oa} +
\|g -\tpi_h^p(g)\|_{0,\partial \oa \cap \GA} \|\psi_\aaa v\|_{0,\partial \oa \cap \GA}
\\
&\leq
|\langle \RR(u_h),\psi_\aaa v \rangle| +
\|f - \pi_h^p (f)\|_{0,\oa}\|v\|_{0,\oa} +
\|g -\tpi_h^p(g)\|_{0,\partial \oa \cap \GA} \|v\|_{0,\partial \oa \cap \GA},
\end{align*}
and we conclude using a similar estimate as in~\eqref{eq_upper_bound_oscK} patchwise. \qed
\end{proof}

In the next lemma, we observe following~\cite[Theorem~7]{braess_pillwein_schoberl_2009a},
\cite[Remark~3.15]{ern_vohralik_2015a}, or~\cite[Corollary~3.6]{ErnVo:20}
that the dual norms of the residual can be characterized using ``continuous versions'' of the
minimization problems defining the equilibrated fluxes $\ssig_h^\aaa$. Recall that
$\Gamma_{\aaa}$ is the boundary $\partial \omega_{\aaa}$ without those faces of
$\Gamma_{\mathrm{D}}$ sharing the vertex $\aaa$.

\begin{lemm}[Dual characterization]
\label{lemma_dual_characterization}
We have
\begin{equation}
\label{eq_dual_characterization}
\|\RR_h(u_h)\|_{-1,\aaa} =
\min_{\substack{
\ttau^\aaa \in \HH(\ddiv,\oa)
\\
\div \ttau^\aaa = d_\aaa \; \text{\em in } \oa
\\
\ttau^\aaa {\cdot} \nn = b_\aaa \; \text{\em on } \Gamma_\aaa}}
\|\ttau^\aaa + \psi_\aaa \grad u_h\|_{0,\oa},
\end{equation}
where $b_\aaa$ and $d_\aaa$ are defined in~\eqref{eq_db}.
\end{lemm}

\begin{proof}
We introduce $r_\aaa$ as the unique element of $H^1_\star(\oa)$ such that
\begin{equation*}
(\grad r_\aaa,\grad v)_{\oa} = \langle \RR_h(u_h),\psi_\aaa v\rangle
\qquad \forall v \in H^1_\star(\oa).
\end{equation*}
Let $v \in H^1_\star(\oa)$.
The definitions~\eqref{eq_db} and~\eqref{eq_res_proj} together with the identity
$\grad(\psi_\aaa v) = \psi_\aaa \grad v + v \grad \psi_\aaa$
show that
\begin{eqnarray*}
\langle \RR_h(u_h),\psi_\aaa v \rangle
&=&
(d_\aaa,v)_{\oa} - (b_\aaa,v)_{\partial \oa \cap \GA} - (\psi_\aaa \grad u_h,\grad v)_{\oa},
\end{eqnarray*}
and therefore
\begin{equation*}
(\grad r_\aaa,\grad v)_{\oa}
=
(d_\aaa,v)_{\oa} - (b_\aaa,v)_{\partial \oa \cap \GA} - (\psi_\aaa \grad u_h,\grad v)_{\oa}
\qquad \forall v \in H^1_\star(\oa).
\end{equation*}
Thus, $\ssig^\aaa \eq -\left (\grad r_\aaa + \psi_\aaa \grad u_h\right ) \in \HH(\ddiv,\oa)$
satisfies $\div \ssig^\aaa = d_\aaa$ and $\ssig^\aaa {\cdot} \nn = b_\aaa$ on $\Gamma_\aaa$
as well as
\begin{equation*}
\left \{
\begin{array}{rcll}
(\ssig^\aaa,\vv)_\oa - (r_\aaa,\div \vv)_\oa &=& -(\psi_\aaa \grad u_h,\vv)_\oa
&
\forall \vv \in \HH_\GA(\ddiv,\oa),
\\
(\div \ssig^\aaa,v)_\oa &=& (d_\aaa,v)_\oa
&
\forall v \in H^1_\star(\oa),
\end{array}
\right .
\end{equation*}
so that $\ssig^\aaa$ is the unique minimizer in the right-hand side of~\eqref{eq_dual_characterization}.
Then, the conclusion follows since we have
%\begin{equation*}
$\|\RR_h(u_h)\|_{-1,\aaa}
=
|r_\aaa|_{1,\oa}
=
\|\ssig^\aaa + \psi_\aaa \grad u_h\|_{0,\oa}$. \qed
%\end{equation*}
\end{proof}

The following key estimate directly follows from \cite[Theorem~7]{braess_pillwein_schoberl_2009a}
in two space dimensions and~\cite[Corollaries 3.3 and 3.8]{ErnVo:20} in three space
dimensions. %Following~\cite[Lemma~3.23]{ern_vohralik_2015a}, one can in addition obtain a computable upper bound on the constant $\Csta$.

\begin{lemm}[Stability of discrete minimization]
\label{lemma_local_estimations}
For every vertex $\aaa \in \VV_h$, we have
\[
\|\ssig^\aaa_h + \psi_\aaa \grad u_h\|_{0,\oa}
\leq
\Csta
\min_{\substack{
\ttau^\aaa \in \HH(\ddiv,\oa)
\\
\div \ttau^\aaa = d_\aaa \; \text{\em in } \oa
\\
\ttau^\aaa {\cdot} \nn = b_\aaa \; \text{\em on } \Gamma_\aaa
}}
\|\ttau^\aaa + \psi_\aaa \grad u_h\|_{0,\oa},
\]
where $\Csta$ only depends on $\kappa_{\TT_\aaa}$.
\end{lemm}

The following proposition gathers intermediate results established throughout
Sections~\ref{sec:abstract_error_estimates} and~\ref{sec:equilibrated_estimator}
and proves the local lower bound~\eqref{eq_est_low_loc} of Theorem~\ref{theorem_lower_bounds}.
We denote by $\VV_K$ the set of vertices of the mesh cell $K \in \mathcal T_h$.

\begin{prop}[Local lower bound]
\label{prop_local_lower_bounds}
We have
\begin{equation*}
%\label{local_lower_bound_full}
\eta_K \leq \ClowK \norm{u-u_h}_{1,k,\oK} + C(\kappa_{\TT_K}) \osc_{\TT_K}(f,g)
\qquad \forall K \in \TT_h,
\end{equation*}
where
\begin{equation*}
\ClowK \eq (d+1) \max_{\aaa \in \VV_K} \Csta \Clowa.
\end{equation*}
\end{prop}

\begin{proof}
Combining Lemmas~\ref{lem_lower_bounds}, \ref{lem_data_osc}, \ref{lemma_dual_characterization},
and~\ref{lemma_local_estimations}, we infer that
\begin{equation} \label{eq_loc_eff_oma}
    \|\sig_h^\aaa + \psi_\aaa \grad u_h\|_{0,\oa}
    \leq
    \Csta \Clowa \norm{u-u_h}_{1,k,\oa} + \Csta C(\kappa_{\TT_\aaa}) \osc_{\TT_\aaa}(f,g).
\end{equation}
Since
\begin{equation*}
\eta_K \leq \sum_{\aaa \in \VV_K} \|\sig^\aaa_h + \psi_\aaa \grad u_h\|_{0,\oa}
\end{equation*}
and as each $K \in \TT_h$ has $(d+1)$ vertices and the neighboring elements have a similar diameter,
the assertion follows. \qed
\end{proof}

Finally, the following estimate is obtained by summation of the local
lower bounds~\eqref{eq_loc_eff_oma} established in the proof of
Proposition~\ref{prop_local_lower_bounds} and proves the global lower bound~\eqref{eq_est_low_glob}
of Theorem~\ref{theorem_lower_bounds}:

\begin{prop}[Global lower bound]
\label{prop_global_lower_bound}
We have
\begin{equation*}
%\label{global_lower_bound_full}
\eta \leq \Clow \norm{u-u_h}_{1,k,\Omega} + C(\kappa) \osc(f,g),
\end{equation*}
where
\begin{equation*}
\Clow \eq (d+1) \max_{\aaa \in \VV_h} \Csta \Clowa.
\end{equation*}
\end{prop}

\begin{rema}[Computable lower bound]
Lemma~\ref{lem_lower_bounds} shows that $\Clowa \leq \Clowerhat$, where the constant
$\Clowerhat$ is fully computable. Since we can compute
an upper bound for $\Csta$, see \cite[Lemma 3.23]{ern_vohralik_2015a},
we are able to provide a fully computable lower bound for the error.
\end{rema}

\section{Preasymptotic error estimates}
\label{sec_pre_asymptotic}

In this section, we establish the results stated in Theorem~\ref{theorem_cba}.
We only detail the proof for Cases~1a) and 1b), as the other
cases (related to the interior problem) easily follow from
standard properties of spectral decomposition and the regularity shift
\begin{equation*}
|\phi|_{2,\Omega} \leq
\|\Delta \phi\|_{0,\Omega},
\end{equation*}
for all $\phi \in H^1_0(\Omega)$ with $\Delta \phi \in L^2(\Omega)$,
which is valid when $\Omega$ a convex polytope, see~\cite{grisvard_1992a}
for instance. The aim of this section is thus to prove the bounds~\eqref{eq_cba_coarse}
and~\eqref{eq_cba_fine} of Theorem~\ref{theorem_cba}.

\subsection{A stability estimate in $L^2(\Omega)$}

Under the assumptions of Case~1a) or 1b), the cornerstone
of the analysis is a stability estimate
that we establish hereafter. Similar upper bounds are available in~\cite{hetmaniuk_2007a,melenk_1995a}
for the setting considered here, and we refer the reader to
\cite{chandlerwilde_spence_gibbs_smyshlyaev_2017a,chaumontfrelet_nicaise_2018a,spence_2014a}
for more complex geometries. However, these estimates are not as sharp as possible
and/or the constant $\Cstabx$ is not computable, since they have been derived having a
priori error estimation (or simply, stability analysis) in mind. We provide here sharper,
fully-computable estimates.

\begin{lemm}[$L^2(\Omega)$ stability estimate]
%\label{lemma_multiplier}
Let $\Omega = \Omega_0 \setminus \overline{D}$, $\GA = \partial \Omega_0$,
and $\GD = \partial D$, where $\Omega_0,D \subset \mathbb R^d$ are two
open, bounded, and connected sets such that $\overline D$ is a proper subset of $\Omega_0$.
Assume that the subset $\mathcal{O}_{\GD,\GA}$ defined in~\eqref{eq_def_Omega_x0} is nonempty.
For all $\phi \in L^2(\Omega)$, let the (adjoint) solution
$u_\phi^\star\in H^1_\GD(\Omega)$ solve \eqref{eq_helmholtz_weak_adjoint}, \ie,
$b(w,u_\phi^\star) = (w,\phi)$ for all $w \in H^1_\GD(\Omega)$. Then, we have
\begin{equation}
\label{eq_stability_L2}
k^2 \|u_\phi^\star\|_{0,\Omega} \leq
\left ((d-1) + \Cstab k h_\Omega \right )
\|\phi\|_{0,\Omega},
\end{equation}
with $\Cstab$ is defined in~\eqref{eq_Cstab}.
\end{lemm}

\begin{proof}
Let $\xx_0 \in \mathcal{O}_{\GD,\GA}$, \ie, we have
\begin{equation*}
(\xx - \xx_0) {\cdot} \nn \leq 0 \; \forall \xx \in \GD \quad
(\xx - \xx_0) {\cdot} \nn > 0 \; \forall \xx \in \GA.
\end{equation*}
Let us set $\yy(\xx) \eq \xx - \xx_0$, $\yy_{\ttau} \eq \yy - (\yy {\cdot} \nn)\nn$,
$\grad_{\ttau} u_\phi^\star \eq \grad u_\phi^\star - (\grad u_\phi^\star {\cdot} \nn) \nn$.
In strong form, the (adjoint) solution $u_\phi^\star$ is such that
\begin{equation*}
\left \{
\begin{array}{rcll}
-k^2 u_\phi^\star - \Delta u_\phi^\star &=& \phi & \text{ in } \Omega,
\\
u_\phi^\star &=& 0 & \text{ on } \GD,
\\
\grad u_\phi^\star{\cdot}\nn + iku_\phi^\star &=& 0 & \text{ on } \GA.
\end{array}
\right .
\end{equation*}
The key idea of the proof is to multiply the first equation by the test function
$\overline{w} \eq \yy {\cdot} \grad \overline{u_\phi^\star}$ (here $\overline{\cdot}$
denotes the complex conjugate) and employ integration by parts techniques. Let us point
out that since $\Omega$ is a polytopal domain and $\GD$ and $\GA$ are well separated,
we have $u_\phi^\star \in H^{\frac32+\varepsilon}(\Omega)$ for some $\varepsilon > 0$,
and therefore $u_\phi^\star$ and $w$ are sufficiently smooth to allow the operations
performed hereafter (see in particular Section 3.3 of~\cite{hetmaniuk_2007a}).

First, applying Green's formula, we have
\begin{equation*}
2\Re \left \{
-k^2 \int_\Omega u_\phi^\star\left (\yy {\cdot} \grad \overline{u_\phi^\star}\right )
\right \}
=
-k^2\int_\Omega \yy {\cdot} \grad |u_\phi^\star|^2
=
dk^2 \|u_\phi^\star\|_{0,\Omega}^2 - k^2 \int_{\partial \Omega} |u_\phi^\star|^2 \yy {\cdot} \nn,
\end{equation*}
where we used the fact that $\div \yy = d$. Recalling Rellich's identity
\begin{equation*}
2\Re \left \{\int_\Omega
  \grad(\yy {\cdot} \grad \overline{u_\phi^\star}) {\cdot} \grad u_\phi^\star
\right \}
=
-(d-2) |u_\phi^\star|_{1,\Omega}^2 +
\int_{\partial \Omega} |\grad u_\phi^\star|^2 \yy {\cdot} \nn,
\end{equation*}
which can be obtained by integration by parts
(see \cite[proof of Proposition 3.3]{hetmaniuk_2007a} for instance), we have
\begin{align*}
& 2\Re \left \{-\int_\Omega \Delta u_\phi^\star \yy {\cdot} \grad \overline{u_\phi^\star} \right \}\\
&=
2 \Re \int_\Omega \grad u_\phi^\star {\cdot} \grad (\yy {\cdot} \grad \overline{u_\phi^\star})
- 2\Re \int_\Omega \left (\grad u_\phi^\star {\cdot} \nn\right ) \yy {\cdot} \grad \overline{u_\phi^\star}
\\
&=
-(d-2)|u_\phi^\star|_{1,\Omega}^2 + \int_{\partial \Omega} |\grad u_\phi^\star|^2 \yy {\cdot} \nn
- 2\Re \int_{\partial \Omega} \left (\grad u_\phi^\star {\cdot} \nn\right ) \yy {\cdot} \grad \overline{u_\phi^\star}.
\end{align*}
It follows that
\begin{align*}
2 \Re \int_\Omega \phi \yy {\cdot} \grad \overline{u_\phi^\star}
={}&
2 \Re \int_\Omega \left (-k^2 u_\phi^\star - \Delta u_\phi^\star\right ) \yy {\cdot} \grad \overline{u_\phi^\star}
\\
={}&
dk^2\|u_\phi^\star\|_{0,\Omega}^2 - (d-2)|u_\phi^\star|_{1,\Omega}^2
\\
-&
k^2 \int_{\partial \Omega} |u_\phi^\star|^2 \yy {\cdot} \nn +
\int_{\partial \Omega} |\grad u_\phi^\star|^2 \yy {\cdot} \nn -
2 \Re \int_{\partial \Omega} \left (\grad u_\phi^\star {\cdot} \nn\right )
\yy {\cdot} \grad \overline{u_\phi^\star}.
\end{align*}
We now simplify the boundary term as
\begin{align*}
\mathscr B
\eq &
\int_{\partial \Omega} |\grad u_\phi^\star|^2 \yy {\cdot} \nn -
2 \Re \int_{\partial \Omega} (\grad u_\phi^\star {\cdot} \nn) \yy {\cdot} \grad \overline{u_\phi^\star}
\\
= &
\int_{\partial \Omega}|\grad_{\ttau} u_\phi^\star|^2 \yy {\cdot} \nn -
\int_{\partial \Omega}|\grad u_\phi^\star {\cdot} \nn|^2 \yy {\cdot} \nn -
2\Re \int_{\partial \Omega} (\grad u_\phi^\star {\cdot} \nn) \yy_{\ttau} {\cdot} \grad_{\ttau} \overline{u_\phi^\star}
\\
= &
\int_\GA |\grad_{\ttau} u_\phi^\star|^2 \yy {\cdot} \nn - \int_\GD |\grad u_\phi^\star {\cdot} \nn|^2 \yy {\cdot} \nn
-
\int_\GA |iku_\phi^\star|^2 \yy {\cdot} \nn + 2\Re
\left \{
\int_\GA iku_\phi^\star \yy_{\ttau} {\cdot} \grad_{\ttau} \overline{u_\phi^\star}
\right \}
\\
\geq  &
\int_\GA |\grad_{\ttau} u_\phi^\star|^2 \yy {\cdot} \nn
-
k^2 \int_\GA |u_\phi^\star|^2 \yy {\cdot} \nn +
2\Re \left \{
\int_\GA iku_\phi^\star \yy_{\ttau} {\cdot} \grad_{\ttau} \overline{u_\phi^\star}
\right \},
\end{align*}
where we used that $\grad_{\ttau} u_\phi^\star = \mathbf 0$ on $\GD$,
$\grad u_\phi^\star {\cdot} \nn = -iku_\phi^\star$ on $\GA$,
and $\yy {\cdot} \nn \leq 0$ on $\GD$. We now have
\begin{align*}
2 \Re \int_\Omega \phi \yy {\cdot} \grad \overline{u_\phi^\star}
\geq {}&
d k^2 \|u_\phi^\star\|_{0,\Omega}
+
\int_\GA |\grad_{\ttau} u_\phi^\star|^2 \yy {\cdot} \nn
-
(d-2)|u_\phi^\star|_{1,\Omega}^2
\\
&
- 2k^2 \int_\GA |u_\phi^\star|^2 \yy {\cdot} \nn +
2 \Re \int_\GA iku_\phi^\star \yy_{\ttau} {\cdot} \grad_{\ttau} \overline{u_\phi^\star},
\end{align*}
that we can rewrite as
\begin{align*}
& dk^2 \|u_\phi^\star\|_{0,\Omega}^2
+
\int_{\GA} |\grad_{\ttau} u_\phi^\star|^2 \yy {\cdot} \nn
\\
& \leq
2 \Re \int_\Omega \phi \yy {\cdot} \grad \overline{u_\phi^\star}
+ (d-2)|u_\phi^\star|_{1,\Omega}^2 + 2k^2 \int_\GA |u_\phi^\star|^2 \yy {\cdot} \nn +
2 \Re ik \left \{\int_\GA u_\phi^\star \yy_{\ttau} {\cdot} \grad_{\ttau} \overline{u_\phi^\star} \right \}
\\
& \leq
2\Re (\phi,\yy {\cdot} \grad u_\phi^\star)
+
(d-2)|u_\phi^\star|_{1,\Omega}^2
+
2k^2 \int_{\GA} |u_\phi^\star|^2 \yy {\cdot} \nn +
2k
\int_{\GA} |u_\phi^\star| |\grad_{\ttau} u_\phi^\star| |\yy_{\ttau}|.
\end{align*}
At this point, since $\yy {\cdot} \nn > 0$ on $\GA$, we use the inequality
\begin{align*}
2k |u_\phi^\star| |\grad_{\ttau} u_\phi^\star| |\yy_{\ttau}|
&=
2k |u_\phi^\star| |\grad_{\ttau} u_\phi^\star| |\yy \times \nn|
\\
&=
2\left (k\frac{|\yy \times \nn|}{\sqrt{\yy {\cdot} \nn}} |u_\phi^\star|\right )
\left (|\grad_{\ttau} u_\phi^\star|\sqrt{\yy {\cdot} \nn}\right )
\\
&\leq
k^2\frac{|\yy \times \nn |^2}{\yy {\cdot} \nn} |u_\phi^\star|^2 +
|\grad_{\ttau} u_\phi^\star|^2\yy {\cdot} \nn,
\end{align*}
and infer that
\begin{align*}
dk^2 \|u_\phi^\star\|_{0,\Omega}^2
&\leq
2 \Re (\phi,\yy {\cdot} \grad u_\phi^\star) + (d-2)|u_\phi^\star|_{1,\Omega}^2 +
k^2 \int_{\GA}
\left (
2\yy {\cdot} \nn + \frac{|\yy \times \nn|^2}{\yy {\cdot} \nn}
\right ) |u_\phi^\star|^2
\\
&\leq
2 \lambda(\xx_0)\|\phi\|_{0,\Omega} |u_\phi^\star|_{1,\Omega} + (d-2)|u_\phi^\star|_{1,\Omega}^2
+
k^2 \mu(\xx_0) \|u_\phi^\star\|_{0,\GA}^2
\\
&\leq
\lambda(\xx_0)^2 \|\phi\|_{0,\Omega}^2 + (d-1)|u_\phi^\star|_{1,\Omega}^2
+
k^2 \mu(\xx_0) \|u_\phi^\star\|_{0,\GA}^2,
\end{align*}
where, recalling that $\yy=\xx-\xx_0$, we have set
\begin{equation*}
\lambda(\xx_0) \eq \sup_{\xx \in \Omega} |\xx-\xx_0|, \quad
\mu(\xx_0) \eq \sup_{\xx \in \GA}
\left (
2 (\xx - \xx_0) {\cdot} \nn + \frac{|(\xx-\xx_0) \times \nn|^2}{(\xx-\xx_0) {\cdot} \nn}
\right ).
\end{equation*}
Moreover, we have
\begin{equation*}
|u_\phi^\star|_{1,\Omega}^2 = \Re b(u_\phi^\star,u_\phi^\star) + k^2\|u_\phi^\star\|_{0,\Omega}^2 \leq
\|\phi\|_{0,\Omega}\|u_\phi^\star\|_{0,\Omega} + k^2\|u_\phi^\star\|_{0,\Omega}^2
\end{equation*}
and
\begin{equation*}
k\|u_\phi^\star\|_{0,\GA}^2
=
-\Im b(u_\phi^\star,u_\phi^\star)
\leq
\|\phi\|_{0,\Omega}\|u_\phi^\star\|_{0,\Omega}.
\end{equation*}
It follows that
\begin{equation*}
k^2 \|u_\phi^\star\|_{0,\Omega}^2 \leq \lambda(\xx_0)^2 \|\phi\|_{0,\Omega}^2 +
((d-1)+k \mu(\xx_0))\|\phi\|_{0,\Omega} \|u_\phi^\star\|_{0,\Omega}.
\end{equation*}
Lemma~\ref{lemma_polynomial} with
$a = k^2$, $b = (d-1)+k\mu(\xx_0)\|\phi\|_{0,\Omega}$, and
$c = \lambda(\xx_0)^2 \|\phi\|_{0,\Omega}^2$ shows that
\begin{eqnarray*}
k^2 \|u_\phi^\star\|_{0,\Omega}
&\leq&
\left (
\frac{(d-1)+k\mu(\xx_0)}{2} +
\sqrt{\left (\frac{(d-1)+k\mu(\xx_0)}{2}\right )^2 + k^2\lambda(\xx_0)^2}
\right ) \|\phi\|_{0,\Omega}
\\
&\leq&
\left ((d-1) + k\left (\lambda(\xx_0) + \mu(\xx_0)\right )
\right ) \|\phi\|_{0,\Omega}.
\end{eqnarray*}
Then~\eqref{eq_stability_L2} follows since
$\Cstab = \inf_{\xx_0\in\mathcal{O}_{\GD,\GA}}(1/h_\Omega) (\lambda(\xx_0) + \mu(\xx_0))$.\qed
\end{proof}

\subsection{A bound on $\Cba$ for scattering by a non-trapping obstacle (Case 1a))}
\label{sec_cba_coarse}

We prove the bound~\eqref{eq_cba_coarse} from Theorem~\ref{theorem_cba}.
Observe first that
\begin{equation}
\label{tmp_cba_coarse}
\Cba = k\sup_{\phi \in L^2(\Omega) \setminus \{0\}}
\inf_{v_h \in V_h} \frac{|u_\phi^\star - v_h|_{1,\Omega}}{\|\phi\|_{0,\Omega}}
\leq
k\sup_{\phi \in L^2(\Omega) \setminus \{0\}} \frac{|u_\phi^\star|_{1,\Omega}}{\|\phi\|_{0,\Omega}}.
\end{equation}
Then, we consider an arbitrary $\phi \in L^2(\Omega)$ and the associated solution
$u_\phi^\star \in H^1_\GD(\Omega)$. Selecting the test function $v = u_\phi^\star$
in~\eqref{eq_helmholtz_weak_adjoint} and considering the real part of the equality, we see that
\begin{equation*}
|u_\phi^\star|_{1,\Omega}^2 = \Re (\phi,u_\phi^\star) + k^2\|u_\phi^\star\|_{0,\Omega}^2.
\end{equation*}
Hence, employing~\eqref{eq_stability_L2}, we have
\begin{eqnarray*}
k^2 |u_\phi^\star|_{1,\Omega}^2
&\leq&
\frac{k^2}{2 \varepsilon}\|\phi\|_{0,\Omega}^2 + \frac{\varepsilon}{2k^2}k^4 \|u_\phi^\star\|_{0,\Omega}^2
+
k^4 \|u_\phi^\star\|_{0,\Omega}^2
\\
&\leq&
%\left (F(\varepsilon) + \left (\frac{d-1}{2} + \Cstabx k h_\Omega\right )^2\right )
\left (F(\varepsilon) + \left ((d-1) + \Cstab k h_\Omega\right )^2\right )
\|\phi\|_{0,\Omega}^2
\end{eqnarray*}
for all $\varepsilon > 0$ with
\begin{equation*}
F(\varepsilon)
=
\frac{k^2}{2\varepsilon} + \frac{\varepsilon}{2k^{2}}
\left ((d-1) + \Cstab k h_\Omega\right )^2.
\end{equation*}
Then, considering the equation $F'(\varepsilon) = 0$, we see that the minimum
of $F$ is achieved for
\begin{equation*}
%\varepsilon_\star = \frac{k^2}{ \frac{d-1}{2} + \Cstabx k h_\Omega },
\varepsilon_\star = \frac{k^2}{ (d-1) + \Cstab k h_\Omega },
\end{equation*}
and
\begin{eqnarray*}
%F(\varepsilon_\star) = \left (\frac{d-1}{2} + \Cstabx k h_\Omega \right ).
F(\varepsilon_\star) = \left ((d-1) + \Cstab k h_\Omega \right ).
\end{eqnarray*}
We thus obtain that
\begin{equation*}
k^2 |u_\phi^\star|_{1,\Omega}^2
\leq
\left \{
%\left (\frac{d-1}{2} + \Cstabx k h_\Omega\right )
\left ((d-1) + \Cstab k h_\Omega\right )
+
%\left (\frac{d-1}{2} + \Cstabx k h_\Omega\right )^2
\left ((d-1) + \Cstabx k h_\Omega\right )^2
\right \}
\|\phi\|_{0,\Omega}^2,
\end{equation*}
for all $\phi \in L^2(\Omega)$, and~\eqref{eq_cba_coarse} follows
from~\eqref{tmp_cba_coarse}.

\subsection{A bound on $\Cba$ for wave propagation in free space (Case 1b))}
\label{sec_cba_fine}

We  prove here the estimate~\eqref{eq_cba_fine} of Theorem~\ref{theorem_cba}.
We thus consider the case where $\GD = \emptyset$ and $\Omega$
is convex. In this case (see~\cite{chaumontfrelet_nicaise_tomezyk_2018a}),
$u_\phi^\star \in H^2(\Omega)$.
Using~\eqref{eq_approx}, we observe that
\begin{equation} \begin{split}
\label{tmp_cba_fine}
\Cba = {} & k\sup_{\phi \in L^2(\Omega) \setminus \{0\}}
\inf_{v_h \in V_h} \frac{|u_\phi^{\star} - v_h|_{1,\Omega}}{\|\phi\|_{0,\Omega}}
\leq
k \Ci \frac{h}{p^\beta} \sup_{\phi \in L^2(\Omega) \setminus \{0\}}
\frac{|u_\phi^{\star}|_{2,\Omega}}{\|\phi\|_{0,\Omega}}.
\end{split} \end{equation}
We can view $u_\phi^\star$ as the unique solution to
\begin{equation*}
(\grad v,\grad u_\phi^\star) - ik(v,u_\phi^\star)_\GA = (v,\widetilde f)
\quad \forall v \in H^1_\GD(\Omega),
\end{equation*}
with $\widetilde f = \phi + k^2 u_\phi^\star$. Then, Theorem 3.2 of
\cite{chaumontfrelet_nicaise_tomezyk_2018a} states
that $u_\phi^\star \in H^2(\Omega)$ with
\begin{equation*}
|u_\phi^\star|_{2,\Omega} \leq \|\widetilde f\|_{0,\Omega}.
\end{equation*}
Then, \eqref{eq_stability_L2} shows that
\begin{equation*}
|u_\phi^\star|_{2,\Omega}
\leq
%\left (1 + \frac{d-1}{2} + \Cstabx k h_\Omega \right ) \|\phi\|_{0,\Omega},
\left (d + \Cstab k h_\Omega \right ) \|\phi\|_{0,\Omega},
\end{equation*}
and~\eqref{eq_cba_fine} follows from~\eqref{tmp_cba_fine}.

\section{Numerical experiments}
\label{sec:numerical_experiments}

We present here two numerical experiments illustrating
Theorems~\ref{theorem_upper_bounds}, \ref{theorem_lower_bounds}, and~\ref{theorem_cba}.

\subsection{Plane wave in free space} \label{sec_PW}

We consider problem~\eqref{eq_helmholtz_strong} in the
square $\Omega = (-1,1)^2$ with $\GD = \emptyset$ and
$\GA = \partial \Omega$.
We fix the angle $\anglepw \eq \pi/3$, set $\dd \eq (\cos \anglepw,\sin \anglepw)$,
and define the plane wave $\xi_\anglepw(\xx) \eq e^{ik\dd {\cdot} \xx}$, $\xx \in \overline{\Omega}$.
We remark that $\xi_\anglepw$ is a homogeneous solution.
The problem is thus to find $u$ such that
\begin{equation}
\label{eq_helmholtz_plane_wave}
\left \{
\begin{array}{rcll}
-k^2 u - \Delta u &=& 0 & \text{ in } \Omega,
\\
\grad u {\cdot} \nn - iku &=& g & \text{ on } \partial \Omega,
\end{array}
\right .
\end{equation}
where $g = \grad \xi_\anglepw {\cdot} \nn - ik\xi_\anglepw$
on $\partial \Omega$. The unique solution is the plane wave $u = \xi_\anglepw$.

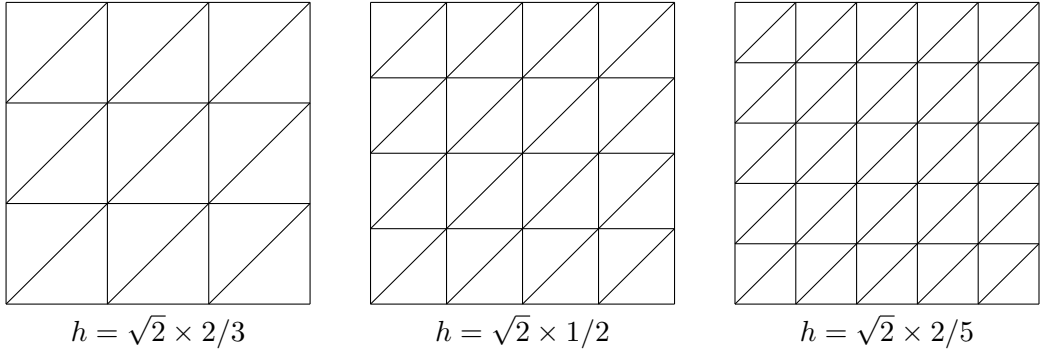
\begin{figure}
\begin{minipage}{.30\linewidth}
\begin{tikzpicture}[scale=4]
\draw (0,0/3) -- (1,0/3);
\draw (0,1/3) -- (1,1/3);
\draw (0,2/3) -- (1,2/3);
\draw (0,3/3) -- (1,3/3);

\draw (0/3,0) -- (0/3,1);
\draw (1/3,0) -- (1/3,1);
\draw (2/3,0) -- (2/3,1);
\draw (3/3,0) -- (3/3,1);

\draw (1/3,0) -- (1,2/3);
\draw (2/3,0) -- (1,1/3);

\draw (0,1/3) -- (2/3,1);
\draw (0,2/3) -- (1/3,1);

\draw (0,0) -- (1,1);

\draw (.5,0.) node[anchor=north] {$h = \sqrt 2 \times 2/3$};
\end{tikzpicture}
\end{minipage}
\begin{minipage}{.30\linewidth}
\begin{tikzpicture}[scale=4]
\draw (0,0/4) -- (1,0/4);
\draw (0,1/4) -- (1,1/4);
\draw (0,2/4) -- (1,2/4);
\draw (0,3/4) -- (1,3/4);
\draw (0,4/4) -- (1,4/4);

\draw (0/4,0) -- (0/4,1);
\draw (1/4,0) -- (1/4,1);
\draw (2/4,0) -- (2/4,1);
\draw (3/4,0) -- (3/4,1);
\draw (4/4,0) -- (4/4,1);

\draw (1/4,0) -- (1,3/4);
\draw (2/4,0) -- (1,2/4);
\draw (3/4,0) -- (1,1/4);

\draw (0,1/4) -- (3/4,1);
\draw (0,2/4) -- (2/4,1);
\draw (0,3/4) -- (1/4,1);

\draw (0,0) -- (1,1);

\draw (.5,0) node[anchor=north] {$h = \sqrt 2 \times 1/2$};
\end{tikzpicture}
\end{minipage}
\begin{minipage}{.30\linewidth}
\begin{tikzpicture}[scale=4]
\draw (0,0/5) -- (1,0/5);
\draw (0,1/5) -- (1,1/5);
\draw (0,2/5) -- (1,2/5);
\draw (0,3/5) -- (1,3/5);
\draw (0,4/5) -- (1,4/5);
\draw (0,5/5) -- (1,5/5);

\draw (0/5,0) -- (0/5,1);
\draw (1/5,0) -- (1/5,1);
\draw (2/5,0) -- (2/5,1);
\draw (3/5,0) -- (3/5,1);
\draw (4/5,0) -- (4/5,1);
\draw (5/5,0) -- (5/5,1);

\draw (1/5,0) -- (1,4/5);
\draw (2/5,0) -- (1,3/5);
\draw (3/5,0) -- (1,2/5);
\draw (4/5,0) -- (1,1/5);

\draw (0,1/5) -- (4/5,1);
\draw (0,2/5) -- (3/5,1);
\draw (0,3/5) -- (2/5,1);
\draw (0,4/5) -- (1/5,1);

\draw (0,0) -- (1,1);

\draw (.5,0) node[anchor=north] {$h = \sqrt 2 \times 2/5$};
\end{tikzpicture}
\end{minipage}
\caption{Cartesian meshes for the plane wave problem of Section~\ref{sec_PW}}
\label{figure_cartesian_meshes}
\end{figure}

We consider different values of the wavenumber $k$. In each case, we discretize problem
\eqref{eq_helmholtz_plane_wave} with meshes based on Cartesian grids
(see Figure~\ref{figure_cartesian_meshes}) with different sizes $h$.
The mesh sizes are selected so that the condition $\frac{kh}{2\pi p} \leq 1$ always holds true.
For all meshes and wavenumbers, we compute the relative estimators
(the factor $100$ allows one to read the relative errors as percentages)
\begin{equation*}
E_{\rm est} \eq
100 \cdot
\frac{\eta}{\norm{u}_{1,k,\Omega}}, \quad
\widetilde E_{\rm est} \eq c_{\rm up} E_{\rm est},
\end{equation*}
where, following~\eqref{eq_Cupperc} of Theorem~\ref{theorem_upper_bounds} and
\eqref{eq_cba_fine} of Theorem~\ref{theorem_cba} (here $d = 2$, $\beta=0$)
\begin{subequations}
\begin{align}
c_{\rm up} & \eq \sqrt{\left (\frac{1}{2} + \sqrt{\frac{1}{4} + (c_{\rm ba})^2}\right ) +
 \left (\frac{1}{2} + \sqrt{\frac{1}{4} +
(c_{\rm ba})^{2}}\right )^2 + \big(c_{\rm ba}\big)^2}, \label{eq_cup}\\
c_{\rm ba} & \eq c_{\rm i} \left (2 + c_{\rm stab} kh_\Omega \right ) k h \label{eq:def_numer_exp_cba}
\end{align}
\end{subequations}
with
\begin{equation*}
h_\Omega = 2\sqrt{2},
\quad
c_{\rm stab} \eq \frac{3+\sqrt{2}}{2\sqrt{2}},
\quad
c_{\rm i} \eq \frac{0.493}{\sqrt 2},
\end{equation*}
see Remarks~\ref{remark_stability} and~\ref{rem_Ci}. $\widetilde E_{\rm est}$ is the relative
percentage form of the guaranteed version of the upper
bound~\eqref{eq_est_up}, where $\Cupper$ is bounded from above by $c_{\rm up} = \sqrt{2}+\tilde\theta_1(c_{\rm ba})$; $E_{\rm est}$
is the relative percentage form of the constant-free equilibrated error estimator
$\eta$ given by~\eqref{eq_definition_eta}, without the prefactor
$\Cupper$ or $c_{\rm up}$. According to our theoretical results,
1) $E_{\rm est}$ and $\widetilde E_{\rm est}$ are $p$-robust; 2)
$\widetilde E_{\rm est}$ gives a guaranteed upper bound; and 3) $c_{\rm ba}$ as defined in~\eqref{eq:def_numer_exp_cba} tends to $0$, but
unfortunately only as $h \to 0$ and not as $p \to \infty$.
Furthermore, as the analytical solution is known, we also introduce
the relative percentage errors
\begin{equation*}
E_{\rm fem} \eq
100 \cdot \frac{\norm{u-u_h}_{1,k,\Omega}}{\norm{u}_{1,k,\Omega}}, \quad
E_{\rm ba} \eq
100 \cdot
\frac{\norm{u-P_h(u)}_{1,k,\Omega}}{\norm{u}_{1,k,\Omega}}, \quad
\end{equation*}
where the best approximation $P_h(u) \in V_h$ to $u$ is numerically computed by solving
\begin{equation*}
k^2 (P_h(u),v_h) + k(P_h(u),v_h)_\GA + (\grad (P_h(u)),\grad v_h) =
k^2 ( u,v_h) + k( u,v_h)_\GA + (\grad u,\grad v_h)
\end{equation*}
for all $v_h \in V_h$.
The behavior of $E_{\rm fem}$, $E_{\rm ba}$, $E_{\rm est}$, and $\widetilde E_{\rm est}$
for polynomial degrees $p=1$, $2$, and $4$ is respectively presented in Figures~\ref{figure_curve_p1_fullnorm}, \ref{figure_curve_p2_fullnorm}, and
\ref{figure_curve_p4_fullnorm}. In addition, Tables~\ref{table_efficiencies_p1},
\ref{table_efficiencies_p2}, and~\ref{table_efficiencies_p4} present
the effectivity indices $E_{\rm est}/E_{\rm fem}$ and
$\widetilde E_{\rm est}/E_{\rm fem}$ of the prefactor-free
and guaranteed relative estimators $E_{\rm est}$ and $\widetilde E_{\rm est}$,
respectively.

% We notice that in
% Figures~\ref{figure_curve_p1_fullnorm}--\ref{figure_curve_p4_fullnorm},
% the resolved regime $\frac{kh}{2 \pi p} \leq 1$ corresponds to the phase where the
% best-approximation error $E_{\rm ba}$ starts to converge, and the asymptotic regime
% $\Cba \leq 1$ corresponds to the phase where the finite element error $E_{\rm fem}$
% starts to converge. For instance for the wavenumber $k = 20\pi$, the settings where
% $\frac{kh}{2 \pi p} \leq 1$ are approximately $h \leq 2\sqrt{2}/32$, $h \leq 2\sqrt{2}/16$,
% and $h \leq 2\sqrt{2}/8$ for respectively $p=1$, $p=2$, and $p=4$, so that all but the coarsest
% meshes for $p=1$ fall into the resolved regime. \todo{(MV, maintained) Could we mark in the figures (for example) from where $\frac{kh}{2 \pi p} \leq 1$ is satisfied?}

At fixed wavenumber $k$, the prefactor-free estimator $E_{\rm est}$ is
reliable and efficient for the error $E_{\rm fem}$ up to a constant independent
of the mesh size $h$ and polynomial degree $p$, so that the values of $E_{\rm est}$ follow those of $E_{\rm fem}$ up to effectivity indices independent of $h$ and $p$. For instance for the wavenumber $k = 20\pi$, where the results cover the unresolved and resolved regimes, see above, the effectivity indices of $E_{\rm est}$ range between $0.11$ and $1.00$.
Also, for fixed wavenumber $k$ and mesh size $h$, the effectivity index actually improves and approaches the optimal value of one for higher values of $p$:
for instance for $k=10\pi$ and $h=2\sqrt{2}/128$,
the effectivity indices respectively read $0.20$, $0.93$, and $1.00$ for $p=1$, $2$, and $4$.
Unfortunately, $E_{\rm est}$ can severely underestimate $E_{\rm fem}$. The
underestimation becomes more pronounced as the wavenumber $k$ gets higher,
which can be seen in Figures~\ref{figure_curve_p1_fullnorm}--\ref{figure_curve_p4_fullnorm}
and Tables~\ref{table_efficiencies_p1}--\ref{table_efficiencies_p4}, see in particular Table~\ref{table_efficiencies_p2}, where the effectivity
index for $p=2$ and $k=60\pi$ drops to $0.07$ on a rather refined mesh with $h=2\sqrt{2}/256$, falling into the resolved regime with $\frac{kh}{2 \pi p} \approx 0.17$. This can happen in the preasymptotic regime $\Cba > 1$, since the above reliability and efficiency of $E_{\rm est}$, though robust
with respect to $h$ and $p$, is not robust with respect to $k$. In accordance with the theory, though, the effectivity index $E_{\rm est}/E_{\rm fem}$ indeed approaches the optimal value of one in the asymptotic regime.

The relative estimator $\widetilde E_{\rm est}$ indeed gives a guaranteed upper bound on the
finite element error $E_{\rm fem}$ in all situations. Its effectivity index
can unfortunately reach very high values.
Although it decreases rather swiftly with mesh refinement for $\mathcal P_1$ elements,
see Table~\ref{table_efficiencies_p1}, we were not able to design the upper
bound~\eqref{eq_cba_fine} on $\Cba$ to be sharp for $p > 1$:
we only employ it with $\beta=0$ which means that $c_{\rm ba}$ does not decrease with
increasing polynomial degree $p$ (cf. the asymptotic behavior of $\Cba$
with respect to both $h$ and $p$ in~\eqref{eq:bnd_Cba_shift} where $\varepsilon =1$
can be taken here).
Consequently, the effectivity indices of $\widetilde E_{\rm est}$ are relatively poor for
higher-order elements in Tables~\ref{table_efficiencies_p2} and~\ref{table_efficiencies_p4},
and, moreover, only improve with decreasing the mesh size $h$ but not with increasing the
polynomial degree $p$. We also see from Table~\ref{table_efficiencies_p1} when $k=\pi$
that asymptotically, the effectivity index of the guaranteed estimator $\widetilde E_{\rm est}$
is close to $\sqrt 2 \simeq 1.41$, which is in agreement with~\eqref{eq_Cupperc}. Recall that
theoretically, this is remedied by the use of the constant $1+\tilde\theta_2(\Cba,\tCba)$
in~\eqref{eq_Cupperf}. In practice, however, we do not have a computable estimate on $\tCba$.%, nor we dispose of a closed form of the function $\widetilde \theta$ from~\eqref{eq_bound_s_fine}.

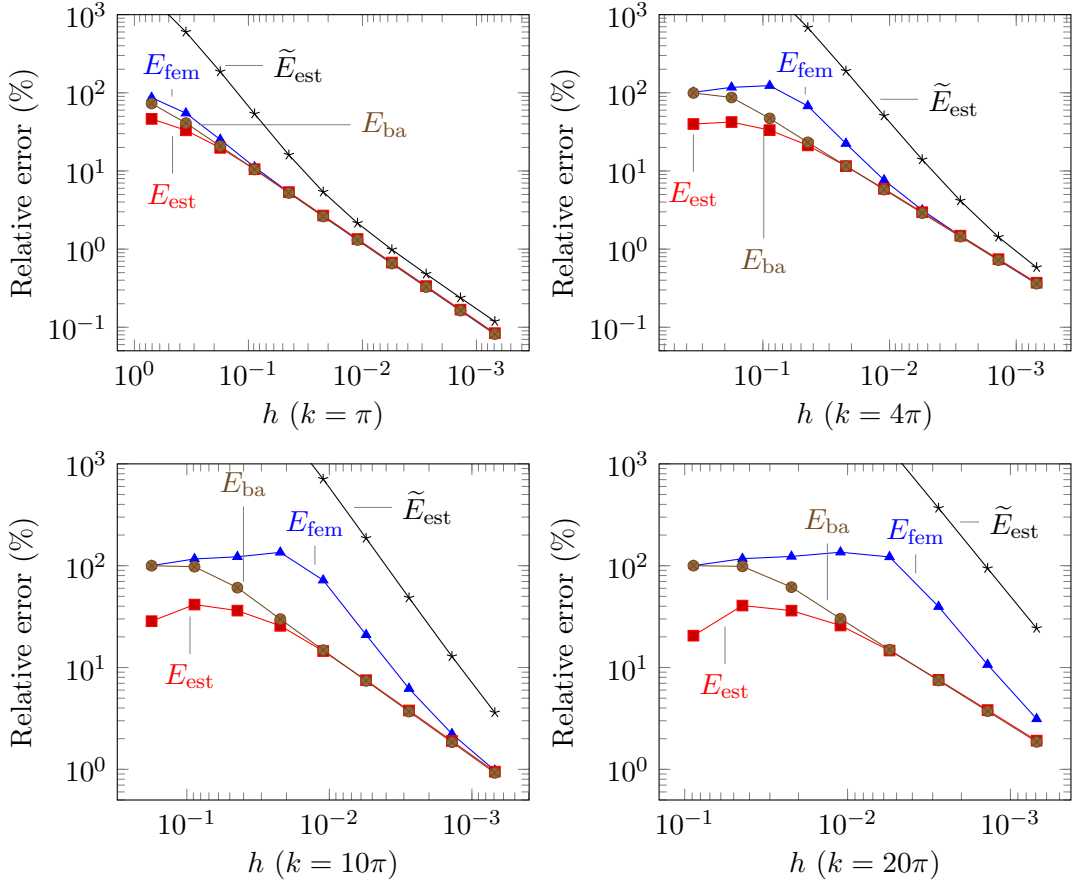
\begin{figure}
\begin{minipage}{.45\linewidth}
\begin{tikzpicture}
\begin{axis}
[
	width=\linewidth,
	xlabel={$h$ ($k = \pi$)},
	ylabel=Relative error (\%),
	xmode=log,
	ymode=log,
	x dir=reverse,
	ymin=5.e-2,
	ymax=1.e3
]

\plot[mark=triangle*,color=blue] table[x expr = sqrt(2)/\thisrow{n},y expr=100*\thisrow{analytical}/10.2024]%
{figures/fullnorm/plane_wave_p1/curve_0.5.dat}
node[pos=.05,pin={[pin distance=.1cm]90:{$E_{\rm fem}$}}] {};
\plot table[x expr = sqrt(2)/\thisrow{n},y expr=100*\thisrow{estimated}/10.2024]%
{figures/fullnorm/plane_wave_p1/curve_0.5.dat}
node[pos=.05,pin=-90:{$E_{\rm est}$}] {};
\plot table[x expr = sqrt(2)/\thisrow{n},y expr=100*\thisrow{best}/10.2024]%
{figures/fullnorm/plane_wave_p1/curve_0.5.dat}
node[pos=.1,pin={[pin distance=2cm]0:{$E_{\rm ba}$}}] {};
\plot table[x expr = sqrt(2)/\thisrow{n},y expr=100*\thisrow{corrected}/10.2024]%
{figures/fullnorm/plane_wave_p1/curve_0.5.dat}
node[pos=.2,pin=0:{$\widetilde E_{\rm est}$}] {};

\end{axis}
\end{tikzpicture}
\end{minipage}
\begin{minipage}{.45\linewidth}
\begin{tikzpicture}
\begin{axis}
[
	xlabel=$h$,
	xlabel={$h$ ($k = 4\pi$)},
	ylabel=Relative error (\%),
	xmode=log,
	ymode=log,
	x dir=reverse,
	ymin=5.e-2,
	ymax=1.e3,
	width=\linewidth
]
\plot[mark=triangle*,color=blue]
table[x expr = sqrt(2)/\thisrow{n},y expr=100*\thisrow{analytical}/36.9293]%
{figures/fullnorm/plane_wave_p1/curve_2.0.dat}
node[pos=.25,pin={[pin distance=.1cm]90:{$E_{\rm fem}$}}] {};
\plot table[x expr = sqrt(2)/\thisrow{n},y expr=100*\thisrow{estimated}/36.9293]%
{figures/fullnorm/plane_wave_p1/curve_2.0.dat}
node[pos=.0,pin=-90:{$E_{\rm est}$}] {};
\plot table[x expr = sqrt(2)/\thisrow{n},y expr=100*\thisrow{best}/36.9293]%
{figures/fullnorm/plane_wave_p1/curve_2.0.dat}
node[pos=.175,pin={[pin distance=1.5cm]-90:{$E_{\rm ba}$}}] {};
\plot table[x expr = sqrt(2)/\thisrow{n},y expr=100*\thisrow{corrected}/36.9293]%
{figures/fullnorm/plane_wave_p1/curve_2.0.dat}
node[pos=.5,pin=0:{$\widetilde E_{\rm est}$}] {};

\end{axis}
\end{tikzpicture}
\end{minipage}

\begin{minipage}{.45\linewidth}
\begin{tikzpicture}
\begin{axis}
[
	xlabel={$h$ ($k = 10\pi$)},
	ylabel=Relative error (\%),
	xmode=log,
	ymode=log,
	x dir=reverse,
	ymin=5.e-1,
	ymax=1.e3,
	width=\linewidth
]

\plot[mark=triangle*,color=blue]
table[x expr = sqrt(2)/\thisrow{n},y expr=100*\thisrow{analytical}/90.2464]%
{figures/fullnorm/plane_wave_p1/curve_5.0.dat}
node[pos=.35,pin={[pin distance=.25cm]90:{$E_{\rm fem}$}}] {};
\plot table[x expr = sqrt(2)/\thisrow{n},y expr=100*\thisrow{estimated}/90.2464]%
{figures/fullnorm/plane_wave_p1/curve_5.0.dat}
node[pos=.1,pin=-90:{$E_{\rm est}$}] {};
\plot table[x expr = sqrt(2)/\thisrow{n},y expr=100*\thisrow{best}/90.2464]%
{figures/fullnorm/plane_wave_p1/curve_5.0.dat}
node[pos=.225,pin={[pin distance=1cm]90:{$E_{\rm ba}$}}] {};
\plot table[x expr = sqrt(2)/\thisrow{n},y expr=100*\thisrow{corrected}/90.2464]%
{figures/fullnorm/plane_wave_p1/curve_5.0.dat}
node[pos=.5,pin=0:{$\widetilde E_{\rm est}$}] {};

\end{axis}
\end{tikzpicture}
\end{minipage}
\begin{minipage}{.45\linewidth}
\begin{tikzpicture}
\begin{axis}
[
	width=\linewidth,
	xlabel={$h$ ($k = 20\pi$)},
	ylabel=Relative error (\%),
	xmode=log,
	ymode=log,
	x dir=reverse,
	ymin=5.e-1,
	ymax=1.e3
]

\plot[mark=triangle*,color=blue]
table[x expr = sqrt(2)/\thisrow{n},y expr=100*\thisrow{analytical}/179.006]%
{figures/fullnorm/plane_wave_p1/curve_10.0.dat}
node[pos=.5,pin={[pin distance=.25cm]90:{$E_{\rm fem}$}}] {};
\plot table[x expr = sqrt(2)/\thisrow{n},y expr=100*\thisrow{estimated}/179.006]%
{figures/fullnorm/plane_wave_p1/curve_10.0.dat}
node[pos=.1,pin=-90:{$E_{\rm est}$}] {};
\plot table[x expr = sqrt(2)/\thisrow{n},y expr=100*\thisrow{best}/179.006]%
{figures/fullnorm/plane_wave_p1/curve_10.0.dat}
node[pos=.35,pin={[pin distance=.75cm]90:{$E_{\rm ba}$}}] {};
\plot table[x expr = sqrt(2)/\thisrow{n},y expr=100*\thisrow{corrected}/179.006]%
{figures/fullnorm/plane_wave_p1/curve_10.0.dat}
node[pos=.7,pin={[pin distance=.25cm]0:{$\widetilde E_{\rm est}$}}] {};

\end{axis}
\end{tikzpicture}
\end{minipage}

\caption{Behaviors of the estimated and analytical errors for
the plane wave test case of Section~\ref{sec_PW} with $\mathcal P_1$ elements}
\label{figure_curve_p1_fullnorm}
\end{figure}

\begin{table}
\begin{tabular}{|l|rr|rr|rr|rr|}
\hline
& \multicolumn{2}{c|}{$k = \pi$} & \multicolumn{2}{c|}{$k = 4\pi$} & \multicolumn{2}{c|}{$k = 10\pi$} & \multicolumn{2}{c|}{$k = 20\pi$}\\
\hline
%$h$      & standard & modified & standard & modified & standard & modified & standard & modified \\
\hspace{.25cm}$h$      & $E_{\rm est}$ & $\widetilde E_{\rm est}$ & $E_{\rm est}$ & $\widetilde E_{\rm est}$ & $E_{\rm est}$ & $\widetilde E_{\rm est}$ & $E_{\rm est}$ & $\widetilde E_{\rm est}$ \\
\hline
$2\sqrt{2}/8   $ &    0.78 &     7.38  &   0.36 &    45.55 &    0.28 &   219.62 &   0.16 &   502.52 \\
$2\sqrt{2}/16  $ &    0.94 &     4.80  &   0.27 &    17.18 &    0.36 &   137.48 &   0.21 &   314.17 \\
$2\sqrt{2}/32  $ &    1.01 &     3.01  &   0.31 &    10.07 &    0.30 &    57.24 &   0.35 &   265.17 \\
$2\sqrt{2}/64  $ &    1.02 &     2.05  &   0.52 &     8.48 &    0.19 &    18.42 &   0.29 &   112.57 \\
$2\sqrt{2}/128 $ &    1.03 &     1.64  &   0.77 &     6.65 &    0.20 &     9.84 &   0.19 &    36.61 \\
$2\sqrt{2}/256 $ &    1.03 &     1.51  &   0.94 &     4.43 &    0.36 &     8.89 &   0.12 &    11.60 \\
$2\sqrt{2}/512 $ &    1.03 &     1.47  &   1.01 &     2.81 &    0.61 &     7.79 &   0.19 &     9.27 \\
$2\sqrt{2}/1024$ &    1.03 &     1.46  &   1.02 &     1.96 &    0.85 &     5.76 &   0.36 &     8.81 \\
$2\sqrt{2}/2048$ &    1.03 &     1.46  &   1.03 &     1.61 &    0.97 &     3.70 &   0.61 &     7.76 \\
\hline
\end{tabular}
\medskip
\caption{Effectivity indices of the asymptotical and guaranteed error estimators in the
plane wave test case of Section~\ref{sec_PW} with $\mathcal P_1$ elements}
\label{table_efficiencies_p1}
\end{table}

\begin{figure}
\begin{minipage}{.45\linewidth}
\begin{tikzpicture}
\begin{axis}
[
	width=\linewidth,
	xlabel={$h$ ($k = 10\pi$)},
	ylabel=Relative error (\%),
	xmode=log,
	ymode=log,
	x dir=reverse,
	ymin=5.e-3,
	ymax=1.e3
]

\plot[mark=triangle*,color=blue]
table[x expr = 1./\thisrow{n},y expr=100*\thisrow{analytical}/90.2608]%
{figures/fullnorm/plane_wave_p2/curve_5.0.dat}
node[pos=.05,pin={[pin distance=.1cm]90:{$E_{\rm fem}$}}] {};
\plot table[x expr = 1./\thisrow{n},y expr=100*\thisrow{estimated}/90.2608]%
{figures/fullnorm/plane_wave_p2/curve_5.0.dat}
node[pos=.05,pin=-90:{$E_{\rm est}$}] {};
\plot table[x expr = 1./\thisrow{n},y expr=100*\thisrow{best}/90.2608]%
{figures/fullnorm/plane_wave_p2/curve_5.0.dat}
node[pos=.1,pin={[pin distance=1.0cm]0:{$E_{\rm ba}$}}] {};
\plot table[x expr = 1./\thisrow{n},y expr=100*\thisrow{corrected}/90.2608]%
{figures/fullnorm/plane_wave_p2/curve_5.0.dat}
node[pos=.5,pin=0:{$\widetilde E_{\rm est}$}] {};
\end{axis}
\end{tikzpicture}
\end{minipage}
\begin{minipage}{.45\linewidth}
\begin{tikzpicture}
\begin{axis}
[
	xlabel=$h$,
	xlabel={$h$ ($k = 20\pi$)},
	ylabel=Relative error (\%),
	xmode=log,
	ymode=log,
	x dir=reverse,
	ymin=5.e-3,
	ymax=1.e3,
	width=\linewidth
]

\plot[mark=triangle*,color=blue]
table[x expr = 1./\thisrow{n},y expr=100*\thisrow{analytical}/179.124]%
{figures/fullnorm/plane_wave_p2/curve_10.0.dat}
node[pos=.25,pin={[pin distance=.1cm]90:{$E_{\rm fem}$}}] {};
\plot table[x expr = 1./\thisrow{n},y expr=100*\thisrow{estimated}/179.124]%
{figures/fullnorm/plane_wave_p2/curve_10.0.dat}
node[pos=.1,pin=-90:{$E_{\rm est}$}] {};
\plot table[x expr = 1./\thisrow{n},y expr=100*\thisrow{best}/179.124]%
{figures/fullnorm/plane_wave_p2/curve_10.0.dat}
node[pos=.2,pin={[pin distance=1cm]0:{$E_{\rm ba}$}}] {};
\plot table[x expr = 1./\thisrow{n},y expr=100*\thisrow{corrected}/179.124]%
{figures/fullnorm/plane_wave_p2/curve_10.0.dat}
node[pos=.7,pin=0:{$\widetilde E_{\rm est}$}] {};
\end{axis}
\end{tikzpicture}
\end{minipage}

\begin{minipage}{.45\linewidth}
\begin{tikzpicture}
\begin{axis}
[
	xlabel={$h$ ($k = 40\pi$)},
	ylabel=Relative error (\%),
	xmode=log,
	ymode=log,
	x dir=reverse,
	ymin=2.e-1,
	ymax=2.e3,
	width=\linewidth
]

\plot[mark=triangle*,color=blue]
table[x expr = 1./\thisrow{n},y expr=100*\thisrow{analytical}/356.835]%
{figures/fullnorm/plane_wave_p2/curve_20.0.dat}
node[pos=.2,pin=90:{$E_{\rm fem}$}] {};
\plot table[x expr = 1./\thisrow{n},y expr=100*\thisrow{estimated}/356.835]%
{figures/fullnorm/plane_wave_p2/curve_20.0.dat}
node[pos=.2,pin=-90:{$E_{\rm est}$}] {};
\plot table[x expr = 1./\thisrow{n},y expr=100*\thisrow{best}/356.835]%
{figures/fullnorm/plane_wave_p2/curve_20.0.dat}
node[pos=.275,pin={[pin distance=1.5cm]0:{$E_{\rm ba}$}}] {};
\plot table[x expr = 1./\thisrow{n},y expr=100*\thisrow{corrected}/356.835]%
{figures/fullnorm/plane_wave_p2/curve_20.0.dat}
node[pos=.7,pin={[pin distance=.25cm]0:{$\widetilde E_{\rm est}$}}] {};
\end{axis}
\end{tikzpicture}
\end{minipage}
\begin{minipage}{.45\linewidth}
\begin{tikzpicture}
\begin{axis}
[
	width=\linewidth,
	xlabel={$h$ ($k = 60\pi$)},
	ylabel=Relative error (\%),
	xmode=log,
	ymode=log,
	x dir=reverse,
	ymin=2.e-1,
	ymax=2.e3
]

\plot[mark=triangle*,color=blue]
table[x expr = 1./\thisrow{n},y expr=100*\thisrow{analytical}/534.502]%
{figures/fullnorm/plane_wave_p2/curve_30.0.dat}
node[pos=.7,pin=75:{$E_{\rm fem}$}] {};
\plot table[x expr = 1./\thisrow{n},y expr=100*\thisrow{estimated}/534.502]%
{figures/fullnorm/plane_wave_p2/curve_30.0.dat}
node[pos=.2,pin=-90:{$E_{\rm est}$}] {};
\plot table[x expr = 1./\thisrow{n},y expr=100*\thisrow{best}/534.502]%
{figures/fullnorm/plane_wave_p2/curve_30.0.dat}
node[pos=.35,pin={[pin distance=1.5cm]22.5:{$E_{\rm ba}$}}] {};
\plot table[x expr = 1./\thisrow{n},y expr=100*\thisrow{corrected}/534.502]%
{figures/fullnorm/plane_wave_p2/curve_30.0.dat}
node[pos=.8,pin=180:{$\widetilde E_{\rm est}$}] {};
\end{axis}
\end{tikzpicture}
\end{minipage}

\caption{Behaviors of the estimated and analytical errors for the
plane wave test case of Section~\ref{sec_PW} with $\mathcal P_2$ elements}
\label{figure_curve_p2_fullnorm}
\end{figure}

\begin{table}
\begin{tabular}{|l|rr|rr|rr|rr|}
\hline
& \multicolumn{2}{c|}{$k = 10\pi$} & \multicolumn{2}{c|}{$k = 20\pi$} & \multicolumn{2}{c|}{$k = 40\pi$} & \multicolumn{2}{c|}{$k = 60\pi$}\\
\hline
%$h$      & standard & modified & standard & modified & standard & modified & standard & modified \\
\hspace{.25cm}$h$      & $E_{\rm est}$ & $\widetilde E_{\rm est}$ & $E_{\rm est}$ & $\widetilde E_{\rm est}$ & $E_{\rm est}$ & $\widetilde E_{\rm est}$ & $E_{\rm est}$ & $\widetilde E_{\rm est}$ \\
\hline
$2\sqrt{2}/32  $ &     0.19 &    46.58 &     0.22 &   213.82 &     0.46 &  1753.21 &     0.12 & 1014.61 \\
$2\sqrt{2}/64  $ &     0.55 &    66.32 &     0.11 &    54.52 &     0.22 &   424.74 &     0.15 &  658.57 \\
$2\sqrt{2}/128 $ &     0.93 &    56.13 &     0.32 &    75.82 &     0.10 &    91.05 &     0.17 &  359.92 \\
$2\sqrt{2}/256 $ &     1.00 &    30.36 &     0.79 &    94.49 &     0.17 &    79.18 &     0.07 &   70.88 \\
$2\sqrt{2}/512 $ &     1.00 &    15.57 &     0.98 &    58.91 &     0.55 &   129.81 &     0.19 &  102.77 \\
$2\sqrt{2}/1024$ &     1.00 &     8.12 &     1.00 &    30.30 &     0.93 &   111.08 &     0.61 &  162.98 \\
\hline
\end{tabular}
\medskip
\caption{Effectivity indices of the asymptotical and guaranteed error estimators
in the plane wave test case of Section~\ref{sec_PW} with $\mathcal P_2$ elements}
\label{table_efficiencies_p2}
\end{table}

\begin{figure}
\begin{minipage}{.45\linewidth}
\begin{tikzpicture}
\begin{axis}
[
	width=\linewidth,
	xlabel={$h$ ($k = 10\pi$)},
	ylabel=Relative error (\%),
	xmode=log,
	ymode=log,
	x dir=reverse,
	ymin=1.e-6,
	ymax=1.e2
]

\plot[mark=triangle*,color=blue]
table[x expr = 1./\thisrow{n},y expr=100*\thisrow{analytical}/90.2608]%
{figures/fullnorm/plane_wave_p4/curve_5.0.dat}
node[pos=.05,pin={[pin distance=.1cm]90:{$E_{\rm fem}$}}] {};
\plot table[x expr = 1./\thisrow{n},y expr=100*\thisrow{estimated}/90.2608]%
{figures/fullnorm/plane_wave_p4/curve_5.0.dat}
node[pos=.05,pin=-90:{$E_{\rm est}$}] {};
\plot table[x expr = 1./\thisrow{n},y expr=100*\thisrow{best}/90.2608]%
{figures/fullnorm/plane_wave_p4/curve_5.0.dat}
node[pos=.1,pin={[pin distance=0.5cm]0:{$E_{\rm ba}$}}] {};
\plot table[x expr = 1./\thisrow{n},y expr=100*\thisrow{corrected}/90.2608]%
{figures/fullnorm/plane_wave_p4/curve_5.0.dat}
node[pos=.5,pin=0:{$\widetilde E_{\rm est}$}] {};
\end{axis}
\end{tikzpicture}
\end{minipage}
\begin{minipage}{.45\linewidth}
\begin{tikzpicture}
\begin{axis}
[
	xlabel=$h$,
	xlabel={$h$ ($k = 20\pi$)},
	ylabel=Relative error (\%),
	xmode=log,
	ymode=log,
	x dir=reverse,
	ymin=1.e-6,
	ymax=1.e2,
	width=\linewidth
]

\plot[mark=triangle*,color=blue]
table[x expr = 1./\thisrow{n},y expr=100*\thisrow{analytical}/179.124]%
{figures/fullnorm/plane_wave_p4/curve_10.0.dat}
node[pos=0.,pin=0:{$E_{\rm fem}$}] {};
\plot table[x expr = 1./\thisrow{n},y expr=100*\thisrow{estimated}/179.124]%
{figures/fullnorm/plane_wave_p4/curve_10.0.dat}
node[pos=.1,pin=-90:{$E_{\rm est}$}] {};
\plot table[x expr = 1./\thisrow{n},y expr=100*\thisrow{best}/179.124]%
{figures/fullnorm/plane_wave_p4/curve_10.0.dat}
node[pos=.2,pin={[pin distance=1cm]0:{$E_{\rm ba}$}}] {};
\plot table[x expr = 1./\thisrow{n},y expr=100*\thisrow{corrected}/179.124]%
{figures/fullnorm/plane_wave_p4/curve_10.0.dat}
node[pos=.6,pin=0:{$\widetilde E_{\rm est}$}] {};
\end{axis}
\end{tikzpicture}
\end{minipage}

\begin{minipage}{.45\linewidth}
\begin{tikzpicture}
\begin{axis}
[
	xlabel={$h$ ($k = 40\pi$)},
	ylabel=Relative error (\%),
	xmode=log,
	ymode=log,
	x dir=reverse,
	ymin=1.e-3,
	ymax=3.e3,
	width=\linewidth
]

\plot[mark=triangle*,color=blue]
table[x expr = 1./\thisrow{n},y expr=100*\thisrow{analytical}/356.842]%
{figures/fullnorm/plane_wave_p4/curve_20.0.dat}
node[pos=.15,pin={[pin distance=.1]45:{$E_{\rm fem}$}}] {};
\plot table[x expr = 1./\thisrow{n},y expr=100*\thisrow{estimated}/356.842]%
{figures/fullnorm/plane_wave_p4/curve_20.0.dat}
node[pos=.0,pin=-90:{$E_{\rm est}$}] {};
\plot table[x expr = 1./\thisrow{n},y expr=100*\thisrow{best}/356.842]%
{figures/fullnorm/plane_wave_p4/curve_20.0.dat}
node[pos=.1,pin={[pin distance=1.5cm]0:{$E_{\rm ba}$}}] {};
\plot table[x expr = 1./\thisrow{n},y expr=100*\thisrow{corrected}/356.842]%
{figures/fullnorm/plane_wave_p4/curve_20.0.dat}
node[pos=.6,pin=0:{$\widetilde E_{\rm est}$}] {};
\end{axis}
\end{tikzpicture}
\end{minipage}
\begin{minipage}{.45\linewidth}
\begin{tikzpicture}
\begin{axis}
[
	width=\linewidth,
	xlabel={$h$ ($k = 60\pi$)},
	ylabel=Relative error (\%),
	xmode=log,
	ymode=log,
	x dir=reverse,
	ymin=1.e-3,
	ymax=3.e3
]

\plot[mark=triangle*,color=blue]
table[x expr = 1./\thisrow{n},y expr=100*\thisrow{analytical}/534.558]%
{figures/fullnorm/plane_wave_p4/curve_30.0.dat}
node[pos=.2,pin={[pin distance=.1cm]75:{$E_{\rm fem}$}}] {};
\plot table[x expr = 1./\thisrow{n},y expr=100*\thisrow{estimated}/534.558]%
{figures/fullnorm/plane_wave_p4/curve_30.0.dat}
node[pos=.15,pin=-90:{$E_{\rm est}$}] {};
\plot table[x expr = 1./\thisrow{n},y expr=100*\thisrow{best}/534.558]%
{figures/fullnorm/plane_wave_p4/curve_30.0.dat}
node[pos=.15,pin={[pin distance=1.5cm]0:{$E_{\rm ba}$}}] {};
\plot table[x expr = 1./\thisrow{n},y expr=100*\thisrow{corrected}/356.842]%
{figures/fullnorm/plane_wave_p4/curve_30.0.dat}
node[pos=.6,pin=0:{$\widetilde E_{\rm est}$}] {};
\end{axis}
\end{tikzpicture}
\end{minipage}

\caption{Behaviors of the estimated and analytical errors for the plane
wave test case of Section~\ref{sec_PW} with $\mathcal P_4$ elements}
\label{figure_curve_p4_fullnorm}
\end{figure}

\begin{table}
\begin{tabular}{|l|rr|rr|rr|rr|}
\hline
& \multicolumn{2}{c|}{$k = 10\pi$} & \multicolumn{2}{c|}{$k = 20\pi$} & \multicolumn{2}{c|}{$k = 40\pi$} & \multicolumn{2}{c|}{$k = 60\pi$}\\
\hline
%$h$      & standard & modified & standard & modified & standard & modified & standard & modified \\
\hspace{.25cm} $h$      & $E_{\rm est}$ & $\widetilde E_{\rm est}$ & $E_{\rm est}$ & $\widetilde E_{\rm est}$ & $E_{\rm est}$ & $\widetilde E_{\rm est}$ & $E_{\rm est}$ & $\widetilde E_{\rm est}$ \\
\hline
$2\sqrt{2}/32  $ &     0.95 &   227.91 &     0.24 &   224.79 &     0.10 &   376.03 &     0.30 &  2548.49 \\
$2\sqrt{2}/64  $ &     0.99 &   119.14 &     0.92 &   438.82 &     0.12 &   234.70 &     0.11 &   451.99 \\
$2\sqrt{2}/128 $ &     1.00 &    60.35 &     0.99 &   236.00 &     0.83 &   787.63 &     0.23 &   486.20 \\
$2\sqrt{2}/256 $ &     1.00 &    30.54 &     1.00 &   119.27 &     0.99 &   469.50 &     0.94 &  1000.73 \\
$2\sqrt{2}/512 $ &     1.00 &    15.58 &     1.00 &    60.06 &     1.00 &   237.14 &     1.00 &   530.32 \\
\hline
\end{tabular}
\medskip
\caption{Effectivity indices of the asymptotical and guaranteed error estimators in the
plane wave test case of Section~\ref{sec_PW} with $\mathcal P_4$ elements}
\label{table_efficiencies_p4}
\end{table}

\subsection{Scattering by a non-trapping obstacle} \label{sec_NTO}

We now consider the scattering of a plane wave by an obstacle. This problem consists in finding $u$ such that
\begin{equation}
\label{eq_helmholtz_scattering}
\left \{
\begin{array}{rcll}
-k^2 u - \Delta u &=& 0 & \text{ in } \Omega,
\\
u &=& 0 & \text{ on } \GD,
\\
\grad u {\cdot} \nn - iku &=& g & \text{ on } \GA,
\end{array}
\right .
\end{equation}
where again $g = \grad \xi_\anglepw {\cdot} \nn -ik\xi_\anglepw$ with $\xi_\anglepw$ given in Section~\ref{sec_PW}.
The computational domain is constructed such that $\Omega = \Omega_0 \setminus \overline{D}$,
$\GD = \partial D$, and $\GA = \partial \Omega_0$, where $\Omega_0 = (-1,1)^2$ and
\begin{equation*}
D = \left \{
\xx \in \Omega \; \left | \; 2|\xx_1| - \frac{1}{2} < \xx_2 < |\xx_1|
\right .
\right \},
\end{equation*}
see the left panel of Figure~\ref{figure_scattering_settings}.
We see that we have
$\xx \cdot \nn \leq 0$ on $\GD$ and $\xx \cdot \nn > 0$ on
$\GA$, so that this setting enters Case 1a) of Theorem
\ref{theorem_cba} with $\xx_0 = (0,0)$.

We select the wavenumbers $k = 2\pi$ and $10\pi$ and employ polynomials
of degree $p=1$, $2$, and $3$. As the analytical solution of the problem
is not available, we employ an accurate numerical solution
as a reference. Specifically, for each mesh $\TT_h$, we employ
the approximation $u \simeq \widetilde u_h$ for the comparison, where $\widetilde u_h$
is computed using the mesh $\TT_h$ with $\mathcal P_6$ finite elements.

In order to generate unstructured adaptive meshes, we consider a simple procedure based on the software platform {\tt mmg}~\cite{mmg3d}.
We start with the initial mesh depicted in the right panel of Figure~\ref{figure_scattering_settings}. The software {\tt mmg}
allows one to impose a map of maximal allowed mesh sizes (or metric). This map is specified
by defining values on the vertices of a previously introduced mesh.
Thus, at each iteration, after solving the problem on the mesh $\TT_h$
and computing the corresponding estimators $\eta_K$, we produce a metric to
generate the mesh of the next iteration. We start by defining a new maximal
mesh size $h^\star_K$ for each element. This is done by sorting the elements
in decreasing values of $\eta_K$, setting $h^\star_K \eq h_K/2$ for the $|\TT_h|/10$
first elements, and defining $h^\star_K \eq 1.1 h_K$ for the last $|\TT_h|/10$ elements
(we set $h^\star_K = h_K$ for the remaining elements).
Then, the maximal mesh size value at the vertex $\aaa \in \VV_h$ is specified as
$h_\aaa^\star \eq \min_{K \in \TT_{\aaa}} h_K^\star$. Examples of meshes produced by this algorithm can be seen in Figure~\ref{figure_scattering_error}.
% We point out that the object of our study is not to analyze the performance of this very simple
% adaptive algorithm (actually, we start with an initial mesh that is already adapted), but only
% to analyze the performance of the estimator on a prototypical family of adaptive meshes.

Figures~\ref{figure_scattering_curve_1} and~\ref{figure_scattering_curve_5} represent the relative
percentage analytical and estimated errors
\begin{equation*}
E_{\rm fem} \eq 100 \cdot \frac{\norm{\widetilde u_h-u_h}_{1,k,\Omega}}{\norm{\widetilde u_h}_{1,k,\Omega}},
\quad
E_{\rm est} \eq 100 \cdot \frac{\eta}{\norm{\widetilde u_h}_{1,k,\Omega}}
\end{equation*}
in the different stages of the adaptive procedure for both $k = 2\pi$ and $k = 10\pi$ and for all $\mathcal P_1$, $\mathcal P_2$, and $\mathcal P_3$ elements. All the observations made
in the example of Section~\ref{sec_PW} still hold true here, even though the meshes are now unstructured
and feature elements of significantly different sizes. The relative estimators $E_{\rm est}$ are in particular a much better match to the approximate relative errors $E_{\rm fem}$ for $\mathcal P_3$ elements than for $\mathcal P_1$ and $\mathcal P_2$ elements in the $k = 10\pi$ case.

Following~\eqref{eq_Cupperc} in Theorem~\ref{theorem_upper_bounds}, Case 1a) of Theorem~\ref{theorem_cba}, and Remark~\ref{remark_stability},
we can also define
\begin{equation*}
\widetilde E_{\rm est} \eq c_{\rm up} E_{\rm est}
\end{equation*}
with $c_{\rm up}$ given by~\eqref{eq_cup} and
\begin{equation*}
c_{\rm ba} \eq \left (
\left (1 + c_{\rm stab} kh_\Omega\right ) +
\left (1 + c_{\rm stab} kh_\Omega\right )^2
\right )^{\frac12},
\quad
c_{\rm stab} \eq \frac{3 + \sqrt{2}}{2\sqrt{2}},
\quad
h_\Omega = 2 \sqrt 2,
\end{equation*}
where we have employed the point $\xx_0 = \mathbf 0$ in Remark~\ref{remark_stability}.
We remark that here $c_{\rm up}$ only depends on $k$
and improves neither with the mesh size $h$ nor with the polynomial degree $p$. We compute
$c_{\rm up} = 42.05$ for $k = 2\pi$, and $c_{\rm up} = 198.94$ for $k = 10\pi$.
We observe that with this definition, $\widetilde E_{\rm est}$ indeed constitutes
a guaranteed upper bound for all meshes, wavenumbers, and polynomial degrees
considered in this example. However, as the overestimation factor for $E_{\rm est}$
is about $1$ asymptotically, the overestimation factor for $\widetilde E_{\rm est}$
will be about 40 and 200 for $k = 2\pi$ and $k = 10\pi$, respectively, which
might unfortunately be too large for being useful in applications.

Finally, Figure~\ref{figure_scattering_error} depicts the local estimators
$\eta_K$ of~\eqref{eq_definition_eta} compared to the actual approximate
errors $e_K = \norm{\widetilde u_h-u_h}_{1,k,K}$,
evaluated using the reference numerical solution $\widetilde u_h$
on a sequence of adapted meshes for $k=10\pi$ and $\mathcal P_3$ elements.
We see that the estimators $\eta_K$ provide a very good representation of the error distribution,
even if the wavenumber is relatively high and the mesh is unstructured,
with a significant ratio between the largest and the smallest element sizes.
This result illustrates the local efficiency of the proposed estimator
as stated in~\eqref{eq_est_low_loc}.
The solution corresponding to the finest mesh is depicted in
Figure~\ref{figure_scattering_solution}.

\begin{figure}
\begin{minipage}{.45\linewidth}
\begin{tikzpicture}[scale=3]
\draw[dashed] (-1,-1) -- (-1,1) -- (1,1) -- (1,-1) -- cycle;

\draw (0,-.5) -- (.5,.5) -- (0,0) -- (-.5,.5) -- cycle;

\draw (.25,0.) node[anchor=west] {$\GD$};

\draw (1,-1) node[anchor=south east] {$\GA$};
\draw (-.8,-1) node[anchor=south east] {$\Omega$};
\end{tikzpicture}
\end{minipage}
\begin{minipage}{.45\linewidth}
\input{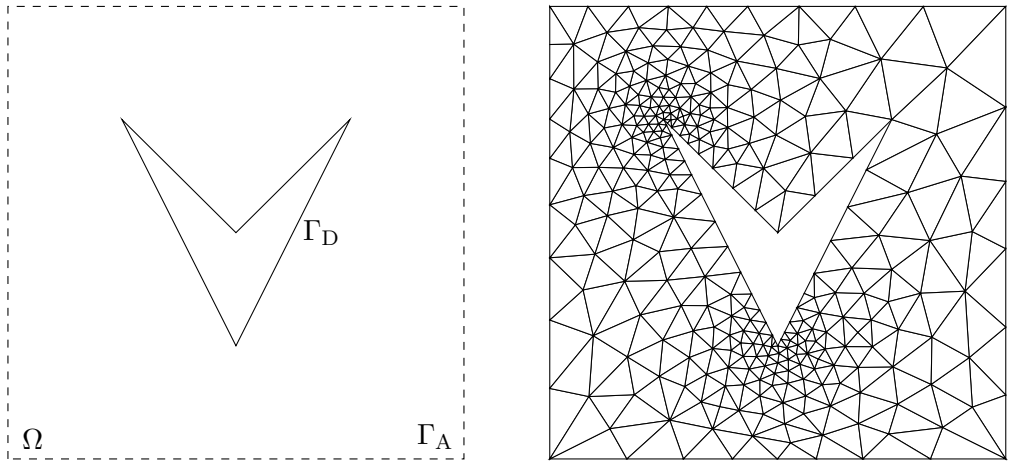}
\end{minipage}
\caption{Scattering problem of Section~\ref{sec_NTO}: domain settings (left) and the initial mesh (right)}
\label{figure_scattering_settings}
\end{figure}

\begin{figure}
\begin{minipage}{.45\linewidth}
\begin{tikzpicture}
\begin{axis}
[
	width=\linewidth,
	xlabel={Iterations ($\PP_1$)},
	ylabel=Relative error (\%),
	ymode=log
]

\plot table[x expr = \thisrow{n},y expr=100*\thisrow{analytical}/85.3087]%
{figures/scattering/data/1/errors_p1.dat}
node[pos=0,pin=0:{$E_{\rm fem}$}] {};
\plot table[x expr = \thisrow{n},y expr=100*\thisrow{estimated}/85.3087]%
{figures/scattering/data/1/errors_p1.dat}
node[pos=.05,pin=-90:{$E_{\rm est}$}] {};
\end{axis}
\end{tikzpicture}
\end{minipage}
\begin{minipage}{.45\linewidth}
\begin{tikzpicture}
\begin{axis}
[
	width=\linewidth,
	xlabel={Iterations ($\PP_2$)},
	ylabel=Relative error (\%),
	ymode=log
]

\plot table[x expr = \thisrow{n},y expr=100*\thisrow{analytical}/85.3087]%
{figures/scattering/data/1/errors_p2.dat}
node[pos=.65,pin=-90:{$E_{\rm fem}$}] {};
\plot table[x expr = \thisrow{n},y expr=100*\thisrow{estimated}/85.3087]%
{figures/scattering/data/1/errors_p2.dat}
node[pos=.65,pin= 90:{$E_{\rm est}$}] {};
\end{axis}
\end{tikzpicture}
\end{minipage}

\begin{minipage}{.45\linewidth}
\begin{tikzpicture}
\begin{axis}
[
	width=\linewidth,
	xlabel={Iterations ($\PP_3$)},
	ylabel=Relative error (\%),
	ymode=log
]

\plot table[x expr = \thisrow{n},y expr=100*\thisrow{analytical}/85.3087]%
{figures/scattering/data/1/errors_p3.dat}
node[pos=.65,pin=-90:{$E_{\rm fem}$}] {};
\plot table[x expr = \thisrow{n},y expr=100*\thisrow{estimated}/85.3087]%
{figures/scattering/data/1/errors_p3.dat}
node[pos=.65,pin= 90:{$E_{\rm est}$}] {};
\end{axis}
\end{tikzpicture}
\end{minipage}
\caption{Behaviors of the estimated and analytical errors in the adaptive procedure for
the scattering problem of Section~\ref{sec_NTO} with $k = 2\pi$}
\label{figure_scattering_curve_1}
\end{figure}
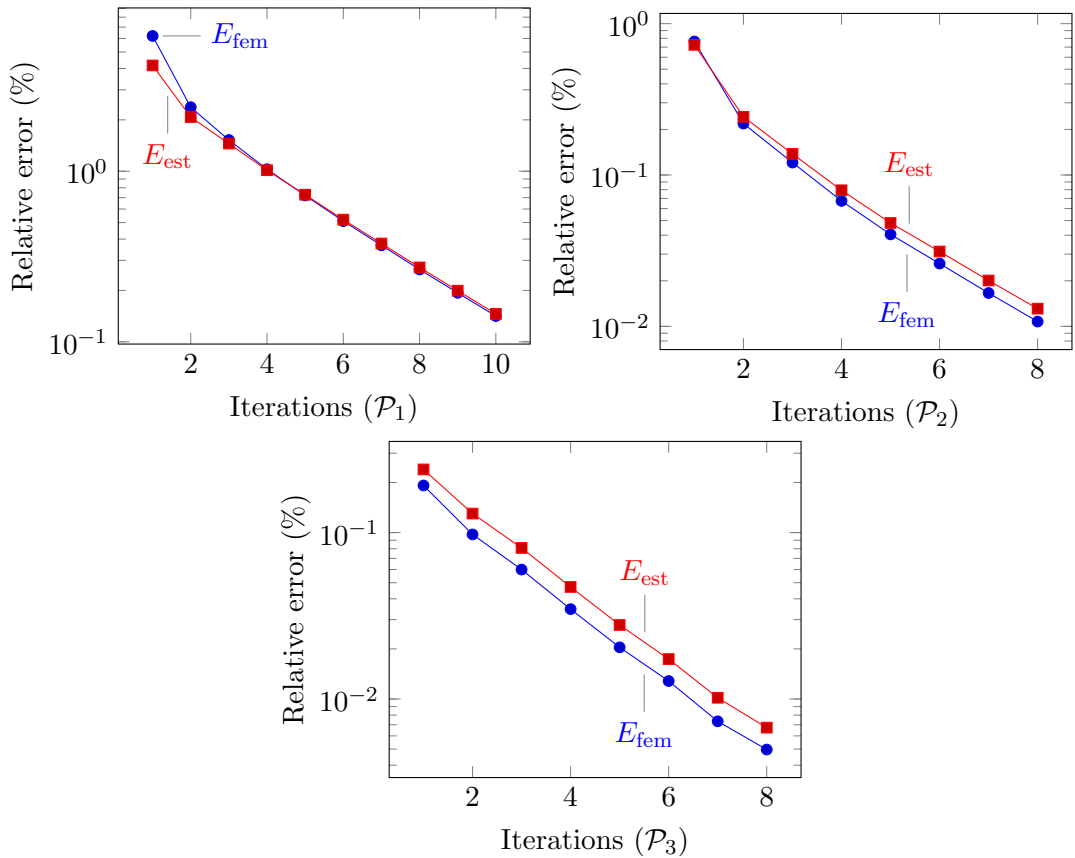

\begin{figure}
\begin{minipage}{.45\linewidth}
\begin{tikzpicture}
\begin{axis}
[
	width=\linewidth,
	xlabel={Iterations ($\PP_1$)},
	ylabel=Relative error (\%),
	ymode=log
]

\plot table[x expr = \thisrow{n},y expr=100*\thisrow{analytical}/85.3087]%
{figures/scattering/data/5/errors_p1.dat}
node[pos=.75,pin=90:{$E_{\rm fem}$}] {};
\plot table[x expr = \thisrow{n},y expr=100*\thisrow{estimated}/85.3087]%
{figures/scattering/data/5/errors_p1.dat}
node[pos=.05,pin=-90:{$E_{\rm est}$}] {};
\end{axis}
\end{tikzpicture}
\end{minipage}
\begin{minipage}{.45\linewidth}
\begin{tikzpicture}
\begin{axis}
[
	width=\linewidth,
	xlabel={Iterations ($\PP_2$)},
	ylabel=Relative error (\%),
	ymode=log
]

\plot table[x expr = \thisrow{n},y expr=100*\thisrow{analytical}/85.3087]%
{figures/scattering/data/5/errors_p2.dat}
node[pos=.15,pin=0:{$E_{\rm fem}$}] {};
\plot table[x expr = \thisrow{n},y expr=100*\thisrow{estimated}/85.3087]%
{figures/scattering/data/5/errors_p2.dat}
node[pos=.05,pin=-90:{$E_{\rm est}$}] {};
\end{axis}
\end{tikzpicture}
\end{minipage}

\begin{minipage}{.45\linewidth}
\begin{tikzpicture}
\begin{axis}
[
	width=\linewidth,
	xlabel={Iterations ($\PP_3$)},
	ylabel=Relative error (\%),
	ymode=log
]

\plot table[x expr = \thisrow{n},y expr=100*\thisrow{analytical}/85.3087]%
{figures/scattering/data/5/errors_p3.dat}
node[pos=.05,pin=0:{$E_{\rm fem}$}] {};
\plot table[x expr = \thisrow{n},y expr=100*\thisrow{estimated}/85.3087]%
{figures/scattering/data/5/errors_p3.dat}
node[pos=.05,pin=-90:{$E_{\rm est}$}] {};
\end{axis}
\end{tikzpicture}
\end{minipage}
\caption{Behaviors of the estimated and analytical errors in the adaptive procedure for
the scattering problem of Section~\ref{sec_NTO} with $k = 10\pi$}
\label{figure_scattering_curve_5}
\end{figure}

\begin{figure}
\begin{centering}
% \input{figures/data/error_level_1.tex}
% 
% \input{figures/data/error_level_2.tex}
% 
% \hspace{.275cm}\input{figures/data/error_level_3.tex}

\includegraphics{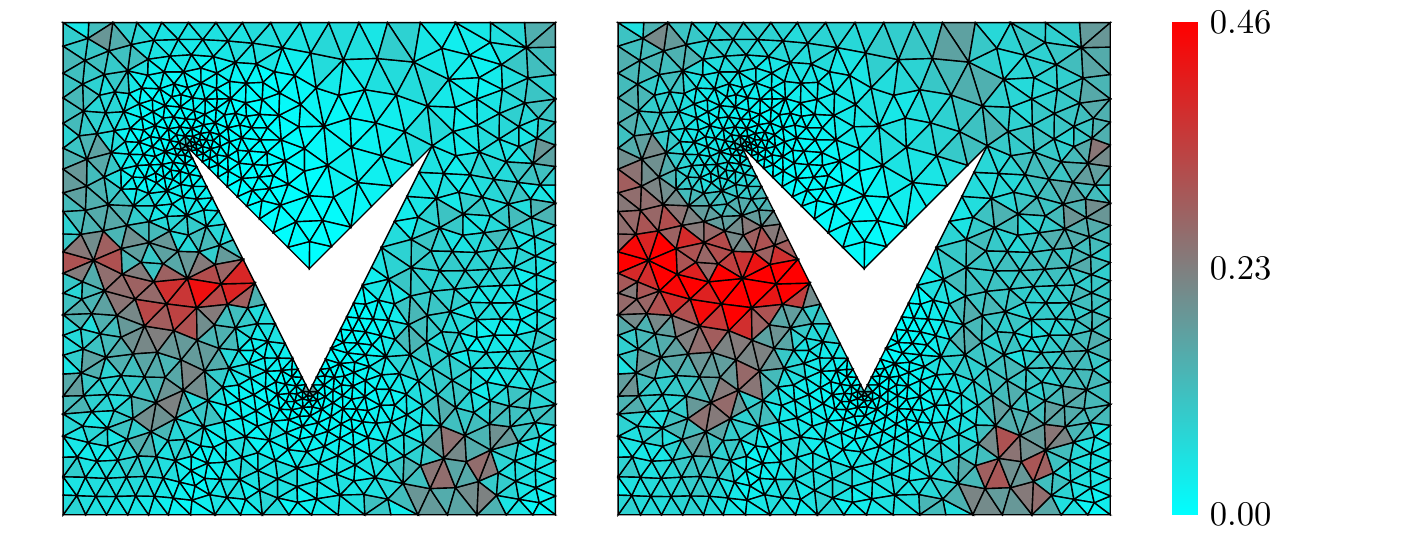}

\includegraphics{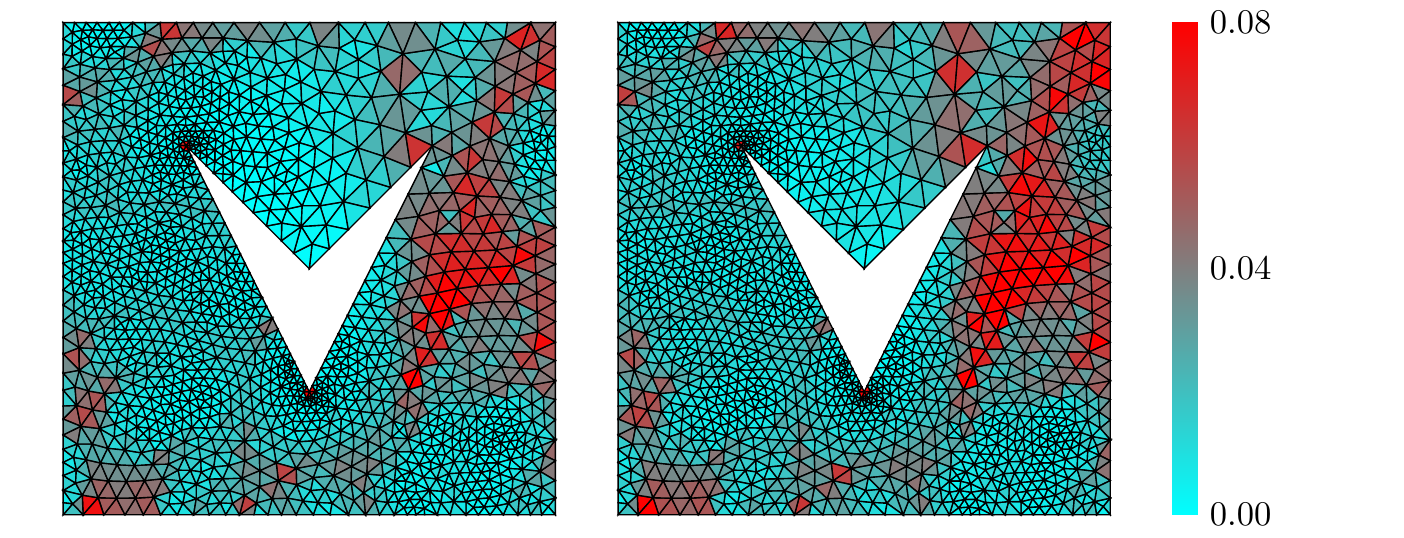}

\includegraphics{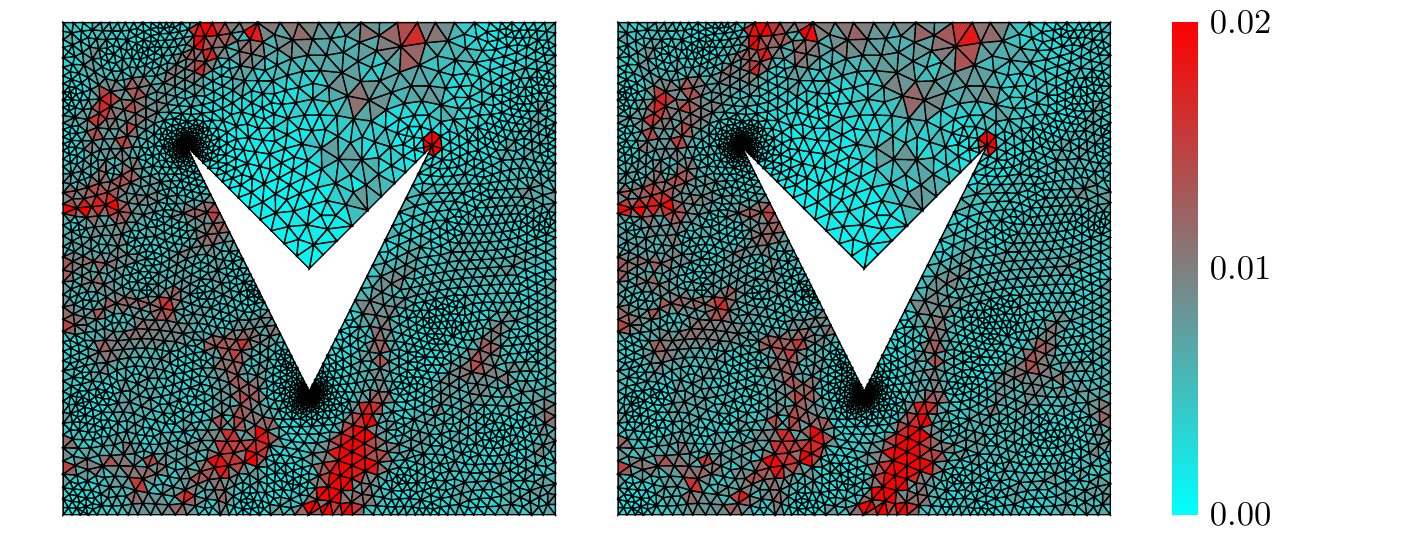}

\end{centering}
\caption{Estimators $\eta_K$ (left) and elementwise errors $\norm{\widetilde u_h-u_h}_{1,k,K}$ (right)
for the scattering problem of Section~\ref{sec_NTO} for $k = 10\pi$ with $\mathcal P_3$ elements}
\label{figure_scattering_error}
\end{figure}

\begin{figure}
\begin{minipage}{.40\linewidth}
\includegraphics[width=6.0cm]{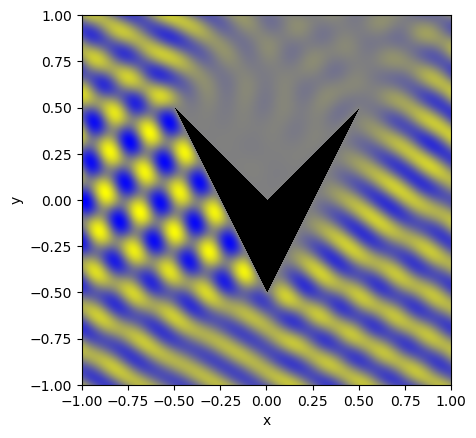}
\end{minipage}
\begin{minipage}{.40\linewidth}
\includegraphics[width=6.0cm]{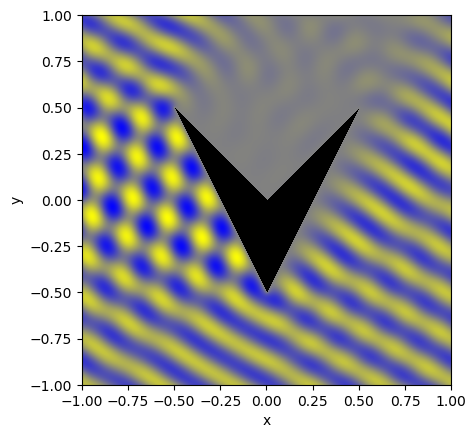}
\end{minipage}
\begin{minipage}{.15\linewidth}
\input{figures/scattering/colorbar.tex}
\vspace{.5cm}
\end{minipage}
\caption{Scattering problem with $\nu = \pi/3$: real (left) and imaginary (right) parts of the solution for $k = 10\pi$}
\label{figure_scattering_solution}
\end{figure}

% \subsection{Adaptivity in the preasymptotic range}
% \label{sec_ada_pre}

\begin{figure}
\center
\begin{minipage}{.40\linewidth}
\hspace{1cm}
\begin{tikzpicture}[scale=2.3]
\definecolor{color_0}{rgb}{1,1,1};
\draw[fill=color_0] (0.5,0.5) -- (0.325025,0.756147) -- (-0.0722941,0.581106) -- cycle;
\definecolor{color_1}{rgb}{1,1,1};
\draw[fill=color_1] (0.5,0.5) -- (0.233898,-0.0322042) -- (0.68575,0.0125127) -- cycle;
\definecolor{color_2}{rgb}{1,1,1};
\draw[fill=color_2] (0.531014,1) -- (0.325025,0.756147) -- (0.724359,0.604281) -- cycle;
\definecolor{color_3}{rgb}{1,1,1};
\draw[fill=color_3] (0.325025,0.756147) -- (0.0866299,1) -- (-0.0722941,0.581106) -- cycle;
\definecolor{color_4}{rgb}{1,1,1};
\draw[fill=color_4] (0.233898,-0.0322042) -- (0,-0.5) -- (0.5364,-0.515509) -- cycle;
\definecolor{color_5}{rgb}{1,1,1};
\draw[fill=color_5] (0.68575,0.0125127) -- (1,-0.484063) -- (1,-0.0152579) -- cycle;
\definecolor{color_6}{rgb}{1,1,1};
\draw[fill=color_6] (1,-0.484063) -- (0.68575,0.0125127) -- (0.5364,-0.515509) -- cycle;
\definecolor{color_7}{rgb}{1,1,1};
\draw[fill=color_7] (0.724359,0.604281) -- (0.5,0.5) -- (0.68575,0.0125127) -- cycle;
\definecolor{color_8}{rgb}{1,1,1};
\draw[fill=color_8] (-1,-0.0855462) -- (-1,-0.506841) -- (-0.593387,-0.454253) -- cycle;
\definecolor{color_9}{rgb}{1,1,1};
\draw[fill=color_9] (0.68575,0.0125127) -- (1,0.446564) -- (0.724359,0.604281) -- cycle;
\definecolor{color_10}{rgb}{1,1,1};
\draw[fill=color_10] (1,1) -- (0.531014,1) -- (0.724359,0.604281) -- cycle;
\definecolor{color_11}{rgb}{1,1,1};
\draw[fill=color_11] (-1,1) -- (-1,0.437706) -- (-0.64925,0.655008) -- cycle;
\definecolor{color_12}{rgb}{1,1,1};
\draw[fill=color_12] (1,-1) -- (1,-0.484063) -- (0.5364,-0.515509) -- cycle;
\definecolor{color_13}{rgb}{1,1,1};
\draw[fill=color_13] (-1,-1) -- (-0.472667,-1) -- (-0.593387,-0.454253) -- cycle;
\definecolor{color_14}{rgb}{1,1,1};
\draw[fill=color_14] (1,0.446564) -- (1,1) -- (0.724359,0.604281) -- cycle;
\definecolor{color_15}{rgb}{1,1,1};
\draw[fill=color_15] (1,-0.0152579) -- (1,0.446564) -- (0.68575,0.0125127) -- cycle;
\definecolor{color_16}{rgb}{1,1,1};
\draw[fill=color_16] (0,0) -- (0.5,0.5) -- (-0.0722941,0.581106) -- cycle;
\definecolor{color_17}{rgb}{1,1,1};
\draw[fill=color_17] (0.195666,-0.748463) -- (0.0487794,-1) -- (0.493543,-1) -- cycle;
\definecolor{color_18}{rgb}{1,1,1};
\draw[fill=color_18] (-0.0722941,0.581106) -- (-0.444756,1) -- (-0.64925,0.655008) -- cycle;
\definecolor{color_19}{rgb}{1,1,1};
\draw[fill=color_19] (0.195666,-0.748463) -- (0.5364,-0.515509) -- (0,-0.5) -- cycle;
\definecolor{color_20}{rgb}{1,1,1};
\draw[fill=color_20] (0.5364,-0.515509) -- (0.493543,-1) -- (1,-1) -- cycle;
\definecolor{color_21}{rgb}{1,1,1};
\draw[fill=color_21] (0.0866299,1) -- (-0.444756,1) -- (-0.0722941,0.581106) -- cycle;
\definecolor{color_22}{rgb}{1,1,1};
\draw[fill=color_22] (-0.678501,0.123646) -- (-1,-0.0855462) -- (-0.593387,-0.454253) -- cycle;
\definecolor{color_23}{rgb}{1,1,1};
\draw[fill=color_23] (-0.290413,0.0808256) -- (-0.5,0.5) -- (-0.678501,0.123646) -- cycle;
\definecolor{color_24}{rgb}{1,1,1};
\draw[fill=color_24] (-0.472667,-1) -- (-0.218115,-0.684186) -- (-0.593387,-0.454253) -- cycle;
\definecolor{color_25}{rgb}{1,1,1};
\draw[fill=color_25] (-1,0.437706) -- (-0.678501,0.123646) -- (-0.64925,0.655008) -- cycle;
\definecolor{color_26}{rgb}{1,1,1};
\draw[fill=color_26] (-0.472667,-1) -- (0.0487794,-1) -- (-0.218115,-0.684186) -- cycle;
\definecolor{color_27}{rgb}{1,1,1};
\draw[fill=color_27] (0.195666,-0.748463) -- (0.493543,-1) -- (0.5364,-0.515509) -- cycle;
\definecolor{color_28}{rgb}{1,1,1};
\draw[fill=color_28] (-0.444756,1) -- (-1,1) -- (-0.64925,0.655008) -- cycle;
\definecolor{color_29}{rgb}{1,1,1};
\draw[fill=color_29] (-1,0.437706) -- (-1,-0.0855462) -- (-0.678501,0.123646) -- cycle;
\definecolor{color_30}{rgb}{1,1,1};
\draw[fill=color_30] (-0.5,0.5) -- (-0.253178,0.253178) -- (-0.0722941,0.581106) -- cycle;
\definecolor{color_31}{rgb}{1,1,1};
\draw[fill=color_31] (-0.678501,0.123646) -- (-0.593387,-0.454253) -- (-0.290413,0.0808256) -- cycle;
\definecolor{color_32}{rgb}{1,1,1};
\draw[fill=color_32] (0.5364,-0.515509) -- (0.68575,0.0125127) -- (0.233898,-0.0322042) -- cycle;
\definecolor{color_33}{rgb}{1,1,1};
\draw[fill=color_33] (-0.64925,0.655008) -- (-0.5,0.5) -- (-0.0722941,0.581106) -- cycle;
\definecolor{color_34}{rgb}{1,1,1};
\draw[fill=color_34] (-0.253178,0.253178) -- (0,0) -- (-0.0722941,0.581106) -- cycle;
\definecolor{color_35}{rgb}{1,1,1};
\draw[fill=color_35] (0,-0.5) -- (-0.109536,-0.280928) -- (-0.218115,-0.684186) -- cycle;
\definecolor{color_36}{rgb}{1,1,1};
\draw[fill=color_36] (-0.218115,-0.684186) -- (-0.109536,-0.280928) -- (-0.593387,-0.454253) -- cycle;
\definecolor{color_37}{rgb}{1,1,1};
\draw[fill=color_37] (0.325025,0.756147) -- (0.5,0.5) -- (0.724359,0.604281) -- cycle;
\definecolor{color_38}{rgb}{1,1,1};
\draw[fill=color_38] (-0.218115,-0.684186) -- (0.0487794,-1) -- (0.195666,-0.748463) -- cycle;
\definecolor{color_39}{rgb}{1,1,1};
\draw[fill=color_39] (0.531014,1) -- (0.0866299,1) -- (0.325025,0.756147) -- cycle;
\definecolor{color_40}{rgb}{1,1,1};
\draw[fill=color_40] (-0.5,0.5) -- (-0.64925,0.655008) -- (-0.678501,0.123646) -- cycle;
\definecolor{color_41}{rgb}{1,1,1};
\draw[fill=color_41] (-0.109536,-0.280928) -- (-0.290413,0.0808256) -- (-0.593387,-0.454253) -- cycle;
\definecolor{color_42}{rgb}{1,1,1};
\draw[fill=color_42] (-1,-0.506841) -- (-1,-1) -- (-0.593387,-0.454253) -- cycle;
\definecolor{color_43}{rgb}{1,1,1};
\draw[fill=color_43] (-0.218115,-0.684186) -- (0.195666,-0.748463) -- (0,-0.5) -- cycle;
\end{tikzpicture}
\vspace{.5cm}
\end{minipage}
\begin{minipage}{.40\linewidth}
\includegraphics[width=6cm]{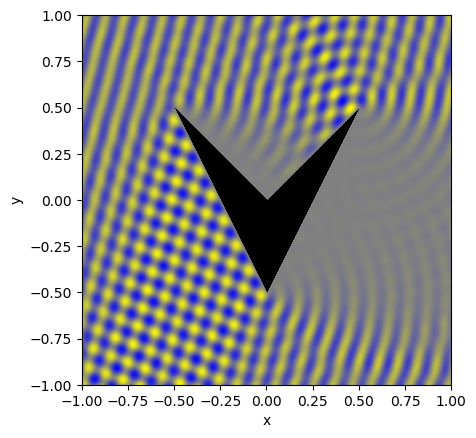}
\end{minipage}
\begin{minipage}{.15\linewidth}
$\;$
\end{minipage}

\begin{minipage}{.40\linewidth}
\includegraphics[width=6cm]{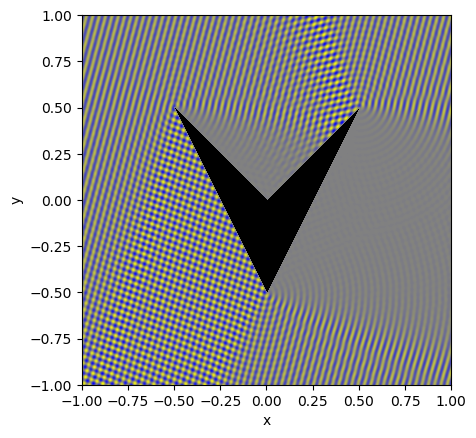}
\end{minipage}
\begin{minipage}{.40\linewidth}
\includegraphics[width=6cm]{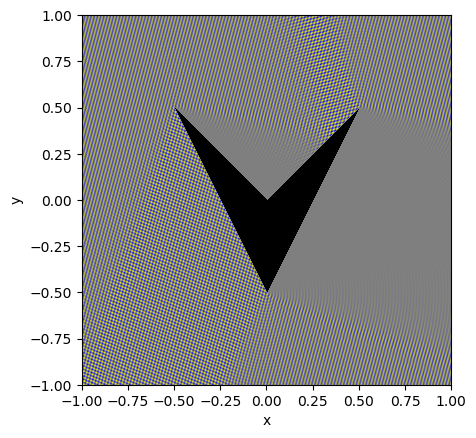}
\end{minipage}
\begin{minipage}{.15\linewidth}
\input{figures/scattering/colorbar2}
\vspace{.5cm}
\end{minipage}
\caption{Initial mesh for the adaptivity experiment and real part of reference solution for
different wavenumbers. Top right panel: $k = 20\pi$. Bottom left panel $k = 60\pi$. Bottom
right panel $k=120\pi$}
\label{figure_init_mesh_adaptivity}
\end{figure}

{\color{black}
\begin{figure}
\begin{minipage}{.45\linewidth}
\begin{tikzpicture}
\begin{axis}
[
	xlabel=Iterations,
	ylabel=Relative error (\%),
	ymode=log,
	ymin=1,
	ymax=500,
	width=\linewidth
]

\plot table[x = iter, y expr =100*\thisrow{true_err}/173]%
{figures/data/adaptivity/P1/curve_error.txt}
node[pos=.8,pin={[pin distance=1cm]90:{$E_{\rm fem}$}}] {};

\plot table[x = iter, y expr =100*\thisrow{est_err}/173]%
{figures/data/adaptivity/P1/curve_error.txt}
node[pos=.2,pin={[pin distance=.5cm]-90:{$E_{\rm est}$}}] {};

\addplot[dashed] coordinates {(10,1) (10,500)};

\end{axis}
\end{tikzpicture}
\end{minipage}
\begin{minipage}{.45\linewidth}
\begin{tikzpicture}
\begin{axis}
[
	xlabel=Iterations ($j$),
	ylabel=Mesh sizes,
	ymode=log,
	ymin=1.e-5,
	ymax=2,
	width=\linewidth
]

\plot table[x = iter, y = hmin]%
{figures/data/adaptivity/P1/curve.txt}
node[pos=.2,pin={[pin distance=.5cm]-90:{$h_{\rm min}$}}] {};

\plot table[x = iter, y = hmax]%
{figures/data/adaptivity/P1/curve.txt}
node[pos=.7,pin={[pin distance=.5cm]90:{$h_{\rm max}$}}] {};

\plot[domain=10:17] {1.5*exp(-ln(2)*x)};

\SlopeTriangle{.65}{-.1}{.15}{-.75}{$2^{-j}$}{}

\addplot[dashed] coordinates {(10,5.e-6) (10,2)};

\end{axis}
\end{tikzpicture}
\end{minipage}

\caption{Behaviors of the estimated and approximate analytical errors and mesh sizes in the adaptive procedure for the scattering problem of Section~\ref{sec_NTO}, $\mathcal P_1$ elements and $k=20\pi$}
\label{figure_adaptivity_p1}
\end{figure}
}

{\color{black}
\begin{figure}
\begin{minipage}{.45\linewidth}
\begin{tikzpicture}
\begin{axis}
[
	xlabel=Iterations,
	ylabel=Relative error (\%),
	ymode=log,
	ymin=5.e-1,
	ymax=1.e3,
	width=\linewidth
]

\plot table[x = iter, y expr =100*\thisrow{true_err}/173]%
{figures/data/adaptivity/P2/curve_error.txt}
node[pos=.85,pin={[pin distance=1cm]90:{$E_{\rm fem}$}}] {};

\plot table[x = iter, y expr =100*\thisrow{est_err}/173]%
{figures/data/adaptivity/P2/curve_error.txt}
node[pos=.2,pin={[pin distance=.5cm]-90:{$E_{\rm est}$}}] {};

\addplot[dashed] coordinates {(12,5.e-1) (12,1.e3)};

\end{axis}
\end{tikzpicture}
\end{minipage}
\begin{minipage}{.45\linewidth}
\begin{tikzpicture}
\begin{axis}
[
	xlabel=Iterations ($j$),
	ylabel=Mesh sizes,
	ymode=log,
	ymin=1.e-5,
	ymax=2,
	width=\linewidth
]

\plot table[x = iter, y = hmin]%
{figures/data/adaptivity/P2/curve.txt}
node[pos=.2,pin={[pin distance=.5cm]-90:{$h_{\rm min}$}}] {};
\plot table[x = iter, y = hmax]%
{figures/data/adaptivity/P2/curve.txt}
node[pos=.8,pin={[pin distance=.5cm]90:{$h_{\rm max}$}}] {};

\plot[domain=12:17] {3*exp(-ln(2)*x)};

\SlopeTriangle{.70}{-.1}{.15}{-.75}{$2^{-j}$}{}

\addplot[dashed] coordinates {(12,1.e-5) (12,2)};

\end{axis}
\end{tikzpicture}
\end{minipage}

\caption{Behaviors of the estimated and approximate analytical errors and mesh sizes in the adaptive procedure for the scattering problem of Section~\ref{sec_NTO}, $\mathcal P_2$ elements and $k=60\pi$}
\label{figure_adaptivity_p2}
\end{figure}
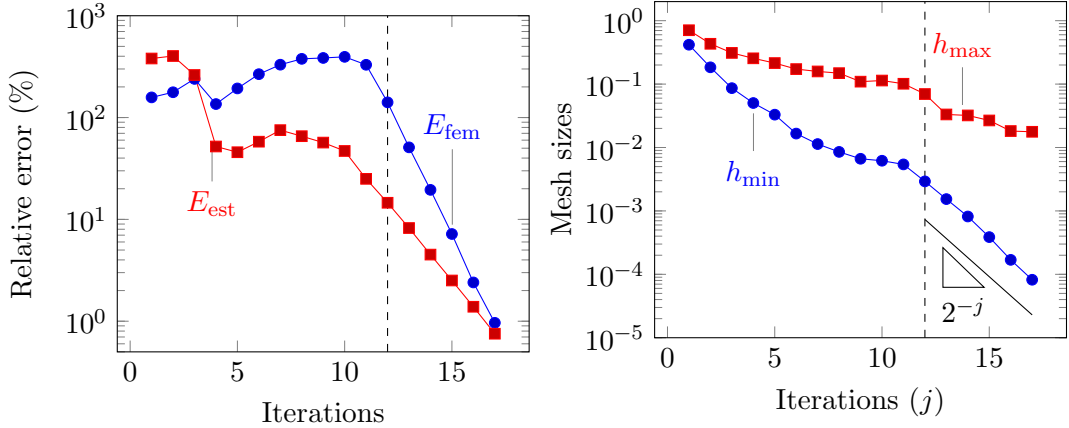
}

{\color{black}
\begin{figure}
\begin{minipage}{.45\linewidth}
\begin{tikzpicture}
\begin{axis}
[
	xlabel=Iterations,
	ylabel=Relative error (\%),
	ymode=log,
	ymin=1.e-2,
	ymax=2.e3,
	width=\linewidth
]

\plot table[x = iter, y expr =100*\thisrow{true_err}/173]%
{figures/data/adaptivity/P4/curve_error.txt}
node[pos=.8,pin={[pin distance=1cm]90:{$E_{\rm fem}$}}] {};

\plot table[x = iter, y expr =100*\thisrow{est_err}/173]%
{figures/data/adaptivity/P4/curve_error.txt}
node[pos=.2,pin={[pin distance=.5cm]-90:{$E_{\rm est}$}}] {};

\addplot[dashed] coordinates {(12,1.e-2) (12,2.e3)};

\end{axis}
\end{tikzpicture}
\end{minipage}
\begin{minipage}{.45\linewidth}
\begin{tikzpicture}
\begin{axis}
[
	xlabel=Iterations ($j$),
	ylabel=Mesh sizes,
	ymode=log,
	ymin=1.e-5,
	ymax=2,
	width=\linewidth
]

\plot table[x = iter, y = hmin]%
{figures/data/adaptivity/P4/curve.txt}
node[pos=.2,pin={[pin distance=.5cm]-90:{$h_{\rm min}$}}] {};

\plot table[x = iter, y = hmax]%
{figures/data/adaptivity/P4/curve.txt}
node[pos=.7,pin={[pin distance=.5cm]90:{$h_{\rm max}$}}] {};

\plot[domain=12:17] {4*exp(-ln(2)*x)};

\SlopeTriangle{.70}{-.1}{.15}{-.75}{$2^{-j}$}{}

\addplot[dashed] coordinates {(12,1.e-5) (12,2)};

\end{axis}
\end{tikzpicture}
\end{minipage}

\caption{Behaviors of the estimated and approximate analytical errors and mesh sizes in the adaptive procedure for the scattering problem of Section~\ref{sec_NTO}, $\mathcal P_4$ elements and $k=120\pi$}
\label{figure_adaptivity_p4}
\end{figure}
}

Let us finally investigate the ability of the local estimators
$\eta_K$ of~\eqref{eq_definition_eta} to drive a mesh adaptive algorithm starting with a very coarse mesh.
To this end, we still consider the scattering of an incident wave by a non-trapping
obstacle~\eqref{eq_helmholtz_scattering}, but we change the incident angle to $\nu = -\pi/12$ and work with higher wavenumbers $k$.
We keep the same refinement algorithm as above but we start with
the (much) coarser mesh depicted in Figure \ref{figure_init_mesh_adaptivity}.
Figures~\ref{figure_adaptivity_p1}, \ref{figure_adaptivity_p2}, and~\ref{figure_adaptivity_p4}
present the results respectively obtained with $\mathcal P_1$ elements and
$k=20\pi$, $\mathcal P_2$ elements and $k=60\pi$, and $\mathcal P_4$ elements
and $k=120\pi$. The reference solutions, computed on a fine mesh with $\mathcal P_6$
elements, are represented in Figure~\ref{figure_init_mesh_adaptivity}.
The left panels of Figures~\ref{figure_adaptivity_p1}--\ref{figure_adaptivity_p4} show that, in all cases, the algorithm converges towards the
correct solution, even though the initial mesh is very coarse
and features less than one degree of freedom per wavelength.
In the right panels of Figures~\ref{figure_adaptivity_p1}--\ref{figure_adaptivity_p4},
we indicate the minimal and maximal element sizes in the mesh at each iteration
(recall that the meshing package divides the mesh size by $2$ in the zone selected for refinement.
In each figure, we indicate by a dashed vertical line the first iteration for which the resolved regime is entered in that there holds $\frac{kh}{2\pi p} \leq 1$.
In the resolved regime where $\frac{kh}{2\pi p} \leq 1$, we can see that the
minimal element sizes are divided by two at each iteration, which means
that the smallest elements in the mesh are always selected for
refinement. This is typical of refinements close to re-entrant
corners, and is expected to correctly capture the corner singularities.
In the unresolved regime, however, the minimal element sizes decrease
more slowly and are closer to the maximal element sizes, indicating a more uniform
refinement of the mesh. This is expected since a global refinement of the mesh is
required in the unresolved regime, before the local behavior of the solution can
be efficiently captured.
Congruently, in the first iterations of the algorithm (approximately~10),
the error $E_{\rm fem}$ stagnates or even slightly increases, while in the remaining iterations,
the error steeply decreases at each step.
An interesting observation is that this seems to appear here soon after the beginning
of the resolved regime where $\frac{kh}{2\pi p} \leq 1$, whereas a similar steep decrease
only appeared later for uniformly refined meshes, see Figures~\ref{figure_curve_p1_fullnorm}
and~\ref{figure_curve_p2_fullnorm}. This earlier decrease of the finite element error
may be explained by the fact that the resolved regime is defined only by the size of the largest
element in the mesh, so that a large part of the mesh can be refined before entering the
resolved regime.

\section{Conclusions}
\label{sec:conclusion}

We have proposed a novel a posteriori error estimator for the Helmholtz problem
with mixed boundary conditions in two and three space dimensions. The estimator is based on equilibrated flux reconstruction that relies on the solution of patchwise mixed finite element problems.
It is reliable, where the reliability constant depends on the approximation factors $\Cba$ and
(possibly) $\tCba$ and tends to one when $\Cba,\tCba \to 0$, so that the estimator becomes asymptotically unknown-constant-free.
We have also proven, via arguments based on elliptic regularity shift, that
the conditions $\Cba, \tCba \to 0$ are met, for most situations of practical interest,
when $\frac{h}{p}\to0$ with fixed $k$. Finally, we have proven that the derived estimator is
locally efficient and polynomial-degree-robust in all regimes and wavenumber-robust in the
asymptotic regime $\Cba \leq 1$.

The approximation factor(s) $\Cba$ (and $\tCba$) are unfortunately in general not computable.
We have managed to provide computable upper bounds on $\Cba$ in particular
settings of interest, including scattering by non-trapping obstacles and wave
propagation in free space. For such configurations, our upper bound thus becomes
guaranteed and fully computable with no unknown constant and no assumptions on the mesh size,
the polynomial degree, or the wavenumber. Unfortunately, our computable bounds on $\Cba$ are in
general too rough, not converging to zero in some cases, or only converging to $0$ with mesh
refinement $h \to 0$ but not with polynomial degree increase $p \to \infty$, in contrast to
the property $\Cba \to 0$ when $\frac{h}{p}\to0$ at fixed $k$. Consequently, an important
overestimation can appear for these guaranteed bounds, though they are still locally efficient
and polynomial-degree-robust. We believe that these issues could be addressed in future work
following~\cite{chaumontfrelet_nicaise_2018a}, in particular by carefully estimating the
multiplicative coefficient of corner singularities.

The presented numerical experiments illustrate our findings, and suggest that the
proposed estimators additionally have the potential to be asymptotically exact, and
indicate that they can drive adaptive mesh refinement starting from coarse meshes and small
polynomial degrees even when starting the adaptive process in the unresolved regime where
$\frac{kh}{2 \pi p} > 1$.

\appendix

\section{Estimate on the best-approximation constant for boundary data}
\label{sec_tcba}

In this appendix, we analyze the behavior of the quantity $\tCba$ defined
in~\eqref{eq_tCba} in terms of $\Cba$ defined in~\eqref{eq_Cba_bis}.
For the sake of simplicity, in this section, the notation
$C(\widehat \Omega,\widehat \Gamma_{\rm D})$ denotes a generic constant
that only depends on the geometry of $\Omega$ and $\GD$ but may vary from one occurrence to the other.
In addition $\Cqi(\kappa)$ is a ``quasi-interpolation'' constant that only
depends on the mesh shape-regularity parameter $\kappa$
\cite{hiptmair_pechstein_2017a,melenk_2005a}.

The results derived in this appendix rely on the following regularity assumption.

\begin{assu}[Additional regularity]
\label{assumption_shift_boundary}
Let $\phi \in L^2(\Omega)$. We assume that if $u \in H^1(\Omega)$
satisfies
\begin{equation*}
\left \{
\begin{array}{rcll}
-\Delta u &=& \phi & \text{ in } \Omega,
\\
u &=& 0 & \text{ on } \GD,
\\
\grad u {\cdot} \nn &=& 0 & \text{ on } \GA,
\end{array}
\right .
\end{equation*}
then $u \in H^2(\widetilde \Omega)$ with
\begin{equation*}
|u|_{2,\widetilde \Omega} \leq C(\widehat \Omega,\widehat \Gamma_{\rm D}) \|\phi\|_{0,\Omega},
\end{equation*}
where $\widetilde \Omega \subset \Omega$ is a neighborhood of $\GA$ (\ie, $\widetilde
\Omega$ is an open subset of $\Omega$ and $\GA$ is a subset of the closure of $\widetilde\Omega$)
and the constant
$C(\widehat \Omega,\widehat \Gamma_{\rm D})$ depends on the shape of $\Omega$
and the splitting of its boundary into $\GD$ and $\GA$ but not on its diameter $h_\Omega$.
Furthermore, we assume that if $\psi \in L^2(\GA)$ and $u \in H^1(\Omega)$ solves
\begin{equation*}
\left \{
\begin{array}{rcll}
-\Delta u &=& 0 & \text{ in } \Omega,
\\
u &=& 0 & \text{ on } \GD,
\\
\grad u {\cdot} \nn &=& \psi & \text{ on } \GA,
\end{array}
\right .
\end{equation*}
then $u \in H^{\frac32}(\widetilde \Omega)$ with
\begin{equation*}
|u|_{\frac32,\widetilde \Omega} \leq C(\widehat \Omega,\widehat \Gamma_{\rm D}) \|\psi\|_{0,\GA}.
\end{equation*}
\end{assu}

Assumption~\ref{assumption_shift_boundary} is not an important restriction.
Indeed, it is typically satisfied in applications, as the boundary $\GA$ is artificially designed to
enclose the region of interest. For instance, $\GA$ is usually selected
as the boundary of a convex polytope for scattering problems, so that
Assumption~\ref{assumption_shift_boundary} holds
(see~\cite{chaumontfrelet_nicaise_tomezyk_2018a} and \cite[Lemma 1]{costabel_1990a}).
In the case of cavity problems, $\GA$ is typically planar, and Assumption
\ref{assumption_shift_boundary} holds if the solid angle between $\GA$
and $\GD$ is less than or equal to $\pi/2$ (we can perform an odd reflection across
the Dirichlet boundary, and recover a situation similar to the scattering problem,
see~\cite{chaumontfrelet_nicaise_tomezyk_2018a}).
As a result, Assumption~\ref{assumption_shift_boundary}
is satisfied in all the configurations depicted in Figure~\ref{figure_domains}.

Our next step is to employ a lifting operator $\lift$ introduced in \cite{melenk_sauter_2011a}
that transforms the boundary right-hand side on $\GA$, say $\psi$, appearing
in the definition~\eqref{eq_tCba} of $\tCba$ into a volume right-hand side $\lift_\psi$.
We remark that there exists a function $\chi \in C^\infty(\overline{\Omega})$ such that
$0 \leq \chi \leq 1$ in $\overline{\Omega}$, $\chi = 0$ outside $\widetilde \Omega$,
and $\chi = 1$ in a neighborhood on $\GA$.
In addition, a simple scaling argument shows that we can choose $\chi$ such that
\begin{equation*}
|\chi|_{j,\infty,\Omega} \leq C(\widehat \Omega,\widehat \Gamma_{\rm D}) h_\Omega^{-j}
\end{equation*}
for all $j \in \mathbb N$. The main novelty of the following result resides when both subsets
$\GD$ and $\GA$ have positive measure and touch each other, so that only a regularity shift to
$H^{\frac32}$ is available owing to Assumption~\ref{assumption_shift_boundary}.

\begin{lemm}[Boundary lifting operator]
\label{lemma_lifting}
Let Assumption~\ref{assumption_shift_boundary} hold.
For all $\psi \in L^2(\GA)$, we define $\lift_\psi$ as the unique element of $H^1_\GD(\Omega)$
such that
\begin{equation}
\label{eq_definition_lifting}
a(w,\lift_\psi) = (w,\psi)_\GA
\end{equation}
for all $w \in H^1_\GD(\Omega)$, where
\begin{equation*}
a(w,\lift_\psi) \eq k^2(w,\lift_\psi) -ik(w,\lift_\psi)_\GA + (\grad w,\grad \lift_\psi).
\end{equation*}
Then we have
\begin{align}
\label{eq_estimate_lifting_GA}
k\|\lift_\psi\|_{0,\GA}
&\leq
\|\psi\|_{0,\GA},
\\
\label{eq_estimate_lifting_L2}
k^2 \|\lift_\psi\|_{0,\Omega}^2 + |\lift_\psi|_{1,\Omega}^2
&\leq
\frac{1}{k} \|\psi\|_{0,\GA}^2.
\end{align}
In addition, we have
\begin{equation}
\label{eq_estimate_lifting_reg}
k^{\frac12} \inf_{v_h \in V_h}
|\chi \lift_\psi - v_h|_{1,\Omega}
\leq
C(\widehat \Omega,\widehat \Gamma_{\rm D}) \Cqi(\kappa) \left (
\left (\frac{kh}{p}\right )^{\frac12}
+
\frac{kh}{p}
\right )\|\psi\|_{0,\GA},
\end{equation}
where the last term is present only if $\GD$ has positive surface measure.
\end{lemm}

\begin{proof}
We first pick $w = \lift_\psi$ as a test function in~\eqref{eq_definition_lifting}.
Taking the imaginary part yields
\begin{equation*}
k\|\lift_\psi\|_{0,\GA}^2 = -\Im (\lift_\psi,\psi)_\GA \leq \|\psi\|_{0,\GA}\|\lift_\psi\|_{0,\GA},
\end{equation*}
and~\eqref{eq_estimate_lifting_GA} follows. Now, we take the real
part and use the above bound to obtain
\begin{equation*}
k^2 \|\lift_\psi\|_{0,\Omega}^2 + |\lift_\psi|_{1,\Omega}^2 = \Re (\lift_\psi,\psi)_{\GA}
\leq
\frac{1}{k} \|\psi\|_{0,\GA}^2.
\end{equation*}
This yields~\eqref{eq_estimate_lifting_L2}. Finally, we observe
that we can see $\lift_\psi$ as the unique element of $H^1_\GD(\Omega)$
such that
\begin{equation*}
(\grad w,\grad \lift_\psi) = -k^2 (w,\lift_\psi) + (w,\psi-ik\lift_\psi)_{\GA}
\end{equation*}
for all $w \in H^1_\GD(\Omega)$.

At this point, we distinguish the case where $\GD$ is of zero measure or not.
If $|\GD| = 0$, then Assumption \ref{assumption_shift_boundary} together with
classical arguments (see~\cite{grisvard_1992a}) away from $\GA$ show that
\begin{eqnarray*}
\|\lift_\psi\|_{\frac32,\Omega}
&\leq&
C(\widehat \Omega,\widehat \Gamma_{\rm D})
\left (
k^2 \|\lift_\psi\|_{\frac12,\Omega} +
\|\psi\|_{0,\GA} + k\|\lift_\psi\|_{0,\GA}
\right )
\\
&\leq&
C(\widehat \Omega,\widehat \Gamma_{\rm D})
\left (
k^2 \sqrt{\|\lift_\psi\|_{0,\Omega}\|\lift_\psi\|_{1,\Omega}} +
\|\psi\|_{0,\GA} + k\|\lift_\psi\|_{0,\GA}
\right ),
\end{eqnarray*}
and we conclude with~\eqref{eq_estimate_lifting_GA} and~\eqref{eq_estimate_lifting_L2} that
\begin{equation*}
\|\chi \lift_\psi\|_{\frac32,\Omega}
\leq
C(\widehat \Omega,\widehat \Gamma_{\rm D})
\|\lift_\psi\|_{\frac32,\Omega}
\leq
C(\widehat \Omega,\widehat \Gamma_{\rm D})
\|\psi\|_{0,\GA}.
\end{equation*}
Then, \eqref{eq_estimate_lifting_reg} follows from standard
approximation theory~\cite{hiptmair_pechstein_2017a,melenk_2005a}.

On the other hand, when $|\GD| \neq 0$, we split $\lift_\psi = \phi_2 + \phi_{\frac32}$ where
$\phi_2, \phi_{\frac32} \in H^1_\GD(\Omega)$ are uniquely defined by
\begin{equation*}
(\grad w,\grad \phi_2    ) = -k^2 (w,\lift_\psi), \quad
(\grad w,\grad \phi_{\frac32}) = (w,\psi-ik\lift_\psi)_{\GA}
\end{equation*}
for all $w \in H^1_\GD(\Omega)$. Picking the test function
$w = \phi_2$, it follows that
\begin{equation*}
|\phi_2|_{1,\Omega}^2
\leq
k^2 \|\lift_\psi\|_{0,\Omega}\|\phi_2\|_{0,\Omega}
\leq
C(\widehat \Omega,\widehat \Gamma_{\rm D})
k^{\frac12} h_\Omega \|\psi\|_{0,\GA}|\phi_2|_{1,\Omega},
\end{equation*}
where we  used the Poincar\'e inequality to handle $\|\phi_2\|_{0,\Omega}$
and~\eqref{eq_estimate_lifting_GA} to estimate $k\|\lift_\psi\|_{0,\Omega}$.
Similarly, employing a multiplicative trace inequality
which combined with the Poincar\'e inequality yields
$\|\phi_{\frac32}\|_{0,\GA}\le C(\widehat \Omega,\widehat \Gamma_{\rm D})
h_\Omega^{\frac12} |\phi_{\frac32}|_{1,\Omega}$, we infer that
\begin{equation*}
|\phi_{\frac32}|_{1,\Omega}^2
\leq
\|\psi - ik\lift_\psi\|_{0,\GA}
\|\phi_{\frac32}\|_{0,\GA} \leq
C(\widehat \Omega,\widehat \Gamma_{\rm D})
h_\Omega^{\frac12} \|\psi\|_{0,\GA}|\phi_{\frac32}|_{1,\Omega}.
\end{equation*}
Then, invoking Assumption
\ref{assumption_shift_boundary}, we have
\begin{equation*}
|\phi_2|_{2,\widetilde \Omega}
\leq
C(\widehat \Omega,\widehat \Gamma_{\rm D})
k^2 \|\lift_\psi\|_{0,\Omega}
\leq
C(\widehat \Omega,\widehat \Gamma_{\rm D})
k^{\frac12} \|\psi\|_{0,\GA},
\end{equation*}
and
\begin{equation*}
|\phi_{\frac32}|_{\frac32,\widetilde \Omega}
\leq
C(\widehat \Omega,\widehat \Gamma_{\rm D})
\|\psi-ik\lift_\psi\|_{0,\GA}
\leq
C(\widehat \Omega,\widehat \Gamma_{\rm D})
\|\psi\|_{0,\GA}.
\end{equation*}
Thus, since $\supp \chi \subset \widetilde \Omega$
and invoking once again the Poincar\'e inequality,
we infer that
\begin{align*}
|\chi \phi_2|_{2,\Omega}
&\leq
\|\chi\|_{0,\Omega}|\phi_2|_{2,\widetilde \Omega} +
|\chi|_{1,\Omega}|\phi_2|_{1,\Omega} +
|\chi|_{2,\Omega}\|\phi_2\|_{0,\Omega}
\\
&\leq
C(\widehat \Omega,\widehat \Gamma_{\rm D})
\left (
|\phi_2|_{2,\widetilde \Omega} +
h_\Omega^{-1} |\phi_2|_{1,\Omega}
\right )
\\
&\leq
C(\widehat \Omega,\widehat \Gamma_{\rm D})
k^{\frac12} \|\psi\|_{0,\GA},
\end{align*}
and similar arguments show that
\begin{equation*}
|\chi \phi_{\frac32}|_{\frac32,\widetilde \Omega}
\leq
C(\widehat \Omega,\widehat \Gamma_{\rm D})
\|\psi\|_{0,\GA}.
\end{equation*}
At this point, \eqref{eq_estimate_lifting_reg} follows from standard
approximation theory~\cite{hiptmair_pechstein_2017a,melenk_2005a}. \qed
\end{proof}

We are now ready to establish the main result of this appendix which proves the
claim~\eqref{eq_bound_s_fine}.

\begin{prop}[Bound on $\tCba$]
%\label{prop_cba}
Let Assumption~\ref{assumption_shift_boundary} hold. Then
\begin{equation*}
%\label{eq_estimate_tcba}
\tCba
\leq
C(\widehat \Omega,\widehat \Gamma_{\rm D})
\Cqi(\kappa) \left (\left (\frac{kh}{p} \right )^{\frac12} + \frac{kh}{p}\right )
+
\left (
C(\widehat \Omega,\widehat \Gamma_{\rm D})
\frac{1}{k h_\Omega} \left (1 + \frac{1}{k h_\Omega}\right ) + 2
\right ) \Cba.
\end{equation*}
\end{prop}

\begin{proof}
We consider an arbitrary element $\psi \in L^2(\GA)$ and define
$u_\psi^\star$ as the unique element of $H^1_\GD(\Omega)$
such that $b(w,u_\psi^\star) = (w,\psi)_\GA$
for all $w \in H^1_\GD(\Omega)$. Then, defining $\lift_\psi \in H^1_\GD(\Omega)$
following Lemma~\ref{lemma_lifting}, we have
\begin{align*}
& b(w,\chi \lift_\psi) \\
&=
-k^2 (w,\chi \lift_\psi) - ik (w,\chi \lift_\psi)_\GA +
(\grad w, \grad (\chi \lift_\psi))
\\
&=
k^2 (\chi w,\lift_\psi) - ik (\chi w,\lift_\psi)_\GA +
(\grad (\chi w),\grad \lift_\psi)
\\
&\quad +
(\grad w, \grad (\chi \lift_\psi)) - (\grad (\chi w),\grad \lift_\psi)
- 2k^2 (\chi w,\lift_\psi)
\\
&=
a(\chi w,\lift_\psi)
+
(\grad w,\lift_\psi \grad \chi) - (w\grad \chi,\grad \lift_\psi)
- 2k^2 (\chi w,\lift_\psi)
\\
&=
(\chi w,\psi)_\GA
+
(\grad w,\lift_\psi \grad \chi) - (w\grad \chi,\grad \lift_\psi)
- 2k^2 (\chi w,\lift_\psi)
\\
&=
(w,\psi)_\GA +
(w,\lift_\psi \grad \chi {\cdot} \nn)_{\partial \Omega} -
(w,\div (\lift_\psi \grad \chi)) - (w,\grad \chi {\cdot} \grad \lift_\psi)
- 2k^2 (w,\chi \lift_\psi)
\\
&=
(w,\psi)_\GA -
(w, \div (\lift_\psi \grad \chi)+\grad \chi {\cdot} \grad \lift_\psi + 2k^2 \chi \lift_\psi),
\end{align*}
where we used the fact that $\grad \chi {\cdot} \nn = 0$ on $\GA$, as $\chi = 1$ in a
neighborhood of $\GA$, and that $w = 0$ on $\GD$ so that
$(w,\lift_\psi \grad \chi \cdot \nn)_{\partial \Omega} = 0$.
It follows that
\begin{equation*}
b(w,u_\psi^\star - \chi \lift_\psi) = (w,\widetilde f),
\end{equation*}
with
$\widetilde f = \div (\lift_\psi \grad \chi)+\grad \chi {\cdot} \grad \lift_\psi+2k^2 \chi \lift_\psi$.
In particular, we have $\widetilde f \in L^2(\Omega)$, and using~\eqref{eq_estimate_lifting_L2},
we infer that
\begin{equation}
\label{tmp_upper_bound_tf}
\|\widetilde f\|_{0,\Omega}
\leq
\left (
C(\widehat \Omega,\widehat \Gamma_{\rm D})
\frac{1}{h_\Omega} \left (1 + \frac{1}{k h_\Omega}\right ) k^{-\frac12}
+ 2k^{\frac12}
\right )
\|\psi\|_{0,\GA}.
\end{equation}
Then, we have
\begin{align*}
k^{\frac12}
\left (\inf_{v_h \in V_h}
|u_\psi^\star - v_h|_{1,\Omega}
\right )
&\leq
k^{\frac12}
\left (
\inf_{y_h \in V_h}
|(u_\psi^\star - \chi \lift_\psi) - y_h|_{1,\Omega}
+
\inf_{w_h \in V_h} |\chi \lift_\psi - w_h|_{1,\Omega}
\right )
\\
&\leq
k^{-\frac12} \Cba \|\widetilde f\|_{0,\Omega}
+
k^{\frac12}
\inf_{w_h \in V_h}
|\chi \lift_\psi - w_h|_{1,\Omega},
\end{align*}
where the first bound stems by taking the function $v_h=y_h+w_h$ where $y_h\in V_h$ and
$w_h\in V_h$ realize the two infimums in the right-hand side.
Using~\eqref{tmp_upper_bound_tf} and Lemma~\ref{lemma_lifting},
we see that
\begin{align*}
k^{\frac12}
\left (\inf_{v_h \in V_h}
|u_\psi^\star - v_h|_{1,\Omega}
\right )
\leq{}&
\left (
C(\widehat \Omega,\widehat \Gamma_{\rm D})
\frac{1}{k h_\Omega} \left (1 + \frac{1}{k h_\Omega} \right ) + 2
\right ) \Cba \|\psi\|_{0,\GA}
\\
& +
C(\widehat \Omega,\widehat \Gamma_{\rm D})
\Cqi(\kappa) \left (
\frac{kh}{p} + \left ( \frac{kh}{p} \right )^{\frac12}
\right )
\|\psi\|_{0,\GA},
\end{align*}
and we conclude by taking the supremum over $\psi \in L^2(\GA)$. \qed
\end{proof}

\begin{coro}[Asymptotic regime]
Under Assumption~\ref{assumption_shift_boundary}, we have
\begin{equation} \label{eq_bound_s_fine}
1 + \tilde\theta_2(\Cba,\tCba) \leq 1 + \widetilde \theta_3 \left (\Cba,\frac{kh}{p}\right ),
\end{equation}
where $\widetilde \theta_3$ is a decreasing function of its two arguments
such that $\lim_{t,t' \to 0} \widetilde \theta_3(t,t') = 0$.
\end{coro}

% \newpage

\bibliographystyle{amsplain}
\bibliography{bibliography}

\end{document}